\documentclass[final, 3p, times]{elsarticle}
\biboptions{sort&compress}
\usepackage{graphicx}
\usepackage{amstext, amsmath, amssymb, amsthm}
\usepackage[shadow]{todonotes}
\usepackage{booktabs}
\usepackage{url}
\usepackage{import}
\usepackage{amsfonts}
\usepackage{amsbsy}
\usepackage{fancyvrb}
\usepackage{stmaryrd}
\usepackage{pdfsync}
\usepackage{graphicx}
\usepackage{algorithm}
\usepackage[shadow]{todonotes}
\usepackage[numbers]{natbib}
\usepackage{enumerate}
\usepackage[utf8]{inputenc}					
\usepackage[T1]{fontenc}					
\usepackage[style=english]{csquotes}	   			
\usepackage{textcomp}						
\usepackage{enumerate}
\usepackage[english]{babel}
\usepackage{amsmath, amssymb, amstext, amsfonts, mathrsfs, mathtools}
\usepackage{verbatim}
\usepackage{xcolor}
\usepackage{graphicx}
\usepackage{epstopdf}
\usepackage[font=footnotesize]{subfig}

\newcommand{\Figref}[1]{Figure \,\ref{#1}}

\newtheorem{theorem}{Theorem}
\newtheorem{proposition}[theorem]{Proposition}
\newtheorem{lemma}[theorem]{Lemma}
\newtheorem{corollary}[theorem]{Corollary}
\newdefinition{remark}[theorem]{Remark}
\renewenvironment{proof}{\noindent \newline \textbf{Proof.}}{\hfill \mbox{\fbox{} } \newline}

\numberwithin{equation}{section}
\numberwithin{figure}{section}
\numberwithin{table}{section}
\numberwithin{theorem}{section}

\usepackage{verbatim} 
\newcommand{\R}{\mathbb{R}}

\newcommand{\RR}{\mathbb{R}}
\newcommand{\NN}{\mathbb{N}}
\newcommand{\foralls}{\forall\,}
\newcommand{\dx}{\,\mathrm{d}x}
\newcommand{\ds}{\,\mathrm{d}s}

\DeclareMathOperator{\diam}{diam}
\DeclareMathOperator{\meas}{meas}

\DeclareMathOperator{\RE}{Re}

\DeclareMathOperator{\CFL}{CFL}

\renewcommand{\div}{\nabla \cdot}
\newcommand{\grad}{\nabla}
\newcommand{\rot}{\nabla \times}

\newcommand{\restr}[2]{ \left. #1 \right|_{#2}}
\newcommand{\tnorm}[1]{{\vert\kern-0.25ex\vert\kern-0.25ex\vert #1 
    \vert\kern-0.25ex\vert\kern-0.25ex\vert}} 
\newcommand{\jump}[1]{\llbracket #1 \rrbracket}
\newcommand{\norm}[1]{\| #1 \|}

\newcommand{\normLtwo}[2]{\| #1 \|_{#2}}

\newcommand{\onehalf}{\frac{1}{2}}

\newcommand{\apriori}{\emph{a~priori}}

\DefineVerbatimEnvironment{code}{Verbatim}{frame=single,rulecolor=\color{blue}}

\newcommand\bsnote[2][]{\todo[inline, caption={2do}, color=green!40 #1]{
\begin{minipage}{\textwidth-4pt}\underline{BS:} #2\end{minipage}}}

\newcommand{\bfepsilon}{{\pmb\epsilon}}

\newcommand{\bfu}{\boldsymbol{u}}
\newcommand{\bff}{\boldsymbol{f}}
\newcommand{\bfg}{\boldsymbol{g}}
\newcommand{\bfh}{\boldsymbol{h}}
\newcommand{\bfv}{\boldsymbol{v}}
\newcommand{\bfc}{\boldsymbol{c}}

\newcommand{\bfz}{\boldsymbol{z}}
\newcommand{\bfn}{\boldsymbol{n}}
\newcommand{\bfx}{\boldsymbol{x}}
\newcommand{\bfzero}{\boldsymbol{0}}
\newcommand{\bfbeta}{\boldsymbol{\beta}}
\newcommand{\bft}{\boldsymbol{t}}

\newcommand{\GammaIn}{\Gamma_{\mathrm{in}}}

\newcommand{\Oast}{\Omega^{\ast}}
\newcommand{\Oasth}{\Omega^{\ast}_h}

\newcommand{\pO}{\Gamma}

\newcommand{\Fast}{\mathcal{F}_{\Gamma}}
\newcommand{\piast}{\pi^{\ast}_h}

\newcommand{\bfpiast}{\boldsymbol{\pi}^{\ast}_h}
\newcommand{\Piast}{\Pi^{\ast}_h}

\newcommand{\mcF}{\mathcal{F}}
\newcommand{\mcT}{\mathcal{T}}
\newcommand{\mcV}{\mathcal{V}}
\newcommand{\mcW}{\mathcal{W}}
\newcommand{\mcX}{\mathcal{X}}
\newcommand{\mcQ}{\mathcal{Q}}

\newcommand{\mcO}{\mathcal{O}}

\newcommand{\mcP}{\mathcal{P}}

\newcommand{\tn}{|\mspace{-1mu}|\mspace{-1mu}|}

\newcommand{\bfw}{\boldsymbol{w}}

\newcommand{\nablan}{\partial_{\bfn}}

\newcommand{\nuhalf}{\nu^\frac{1}{2}}

\newcommand{\sigmahalf}{\sigma^{\frac{1}{2}}}

\newcommand{\GammaGNBC}{\Gamma}

\journal{\phantom{journal}}

\begin{document}

\begin{frontmatter}

\title{A Nitsche cut finite element method for the Oseen problem\\with general Navier boundary conditions}

\author[lnm]
{M.~Winter}
\ead{winter@lnm.mw.tum.de}
\ead[url]{www.lnm.mw.tum.de/staff/magnus-winter}

\author[lnm]
{B.~Schott\corref{cor1}}
\ead{schott@lnm.mw.tum.de}
\ead[url]{www.lnm.mw.tum.de/staff/benedikt-schott}

\author[umit]
{A.~Massing}
\ead{andre.massing@umu.se}
\ead[url]{www.andremassing.com}

\author[lnm]
{W.A.~Wall}
\ead{wall@lnm.mw.tum.de}
\ead[url]{www.lnm.mw.tum.de/staff/wall}

\cortext[cor1]{Corresponding author}

\address[lnm]{Institute for Computational Mechanics, Technical University of Munich, Boltzmannstraße 15, 85747 Garching, Germany}
\address[umit]{Department of Mathematics and Mathematical Statistics, Ume\aa{} University, SE-901 87 Ume\aa{}, Sweden}

\begin{abstract}
In this work a Nitsche-based imposition of generalized Navier conditions on cut meshes for the Oseen problem is presented. 
Other methods from literature dealing with the generalized Navier condition impose this condition by means of substituting the tangential Robin condition in a classical Galerkin way. 
These methods work fine for a large slip length coefficient but lead to conditioning and stability issues when it approaches zero. 
We introduce a novel method for the weak imposition of the generalized Navier condition which remains well-posed and stable for arbitrary choice of slip length, including zero.
The method proposed here builds on the formulation done by \citet{JuntunenStenberg2009a}. 
They impose a Robin condition for the Poisson problem by means of Nitsche's method for an arbitrary combination of the Dirichlet and Neumann parts of the condition.
The analysis conducted for the proposed method is done in a similar fashion as in \citet{MassingSchottWall2016_CMAME_Arxiv_submit}, but 
is done here for a more general type of boundary condition. 
The analysis proves stability for all flow regimes and all choices of slip lengths.
Also an $L^2$-optimal estimate for the velocity error is shown, which was not conducted in the previously mentioned work.
A numerical example is carried out for varying slip lengths to verify the robustness and stability of the method with respect to the choice of slip length.
Even though proofs and formulations are presented for the more general case of
an unfitted grid method,
they can easily be reduced to the simpler case of a boundary-fitted grid with the removal of the ghost-penalty stabilization terms.
\end{abstract}


\end{frontmatter}



\section{Introduction}
\label{sec:introduction}

For an incompressible Newtonian fluid in $\Omega \subset \RR^d$, $d=2,3$, which is a bounded, connected domain with the boundary $\partial \Omega$, the following Navier-Stokes equations with a general Navier boundary condition are valid,

\begin{alignat}{2}
  \label{eq:navier-stokes-problem-momentum}
  \partial_t \bfu + \bfu\cdot\nabla\bfu - \nabla\cdot(2\nu\bfepsilon(\bfu)) + \nabla p 
&= \bff \quad &&\text{ in } \Omega,
  \\
  \div\bfu &= 0 \quad &&\text{ in } \Omega,
  \label{eq:navier-stokes-problem-compressible}
  \\
  \left( \bfu - \bfg \right) \boldsymbol{P}^n &= \bfzero
  \quad &&\text{ on } \partial\Omega,
    \label{eq:navier-stokes-problem-boundary-GNBC-normal}\\
  \left( \varepsilon  (2 \nu \bfepsilon(\bfu) \bfn - \bfh) + \nu (\bfu -\bfg) \right) \boldsymbol{P}^t &= \bfzero
  \quad &&\text{ on } \partial\Omega,
  \label{eq:navier-stokes-problem-boundary-GNBC-tangential}
\end{alignat}
where $\bfepsilon(\bfu) := \frac{1}{2} ( \nabla \bfu + (\nabla \bfu)^T )$ and $\bfu, \bfg, \bfh : \left[ 0,T \right] \times \Omega \rightarrow \RR^d$ and $\nu, p:\left[ 0,T \right] \times \Omega \rightarrow \RR$. Here $\bfu$ is the velocity and $p$~and $\nu$ the pressure and kinematic viscosity of the fluid respectively. The terms $\bfg$ and $\bfh$ are the velocity and traction prescribed at the boundary $\partial \Omega$.
The projection matrices ($\boldsymbol{P}^n,\boldsymbol{P}^t\in\mathbb{R}^{d\times d}$) are constructed from the outward pointing unit normal $\bfn$ of the boundary $\partial \Omega$.
The projection in normal direction is defined as $\boldsymbol{P}^n := \bfn \otimes \bfn$ and the projection onto the tangential plane of the surface as $\boldsymbol{P}^t := \boldsymbol{I} - \bfn \otimes \bfn$, where  $\boldsymbol{I}$ is the $d\times d$ identity matrix.

The boundary condition applied to the problem \eqref{eq:navier-stokes-problem-momentum}-\eqref{eq:navier-stokes-problem-compressible} is the general Navier boundary condition.
It prescribes a Dirichlet condition in the normal direction \eqref{eq:navier-stokes-problem-boundary-GNBC-normal}, where the normal velocity of the fluid needs to be the same as the normal velocity prescribed at the boundary. In the tangential plane \eqref{eq:navier-stokes-problem-boundary-GNBC-tangential}, a Robin condition is imposed, which is a condition comprised of a linear combination  of a Dirichlet and Neumann condition.
The slip length parameter $\varepsilon: \partial \Omega \rightarrow \mathbb{R}_0^+\cup\{\infty\}$ determines the influence of the 
two
parts in the tangential direction. 
In the limiting case where $\varepsilon \rightarrow 0$, the formulation reduces to the classic Dirichlet condition $\bfu = \bfg \text{ on } \partial \Omega$; and in the case where $\varepsilon \rightarrow \infty$, a combination of a Dirichlet condition in normal direction $\bfu \boldsymbol{P}^n = \bfg \boldsymbol{P}^n$ and a Neumann condition in tangential direction $2 \nu \bfepsilon(\bfu) \bfn \boldsymbol{P}^t = \bfh \boldsymbol{P}^t$ is enforced. 

This type of boundary condition was first proposed by Navier \cite{Navier1823} with $\bfg = \bfh = \bf0$ as a boundary condition for incompressible fluids at solid walls.
However, in the majority of cases, it has been verified that the slip length is negligibly small and as such the no-slip condition, which is a pure Dirichlet type condition with $\bfu = \bfg$, is normally used to describe the motion of fluids.
Nevertheless, there are some cases where a Navier slip type model has merit, such as for flow over super-hydrophobic surfaces \cite{OuPerotRothstein2004_PoF} and in the modeling of rough surfaces \cite{Mikelic2009_QP_review}.
It can also be shown that for contact between smooth rigid bodies in an incompressible fluid, the no-slip condition is not a feasible boundary condition \cite{HillairetTakahashi2009_SIAM}. One option to alleviate this issue is to employ a Navier slip boundary condition at the boundaries of the rigid bodies \cite{NeustupaPenel2010_AMFM, GerardHillairetWang2015_JMPA}.
Furthermore, in modeling the motion of contact lines for multiphase flows, it is well known that the use of the no-slip boundary condition yields non-physical infinite dissipation in the vicinity of the contact line \cite{HuhScriven1971_JCIS}. To overcome this issue, an alternative is to model the contact line motion by using a Navier slip type boundary condition \cite{QianWangSheng2006_JFM, GerbeauLelievre2009_CMAME}.

The work done on how to impose the Navier boundary condition can be divided based on how it enforces the no-penetration condition \eqref{eq:navier-stokes-problem-boundary-GNBC-normal}, i.e. the condition in normal direction. In the work done by Verf\"urth, a strong imposition is proposed in \cite{Verfuerth1985_RAIRO}, and by means of a weak imposition through Lagrange multipliers in \cite{Verfuerth1986_NM, Verfuerth1991}. The constraint in normal direction can also be enforced weakly by means of Nitsche's method \cite{UrquizaGaronFarinas2014} or by a penalty method \cite{John2002_JCAM}.
In all the previously mentioned works, the tangential condition \eqref{eq:navier-stokes-problem-boundary-GNBC-tangential} is treated by substituting the traction in the variational formulation at the boundary with the following,
\begin{equation}
  2 \nu \bfepsilon(\bfu) \bfn \boldsymbol{P}^t = \bfh \boldsymbol{P}^t  -\frac{\nu}{\varepsilon}\left( \bfu -\bfg \right) \boldsymbol{P}^t 
  \quad \text{ on } \partial\Omega.
  \label{eq:navier-stokes-problem-boundary-GNBC-tangential-substitution}
\end{equation}
This method of imposing the boundary condition leads to numerical difficulties when $\varepsilon \rightarrow 0$ and is not defined for a Dirichlet condition $\varepsilon = 0$.
Thus, the slip length can not be chosen arbitrarily, and it is difficult to predict at what value numerical issues will arise. 
The same problem can be observed in the Poisson equation with Robin boundary conditions, for which a solution was presented by \citet{JuntunenStenberg2009a}.
In their formulation, the Robin condition is imposed weakly by means of Nitsche's method, which permits any choice of $\varepsilon$. 
Based on this method, the present work proposes a formulation to extend these results to the Navier-Stokes equations.




A few examples have already been mentioned, such as contact between rigid bodies in an incompressible fluid and moving contact lines in multi-phase flows, where large deformation and topological changes occur and where the Navier slip boundary condition has merit.
To simulate these types of problems, an unfitted mesh instead of a boundary-fitted mesh method is advantageous.
This is due to the fact that the boundary-fitted mesh methods, which require Arbitrary-Lagrangian-Eulerian (ALE) based mesh moving algorithms, will break down and necessitate expensive remeshing \cite{Wall2008} under those conditions.


Unfitted mesh methods have already successfully been used in simulating a variety of different problems. 
These include single-phase \cite{SchottWall2014} and multi-phase flows \cite{ChessaBelytschko2003, GrossReusken2007, SchottRasthoferGravemeierWall2015, HansboLarsonZahedi2014}, which were conducted by using an implicit level set to describe the interface between the fluids and an unfitted finite element approach for solving the governing equations. 
For fluid-structure interaction problems with large deformations of the structure \cite{Gerstenberger2008, Legay2006, CourtFournie2015, Burman2014a}, and in the case of contact of submerged bodies \cite{Mayer2010}, a fixed-grid Eulerian approach for the fluid and a Lagrangian description for the structure have been applied successfully.
Unfitted mesh methods also offer the possibility to use embedding meshes \cite{Hansbo2003, Shahmiri2011, MassingLarsonLoggEtAl2015, SchottShahmiriKruseWall2015} to simplify meshing or to improve accuracy in certain regions of the mesh.

The Cut Finite Element Method (CutFEM) is the unfitted mesh method of choice in this work. 
This method stems from the eXtended Finite Element Method \cite{Moes1999,Belytschko1999}; for an overview of CutFEM see, e.g., \cite{BurmanClausHansboEtAl2014}.
To prescribe boundary conditions on non-fitted domain boundaries, a common choice is to impose them weakly. This can be done for instance with Nitsche's method \cite{Nitsche1971}. Stability and \apriori~error estimates have been derived for this setup for the Poisson \cite{BurmanHansbo2012}, Stokes \cite{Burman2014b} and recently for the Oseen problem \cite{MassingSchottWall2016_CMAME_Arxiv_submit} with weakly imposed Dirichlet boundary conditions. Another advantage of Nitsche's method is its flexibility and straightforward extensibility to coupled multiphysics problems \cite{Burman2014a}, where the strength of the imposition of the boundary conditions is regulated by the resolution of the physical variables of the computational mesh.


The Oseen equations, which are addressed in this work, can be viewed as a linearized version of the Navier-Stokes equations when time-stepping methods are applied. 
Since this work is also investigating the stability and \apriori~error properties of an unfitted Nitsche's method for the Oseen problem, but with a more general type of boundary condition, a large portion of the analysis relies on the work done in \cite{MassingSchottWall2016_CMAME_Arxiv_submit}. As the proposed method uses equal order continuous finite elements for the velocity and pressure, stabilization of the discretized equations is a necessity. The method chosen here is the continuous interior penalty (CIP) method by \citet{BurmanFernandezHansbo2006}; see \cite{BraackBurmanJohnEtAl2007} for an overview of different stabilization methods applicable to the Oseen problem.
Further stabilization is necessary at the boundaries as these are cut out from the background mesh. If such stabilization is not applied, the method is not necessarily stable for pathological cut cases. Our choice of stabilization at the boundary is the ghost-penalty method introduced by \citet{Burman2010} for the Poisson equation and adapted to the Oseen and Navier-Stokes equations by \cite{MassingSchottWall2016_CMAME_Arxiv_submit} and \cite{SchottWall2014}.


The contributions of this work can be summarized as follows:
A weakly imposed general Navier boundary condition is introduced. 
It consists of a normal and a tangential component, which are separated by the use of projection operators at the boundary.
In the normal direction a Dirichlet condition is imposed by means of Nitsche's method. In the tangential plane a Robin condition is also imposed by a Nitsche's method, which has the advantage over the classical Galerkin substitution \eqref{eq:navier-stokes-problem-boundary-GNBC-tangential-substitution} in that it remains stable and well-posed for all slip lengths $\varepsilon$. The formulation in the tangential plane is inspired by the work of \cite{JuntunenStenberg2009a}, where a similar problem was solved for the Robin boundary condition for the Poisson equation.
A numerical analysis is conducted to prove that our formulation is stable and has optimal convergence behavior. 
The analysis builds upon the work done in \cite{MassingSchottWall2016_CMAME_Arxiv_submit}, where stability and optimality of the error estimates for the Oseen equations with CutFEM was done for a weak imposition of a Dirichlet boundary condition. 
Here, however, the analysis is carried out for a more general case of the boundary condition. 




%
The paper is organized as follows.
Section~\ref{sec:oseen-problem} introduces the Oseen problem, and the necessary spaces and assumptions are specified;
also the variational formulation of the problem is presented.
In Section~\ref{sec:cut-finite-element} the definitions for the cut finite element spaces are introduced.
The general Navier boundary condition imposed by Nitsche's method is presented and its differences and
advantages are highlighted against alternative methods.
The stabilized weak form of the Oseen equations, discretized by means of the CutFEM and the necessary stabilizations of the domain through the CIP method and the cut boundaries by means of ghost-penalty stabilization, are also explained in this section.
In Section~\ref{sec:interpolation-est} basic approximation properties and interpolation operators and norms are introduced.
Sections~\ref{sec:stability-properties} and~\ref{sec:apriori-analysis} are dedicated to proving the inf-sup stability and the optimal \apriori~error of the proposed method.
In Section~\ref{sec:numexamples} the method is applied to a numerical example and a convergence study is conducted to verify the theoretical results.
The final Section~\ref{sec:conclusions} gives a summary and conclusion of this work.

\section{The Oseen Problem with General Navier Boundary Conditions}
\label{sec:oseen-problem}

\subsection{Basic Notation on Function Spaces}
\label{ssec:notation}
We define our domain as $\Omega \subset \RR^d$, $d \in \{2,3\}$, where $\Omega$ is either a bounded convex domain or a plane bounded domain with Lipschitz and piecewise $C^2$ boundary with convex angles.
The standard Sobolev spaces are denoted by $W^{m,q}(U)$ for $U \in \{\Omega, \Gamma \}$ where
$ 0\leqslant m < \infty$ and $1 \leqslant q \leqslant \infty$ with associated norms $\|\cdot \|_{m,q,U}$. 
We write $H^m(U) = W^{m,2}(U)$ with norm $\|\cdot\|_{m,U}$.
For the associated inner products we write $(\cdot,\cdot)_{m,U}$ for measurable subsets $U\subseteq\RR^d$ and $\langle \cdot,\cdot \rangle_{m,U}$
for subsets $U\subseteq\RR^{d-1}$.
Occasionally, we write $(\cdot,\cdot)_{U}$, $\langle\cdot,\cdot\rangle_{U}$ and $\|\cdot \|_{U}$ if $m=0$.
Fractional Sobolev trace spaces $[H^{s-\onehalf}(Y)]^d$ at subsets $Y\subset \RR^{d-1}$,
which in practice will be a part of the domain boundary $\Gamma$,
denote the set of boundary traces of all $\RR^d$-valued functions in $[H^{s}(\Omega)]^d$.
For the Oseen problem, we make use of the specific function spaces
$H_0(\div;\Omega) \subset [L^2(U)]^d$, which denotes the space of divergence-free
functions, and $L^2_0(\Omega)$, which denotes the
function space consisting of functions in $L^2(\Omega)$ with zero average on $\Omega$.
To shorten the presentation, for broken norms and inner products which involve a collection of geometric entities $\mcP_h$,
i.e. finite elements or facets, we write $\|\cdot\|_{\mcP_h}^2 =
\sum_{P\in\mcP_h} \|\cdot\|_P^2$ whenever $\|\cdot\|_P$ is well-defined.

\subsection{Problem Formulation}
\label{ssec:problem_formulation}

Considering the non-linear Navier-Stokes equations \eqref{eq:navier-stokes-problem-momentum}--\eqref{eq:navier-stokes-problem-compressible},
after applying a time discretization method and a linearization step, 
many solution algorithms can be reduced to solving a sequence of auxiliary problems of Oseen type
for the velocity field $\bfu:\Omega\rightarrow \mathbb{R}^d$ and the pressure field $p:\Omega\rightarrow \mathbb{R}$:
\begin{alignat}{2}
  \label{eq:oseen-problem-momentum}
  \sigma \bfu + \bfbeta\cdot\nabla\bfu - \nabla\cdot(2\nu\bfepsilon(\bfu)) + \nabla p 
&= \bff \quad &&\text{ in } \Omega,
  \\
  \div\bfu &= 0 \quad &&\text{ in } \Omega,
  \label{eq:oseen-problem-compressible}
  \\
  \left( \bfu - \bfg \right) \boldsymbol{P}^n &= \bfzero
  \quad &&\text{ on } \GammaGNBC,
    \label{eq:oseen-problem-boundary-GNBC-normal}\\
  \left( \varepsilon  (2 \nu \bfepsilon(\bfu) \bfn - \bfh) + \nu (\bfu -\bfg) \right) \boldsymbol{P}^t &= \bfzero
  \quad &&\text{ on } \GammaGNBC,
  \label{eq:oseen-problem-boundary-GNBC-tangential}
\end{alignat}
where $\bfepsilon(\bfu) := \frac{1}{2}(\nabla \bfu + (\nabla \bfu)^T)$
denotes the strain rate tensor,
$\bfbeta \in [W^{1,\infty}(\Omega)]^d \cap H_0(\div;\Omega)$
the given divergence-free advective velocity field and
$\bff\in [L^2(\Omega)]^d$ the body force.
The assumption is made that the reaction coefficient $\sigma$ and the viscosity $\nu$ are positive real-valued constants.
For the boundary conditions \eqref{eq:oseen-problem-boundary-GNBC-normal}--\eqref{eq:oseen-problem-boundary-GNBC-tangential}
the functions are defined as
\begin{equation}
\bfg\in [H^{3/2}(\GammaGNBC)]^d 
\qquad\text{ and }\qquad 
\bfh\in [H^{1/2}(\GammaGNBC)]^d,
\end{equation}
and the normal and tangential projection matrices are constructed from the outward pointing unit normal $\bfn$ of the boundary $\GammaGNBC$ as
$\boldsymbol{P}^n := \bfn \otimes \bfn$ and $\boldsymbol{P}^t := \boldsymbol{I} - \bfn \otimes \bfn$, where $\boldsymbol{I}$ is the $d\times d$ identity matrix.
Utilizing a non-negative slip length function $\varepsilon: \GammaGNBC\rightarrow \mathbb{R}_0^+\cup\{\infty\}$,
the classical full Dirichlet boundary conditions can be recovered by setting $\varepsilon=0$,
which for $\nu > 0$ states that $\bfu= \bfg \text{ on } \Gamma_D:=\{\bfx\in\GammaGNBC \text{ with } \varepsilon=0\}\subseteq{\GammaGNBC}$;
however, the tangential condition \eqref{eq:oseen-problem-boundary-GNBC-tangential} does not contribute in the Darcy limit, i.e. for $\nu =0$, 
and as such, the normal condition \eqref{eq:oseen-problem-boundary-GNBC-normal} is sufficient to define a Dirichlet boundary condition.
By choosing $\varepsilon=\infty$ and $\bfh = \bf0$ the boundary condition reduces to a full-slip condition
as the tangential velocity on $\GammaGNBC$ is not constrained anymore.
The part of $\GammaGNBC$ with $0\leqslant \varepsilon < \infty$ is denoted by $\GammaGNBC^\varepsilon$
and is assumed being non-vanishing, i.e. $\meas(\GammaGNBC^\varepsilon)>0$,
to avoid not well-posed pure slip-boundary problems.
It is further assumed that the inflow boundary $\Gamma_{\mathrm{in}}:=\{\bfx\in\Gamma~|~\bfbeta\cdot\bfn<0\}$
is part of the Dirichlet boundary, i.e. $\Gamma_{\mathrm{in}}\subseteq \Gamma_D$.
In the discrete setting, the constant pressure mode needs to be filtered out during the solution procedure, as it is determined only up to a constant.

\subsection{Variational Formulation}
\label{ssec:variational_formulation}

Let us denote the velocity and pressure space by
\begin{align}
\mcV_{\bfg}&:=\{\bfu\in [H^1(\Omega)]^d ~|~(\bfu-\bfg)\boldsymbol{P}^n=\bfzero \text{ on } \GammaGNBC \wedge \bfu=\bfg \text{ on } \Gamma_D\},
\\
  \mcV^n_{\bfzero}&:=\{\bfu\in [H^1(\Omega)]^d ~|~\bfu\boldsymbol{P}^n=\bfzero \text{ on } \GammaGNBC \wedge \bfu=\bfzero \text{ on } \Gamma_D\},
\\
  \mcQ &:= L_0^2(\Omega).
\end{align}
The corresponding weak formulation of the Oseen
problem~\eqref{eq:oseen-problem-momentum}--\eqref{eq:oseen-problem-boundary-GNBC-tangential} is to find the velocity and the pressure field
 $(\bfu,p) \in \mcV_{\bfg} \times \mcQ$
such that
\begin{align}
  a(\bfu, \bfv) + b(p, \bfv) - b(q, \bfu) + \langle \frac{\nu}{\epsilon} \bfu \boldsymbol{P}^t , \bfv \rangle_\Gamma = l(\bfv) + \langle (\bfh + \frac{\nu}{\epsilon} \bfg) \boldsymbol{P}^t , \bfv \rangle_\Gamma\quad \foralls
  (\bfv,q) \in \mcV^n_{\bfzero} \times \mcQ,
  \label{eq:oseen_weak}
\end{align}
where
\begin{align}
  \label{eq:a-form-def}
  a(\bfu, \bfv) &:= (\sigma \bfu, \bfv)_{\Omega} +
  (\bfbeta\cdot\nabla\bfu, \bfv)_{\Omega} + (2\nu\bfepsilon(\bfu),
  \bfepsilon(\bfv))_{\Omega},    \\
  \label{eq:b-form-def}
  b(p, \bfv) &:= - (p, \div\bfv)_{\Omega}, 
  \\
  l(\bfv) &:= (\bff, \bfv)_{\Omega}.
  \label{eq:l-form-def}
\end{align}
For the well-posedness and solvability of the continuous Oseen
problem~\eqref{eq:oseen_weak}, we refer the reader to, e.g. \cite{RaviartGirault1986,Solonnikov1983, Verfuerth1991}.

\section{A Stabilized Nitsche-type Cut Finite Element Method for the Oseen Problem}
\label{sec:cut-finite-element}

This section is devoted to the presentation of our cut finite element method.
After defining suitable cut finite element function spaces, we address the weak imposition of
general Navier boundary conditions for the Oseen problem by a Nitsche-type method. A possible stabilization technique
to overcome issues of classical finite element approximations is recalled and we discuss how the resulting discrete formulation can be extended to non-boundary-fitted
approximation with the help of boundary-zone ghost-penalty stabilizations.
Finally, we summarize our cut finite element formulation and introduce suitable norms for the numerical stability and \apriori~error analysis.

\subsection{Computational Meshes and Cut Finite Element Spaces}
\label{ssec:cutfem-spaces}

While for standard finite element methods the computational mesh is fitted to the boundary
or an interpolatory approximation is given, in cut finite element methods
the boundary is allowed to intersect the mesh.
For the sake of simplicity, in this work, we assume a sequence of quasi-uniform meshes $\widehat{\mcT}_h = \{T\}$,
each consisting of shape-regular finite elements $T$ with mesh size parameter~$h$
and covering the physical domain~$\Omega$.
For each background mesh $\widehat{\mcT}_h$, the finite element solution
is then approximated on an \emph{active} part of the background mesh
\begin{align}
  \mcT_h := \{ T \in \widehat{\mcT}_h : T \cap \Omega \neq \emptyset \},
\end{align}
consisting of all elements in $\widehat{\mcT}_h$ which intersect the physical domain $\Omega$.
The possibly enlarged domain which is covered by the union of all elements $T \in \mcT_h$ is denoted by $\Oasth$.
A mesh $\mcT_h$ is called a \emph{fitted mesh} if $\overline{\Omega} = \overline{\Oast_h}$
and an \emph{unfitted mesh} if $\overline{\Omega} \subsetneq \overline{\Oast_h}$. 
The subset of all elements of the respective active mesh which are located in the vicinity of the boundary $\Gamma$
are denoted as
\begin{equation}
  \label{eq:define-cutting-cell-mesh}
  \mcT_{\pO} := \{T \in \mcT_h: T \cap \pO \neq \emptyset \}.
\end{equation}

%
\begin{figure}[tb]
  \begin{center}
    \includegraphics[width=0.42\textwidth]{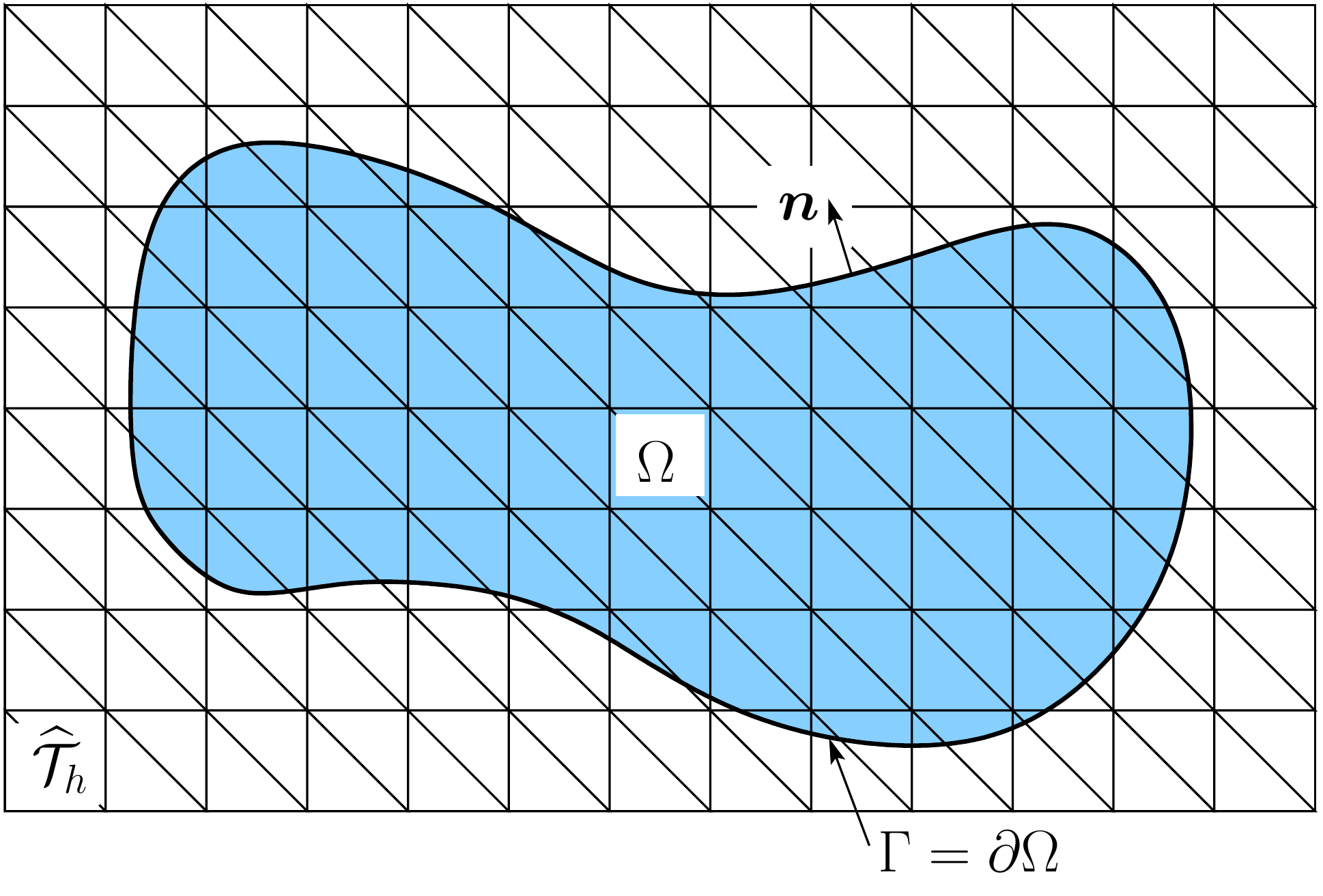} 
    \qquad\quad
    \includegraphics[width=0.42\textwidth]{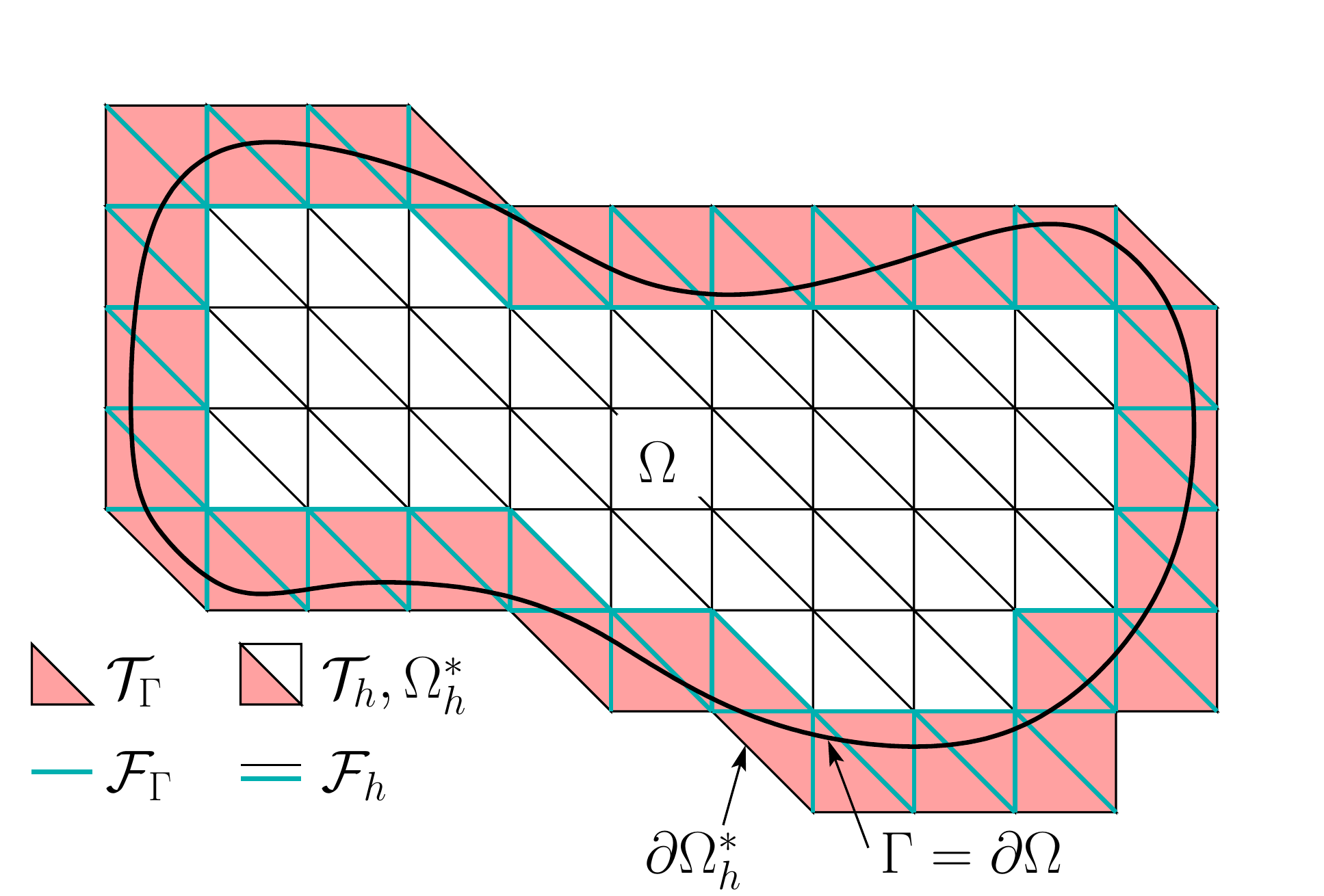}
  \end{center}
  \caption{Left: The physical domain $\Omega$ is defined as the
    inside of a given boundary $\pO$ with outward pointing unit normal~$\bfn$ embedded into a fixed
    background mesh $\widehat{\mcT}_h$. Right: The fictitious
    domain~$\Oast_h$ is the union of the minimal subset $\mcT_h
    \subset \widehat{\mcT}_h$ covering $\Omega$ and defines the active computational mesh.
Sets of elements $\mcT_{\Gamma}$ and facets $\mcF_{\Gamma}$ are indicated.}
  \label{fig:computational-domain}
\end{figure}
%

\noindent For stabilization purposes we need to define the set of all facets by $\mcF_h$,
where $\mcF_i$ is the set of all \emph{interior facets} $F$ which are
shared by exactly two elements, denoted by $T^+_F$ and $T^-_F$.
The notation $\mcF_{\Gamma}$ is used for the set of all interior facets
belonging to elements intersected by the boundary $\pO$,
\begin{equation}
  \label{eq:ghost-penalty-facets}
  \mcF_{\Gamma} := \{ F \in \mcF_i :\;
  T^+_F \cap \pO \neq  \emptyset
  \vee
  T^-_F \cap \pO \neq  \emptyset
  \}.
\end{equation}
\Figref{fig:computational-domain} summarizes the notation.

By analogy to other preceding works on cut finite element methods \cite{HansboHansbo2002,BurmanHansbo2012,MassingLarsonLoggEtAl2014,BurmanClausMassing2015,MassingSchottWall2016_CMAME_Arxiv_submit},
we assume that the domain and the boundary are reasonably resolved by each quasi-uniform mesh $\mcT_h$.
This entails that the intersection between the facet $F \in \mcF_i$ and the boundary $\Gamma$ is simply connected.
Furthermore, there needs to exist a plane $S_T$ with a piecewise smooth parametrization $\Phi: S_T \cap T \rightarrow \Gamma \cap T$ for each element $T$ intersected by $\Gamma$. 
Finally an assumption is made that for each element $T \in \mcT_{\pO}$ there exists $T' \in \mcT_h \setminus \mcT_{\pO}$ such that the sequence $\{T\}_{j=1}^N$ with $T_1 = T,\,T_N = T'$ and $T_j \cap T_{j+1}  \in \mcF_i,\; j = 1,\ldots, N-1$ is at most $N$ elements.

Let $\mcX_h$ be the finite element spaces consisting of continuous 
piecewise polynomials of order $k$ for a given mesh $\mcT_h$
\begin{align}
  \mcX_h &= \left\{ v_h \in C^0(\Oast_h): \restr{v_h}{T} \in \mcP^k(T)\, \foralls T \in \mcT_h \right\}.
\end{align}
For the discrete cut finite element approximation of the solution to the Oseen problem,
we use equal-order interpolations for velocity and pressure, where
\begin{align}
  \mcV_h = [\mcX_h]^d, \quad \mcQ_h = \mcX_h, \quad \mcW_h = \mcV_h
  \times \mcQ_h
  \label{eq:approximation_spaces}
\end{align}
are the discrete velocity space, the discrete pressure space and the total approximation space.

\subsection{Weak Imposition of Generalized Navier Boundary Conditions}
\label{ssec:weak_imposition_bcs_oseen}

In this section, we present the imposition of generalized Navier boundary conditions for the Oseen problem.
As previously mentioned, this boundary condition consists of a Dirichlet condition in normal direction to the boundary and a Robin condition in the tangential plane of the boundary.
%
Since the boundary-normal constraint \eqref{eq:oseen-problem-boundary-GNBC-normal} is of Dirichlet-type, either a strong or a weak imposition can be used. Nevertheless, in the case of an unfitted mesh, a weak imposition is recommended.
In the case when the slip-length coefficient is sufficiently large,
the tangential constraint
\eqref{eq:oseen-problem-boundary-GNBC-tangential} can be imposed by direct substitution of the boundary terms,
which stems from integrating the viscous bulk term in \eqref{eq:oseen-problem-momentum} by parts,
i.e. by substituting
\begin{equation}
\label{eq:substitution_method}
  -\langle (2\nu\bfepsilon(\bfu_h)\bfn)\boldsymbol{P}^t, \bfv_h \rangle_{\GammaGNBC} = \langle ( -\bfh +  \varepsilon^{-1} \nu (\bfu -\bfg) ) \boldsymbol{P}^t, \bfv_h\rangle_{\GammaGNBC}.
\end{equation}
Such a formulation provides a reasonable and stable method for large slip-length coefficients $\varepsilon$.
However, it is well known that for $\varepsilon\rightarrow 0$, 
i.e. when we approach a Dirichlet condition in the tangential direction,
the conditioning of the formulation deteriorates and, as a result, does not provide an accurate and stable formulation,
see discussion by \citet{JuntunenStenberg2009a}.
To overcome this issue in the case of the Poisson problem,
Juntunen and Stenberg
proposed a 
Nitsche's method of imposing a Robin boundary condition, 
which provides stability, optimal \apriori~error estimates and a bounded conditioning w.r.t. small choices of $\varepsilon \in [0, \infty]$.

Next, we present our novel Nitsche-type method for the Oseen problem with general Navier boundary conditions.
It combines the advantages of the Nitsche-type method by \citet{JuntunenStenberg2009a} for generalized Robin-type conditions
with the Nitsche formulations developed by \citet{BurmanFernandezHansbo2006}
and \citet{MassingSchottWall2016_CMAME_Arxiv_submit} for incompressible low- and high-Reynolds-number flow with Dirichlet-type constraints.

The unfitted Nitsche-type finite element formulation for 
the Oseen problem is to find $U_h=(\bfu_h, p_h) \in \mcW_h$ such that
for all $V_h=(\bfv_h, q_h) \in \mcW_h$
\begin{align}
 A_h(U_h,V_h) + S_h(U_h, V_h) + G_h(U_h, V_h)
&= L_h(V_h),
 \label{eq:oseen-discrete-fitted-stabilized}
\end{align}
where
\begin{align}
A_h(U_h,V_h)
  &:= a_h(\bfu_h, \bfv_h)+b_h(p_h, \bfv_h)-b_h(q_h,\bfu_h),
 \label{eq:Ah-form-def}
\intertext{with}
  \label{eq:ah-form-def}
  a_h(\bfu_h, \bfv_h)
  &:= a(\bfu_h, \bfv_h)   - \langle 2\nu\bfepsilon(\bfu_h)\bfn, \bfv_h \rangle_{\GammaGNBC}
  \\
  &\quad\quad
  - \zeta_u \langle \bfu_h \cdot \bfn, (2 \nu \bfepsilon(\bfv_h)\bfn) \cdot \bfn \rangle_{\GammaGNBC}
  + \langle
    \frac{\nu}{\gamma^n h} \bfu_h \cdot \bfn,\bfv_h \cdot \bfn
  \rangle_{\GammaGNBC}
  \label{eq:ah-form-def_normal_adj_consistency_penalty}
  \\
  &\quad\quad
  + \langle  \frac{\phi_u}{\gamma^n h} \bfu_h\cdot\bfn, \bfv_h\cdot\bfn \rangle_{\GammaGNBC}
  - \langle (\bfbeta\cdot\bfn)\bfu_h, \bfv_h \rangle_{\GammaIn}
  \label{eq:ah-form-def_normal_regime_penalty}
  \\
  &\quad\quad
  + \langle \frac{1}{\varepsilon + \gamma^t h}  \varepsilon (2 \nu \bfepsilon(\bfu_h)\bfn )  \boldsymbol{P}^t
  , \bfv_h \rangle_{\GammaGNBC}
  +\langle \frac{1}{\varepsilon + \gamma^t h} \nu \bfu_h \boldsymbol{P}^t
   , \bfv_h \rangle_{\GammaGNBC}
  \label{eq:ah-form-def_tangential_1}
  \\
  &\quad\quad  
  - \zeta_u \langle \frac{ \gamma^t h }{\varepsilon + \gamma^t h} \varepsilon (2 \nu \bfepsilon(\bfu_h) \bfn)  \boldsymbol{P}^t
  , 2 \bfepsilon(\bfv_h) \bfn \rangle_{\GammaGNBC}
  - \zeta_u \langle \frac{  \gamma^t h}{\varepsilon + \gamma^t h} \nu \bfu_h  \boldsymbol{P}^t
  , 2 \bfepsilon(\bfv_h) \bfn \rangle_{\GammaGNBC},
  \label{eq:ah-form-def_tangential_2}
  \\
  \label{eq:bh-form-def}
  b_h(p_h, \bfv_h) &:= b(p_h, \bfv_h)+
  \langle p_h,
  \bfv_h\cdot\bfn
  \rangle_{\GammaGNBC},
  \\
  \label{eq:Lh-form-def}
  L_h(V_h)
  &:=
  l(\bfv_h)
\\
&\quad\quad
  - \zeta_u \langle \bfg \cdot \bfn, ( 2 \nu \bfepsilon(\bfv_h)\bfn ) \cdot \bfn \rangle_{\GammaGNBC}
  + \langle \frac{\nu}{\gamma^n h}  \bfg \cdot \bfn,\bfv_h \cdot \bfn \rangle_{\GammaGNBC}
\label{eq:lh-form-def_normal_adj_consistency_penalty}
\\
&\quad\quad
  - \langle \bfg \cdot \bfn, q_h \rangle_{\GammaGNBC}
  + \langle  \frac{\phi_u}{\gamma^n h} \bfg\cdot\bfn, \bfv_h\cdot\bfn \rangle_\Gamma
  - \langle (\bfbeta\cdot\bfn)\bfg, \bfv_h \rangle_{\GammaIn}
\label{eq:lh-form-def_normal_regime_consistency_penalty}
\\
&\quad\quad
  + \langle \frac{1}{\varepsilon + \gamma^t h} \varepsilon \bfh \boldsymbol{P}^t
  , \bfv_h \rangle_{\GammaGNBC}
  +\langle \frac{1}{\varepsilon + \gamma^t h} \nu \bfg \boldsymbol{P}^t
   , \bfv_h \rangle_{\GammaGNBC} 
\label{eq:lh-form-def_tangential_1}
\\
&\quad\quad
  - \zeta_u \langle \frac{  \gamma^t h  }{\varepsilon + \gamma^t h} \varepsilon \bfh  \boldsymbol{P}^t
  , 2 \bfepsilon(\bfv_h) \bfn \rangle_{\GammaGNBC}
  - \zeta_u \langle \frac{ \gamma^t h }{\varepsilon + \gamma^t h} \nu \bfg  \boldsymbol{P}^t
  , 2 \bfepsilon(\bfv_h) \bfn \rangle_{\GammaGNBC},
\label{eq:lh-form-def_tangential_2}
\end{align}
where $\phi_{u}: \Omega_h^* \rightarrow \RR$ is defined on each element as
\begin{equation}
\label{eq:cip-Nitsche_scaling}
  \phi_{u,T} := \nu +\|\bfbeta\|_{0,\infty,T} h + \sigma h^2.
\end{equation}
The stabilization operators $S_h, G_h$ will be specified later in Sections~\ref{ssec:stabilized-oseen} and \ref{ssec:cutfem-oseen}.
Starting from the weak formulation \eqref{eq:oseen_weak},
it can be seen that due to the weak constraint enforcement the standard consistency boundary terms remain in the momentum equation \eqref{eq:ah-form-def} and \eqref{eq:bh-form-def}.
Equivalent to a Nitsche formulation for a pure Dirichlet boundary condition, these standard consistency terms on the boundary $\Gamma$ are potential sources for instabilities and are analyzed further in Sections~\ref{sec:stability-properties} and \ref{sec:apriori-analysis}.
The constraints being imposed in wall-normal and tangential directions, i.e. \eqref{eq:oseen-problem-boundary-GNBC-normal} and
\eqref{eq:oseen-problem-boundary-GNBC-tangential},
are enforced by adding additional consistent boundary terms.




\noindent\textbf{Wall-normal constraint:}
The enforcement of the boundary-normal constraint follows the standard Nitsche-technique
by adding an adjoint (in-)consistent viscous term (depending on $\zeta_u\in\{-1,1\})$ and
a consistent and optimal convergent viscous symmetric penalty term \eqref{eq:ah-form-def_normal_adj_consistency_penalty}
with appropriate right-hand-side terms \eqref{eq:lh-form-def_normal_adj_consistency_penalty}.
Choosing an adjoint-consistent method ($\zeta_u=1$), the discrete formulation requires the penalty parameter $0 <\gamma^n \leqslant C$
to be chosen small enough, where the constant~$C$ depends on the shape and polynomial order of the finite element, cf. the trace estimate \eqref{eq:inverse-estimates} in Section~\ref{sec:interpolation-est}.
Furthermore, this choice allows to deduce optimal convergence for the velocity $L^2$-error as stated in Theorem~\ref{thm:apriori-estimate-L2}.
Even though an adjoint-inconsistent formulation ($\zeta_u=-1$) enjoys improved inf-sup stability
for any $0<\gamma^n< \infty$ and thereby ensures optimal convergence w.r.t. an energy norm,
optimality for the velocity $L^2$-error is not guaranteed in this case anymore.
For an analysis of penalty-free Nitsche methods, i.e. for $\gamma_n=\infty$, the interested reader is referred to, e.g. \cite{Burman2012a,Boiveau2014}.
Note that even for vanishing viscosity the wall-normal constraints needs to be enforced.
Thus, to ensure inf-sup stability and optimal convergence for all flow regimes, 
an adjoint-consistent pressure term \eqref{eq:bh-form-def} and
a symmetric penalty term \eqref{eq:ah-form-def_normal_regime_penalty},
which accounts for the different flow regimes
as reflected by the definition of the piecewise constant stabilization scaling function $\phi_u$, are added to the formulation.
The respective right-hand-side terms \eqref{eq:lh-form-def_normal_regime_consistency_penalty} again guarantee consistency of the method.
Note that additionally at inflow boundaries $\Gamma_{\mathrm{in}} \subseteq \Gamma_D$ where it holds that $\varepsilon=0$ and $\bfbeta\cdot\bfn<0$,
Dirichlet boundary conditions need to be imposed in all spatial directions as stated by the additional advective inflow stabilization terms in
\eqref{eq:ah-form-def_normal_regime_penalty} and \eqref{eq:lh-form-def_normal_regime_consistency_penalty}.

\noindent\textbf{Wall-tangential constraint:}
The enforcement of the tangential boundary condition follows the technique introduced by \citet{JuntunenStenberg2009a}.
By adding terms where the tangential condition \eqref{eq:oseen-problem-boundary-GNBC-tangential} is tested with $\bfv_h$ and $\bfepsilon(\bfv_h)$, consistency is ensured.
Thereby, choosing the weights of these added terms as presented, guarantees coercivity, optimal \apriori~error estimates and bounded system conditioning w.r.t. the
choice of the slip-length coefficient $\varepsilon\in[0,\infty]$.
Similar to the viscous wall-normal constraint enforcement,
the choice between an adjoint-consistent ($\zeta_u=1$) and an adjoint-inconsistent ($\zeta_u=-1$)
formulation poses equivalent restrictions to $\gamma^t$, i.e. $0<\gamma^t\leqslant C$ for $\zeta_u=1$
(with $C$ stemming from estimate~\eqref{eq:inverse-estimates}) and $0<\gamma^t < \infty$ for $\zeta_u=-1$.

\begin{remark}
 In the limiting case of $\varepsilon \rightarrow 0$, it can easily be seen that the formulation \eqref{eq:oseen-discrete-fitted-stabilized} ends up the same as already presented in \cite{MassingSchottWall2016_CMAME_Arxiv_submit} for the imposition of weak Dirichlet boundary conditions by means of Nitsche's method.
\end{remark}

\begin{remark}
\label{remark:oseen_weak_formulation_large_eps}
  For the case where $\varepsilon \rightarrow \infty$ and $\gamma^t h \ll \varepsilon$ the formulation does not return the familiar imposition of a Neumann condition in the tangential plane.
For clarity, we explicitly state the terms $a_h(\bfu_h, \bfv_h)$ and $L_h(V_h)$ here,
\begin{align}
  a_h(\bfu_h, \bfv_h)
  &:= a(\bfu_h, \bfv_h)   - \langle 2\nu\bfepsilon(\bfu_h)\bfn, \bfv_h \rangle_{\GammaGNBC}
    + \langle  (2 \nu \bfepsilon(\bfu_h)\bfn )  \boldsymbol{P}^t , \bfv_h \rangle_{\GammaGNBC}
  \\
  &\quad\quad
  - \zeta_u \langle \bfu_h \cdot \bfn, (2 \nu \bfepsilon(\bfv_h)\bfn) \cdot \bfn \rangle_{\GammaGNBC}
  + \langle
    \frac{\nu}{\gamma^n h} \bfu_h \cdot \bfn,\bfv_h \cdot \bfn
  \rangle_{\GammaGNBC}
  \label{eq:ah-form-def_normal_adj_consistency_penalty_fullNeumann}
  \\
  &\quad\quad
  + \langle  \frac{\phi_u}{\gamma^n h} \bfu_h\cdot\bfn, \bfv_h\cdot\bfn \rangle_{\GammaGNBC}
  - \langle (\bfbeta\cdot\bfn)\bfu_h, \bfv_h \rangle_{\GammaIn}
  \label{eq:ah-form-def_normal_regime_penalty_fullNeumann}
  \\
  &\quad\quad
  - \zeta_u \langle \gamma^t h (2 \nu \bfepsilon(\bfu_h) \bfn)  \boldsymbol{P}^t
  , 2 \bfepsilon(\bfv_h) \bfn \rangle_{\GammaGNBC},
  \label{eq:ah-form-def_tangential_1_fullNeumann}
  \\
    L_h(V_h)
  &:=
  l(\bfv_h)   + \langle \bfh \boldsymbol{P}^t, \bfv_h \rangle_{\GammaGNBC}
\\
&\quad\quad
  - \zeta_u \langle \bfg \cdot \bfn, ( 2 \nu \bfepsilon(\bfv_h)\bfn ) \cdot \bfn \rangle_{\GammaGNBC}
  + \langle \frac{\nu}{\gamma^n h}  \bfg \cdot \bfn,\bfv_h \cdot \bfn \rangle_{\GammaGNBC}
\label{eq:lh-form-def_normal_adj_consistency_penalty_fullNeumann}
\\
&\quad\quad
  - \langle \bfg \cdot \bfn, q_h \rangle_{\GammaGNBC}
  + \langle  \frac{\phi_u}{\gamma^n h} \bfg\cdot\bfn, \bfv_h\cdot\bfn \rangle_\Gamma
  - \langle (\bfbeta\cdot\bfn)\bfg, \bfv_h \rangle_{\GammaIn}
\label{eq:lh-form-def_normal_regime_consistency_penalty_fullNeumann}
\\
&\quad\quad
    - \zeta_u \langle \gamma^t h \bfh  \boldsymbol{P}^t
  , 2 \bfepsilon(\bfv_h) \bfn \rangle_{\GammaGNBC}.
\label{eq:lh-form-def_tangential_1_fullNeumann}
\end{align}
This limiting case gives a standard Neumann condition as in \eqref{eq:lh-form-def_normal_adj_consistency_penalty_fullNeumann} with the addition of the terms \eqref{eq:ah-form-def_tangential_1_fullNeumann} and \eqref{eq:lh-form-def_tangential_1_fullNeumann}.
These terms are, however, added consistently to the formulation and as such do not ruin its validity.
Nevertheless, these terms are of importance as they impose limitations on the choice of~$\gamma^t$.
The conditioning of this formulation becomes increasingly bad as $\gamma^t$ is increasing.
From this it is clear that the parameter is effectively limited from above even in the case of an adjoint-inconsistent formulation, i.e. $\zeta_u=-1$,
which is proven to be inf-sup stable for $0\leqslant \gamma^t < \infty$ in Section~\ref{sec:stability-properties}.
Another observation is that in the other limiting case where $\gamma^t \rightarrow 0$ we end up with a classical standard Galerkin imposition of a Neumann boundary condition in the tangential plane,
as the terms \eqref{eq:ah-form-def_tangential_1_fullNeumann} and \eqref{eq:lh-form-def_tangential_1_fullNeumann} disappear.
\end{remark}

\subsection{The Continuous Interior Penalty (CIP) Stabilizations}
\label{ssec:stabilized-oseen}

It is well-known that a finite element based
discretization of formulation \eqref{eq:oseen_weak} needs to be stabilized
to allow for equal-order interpolation spaces $\mcW_h = \mcV_h \times \mcQ_h$
due to its saddle-point structure, see e.g. \cite{BrezziFortin1991}, and 
to suppress spurious oscillations in the numerical
solution in the case of convection-dominant flow.
For a detailed overview of different stabilization techniques, see e.g. the overview article \cite{BraackBurmanJohnEtAl2007}.

Following the presentation of the cut finite element formulation from \cite{MassingSchottWall2016_CMAME_Arxiv_submit},
we apply continuous interior penalty (CIP) stabilization operators originally introduced by \citet{BurmanFernandezHansbo2006}
and adapted by \citet{MassingSchottWall2016_CMAME_Arxiv_submit} such that
\begin{equation}
\label{eq:Sh-form-def}
 S_h(U_h,W_h) :=  s_{\beta}(\bfu_h, \bfv_h) +  s_{u}(\bfu_h, \bfv_h) +  s_p(p_h, q_h)
\end{equation}
consists of symmetric stabilization terms, which penalize the jump of the velocity and pressure gradients over interior element facets~$F\in\mcF_i$.
The stabilization operators are defined by
\begin{align}
  \label{eq:cip-s_beta}
  s_{\beta}(\bfu_h, \bfv_h)
  &:=
  \gamma_{\beta}
  \sum_{F\in\mathcal{F}_i}
  \phi_{\beta,F} h
  \langle\jump{\bfbeta \cdot\grad \bfu_h},\jump{\bfbeta \cdot\grad\bfv_h}\rangle_F,
  \\
  \label{eq:cip-s_u}
  s_{u}(\bfu_h, \bfv_h)
  &:=
  \gamma_u
  \sum_{F\in\mathcal{F}_i}
  \phi_{u,F} h
  \langle\jump{\div\bfu_h},\jump{\div\bfv_h}\rangle_F,
  \\
  \label{eq:cip-s_p}
  s_p(p_h, q_h) &:=
  \gamma_p 
  \sum_{F\in\mathcal{F}_i}
  \phi_{p,F} h
  \langle \jump{\bfn_F \cdot \nabla p_h},\jump{\bfn_F \cdot \nabla q_h}\rangle_F,
\end{align}
where for any, possibly vector-valued, piecewise discontinuous function $\phi$
on the computational mesh $\mcT_h$, we denote the jump and average over an interior facet $F\in\mcF_i$
with $\jump{\phi}:= (\phi_{F}^+ - \phi_{F}^-)$ and $\phi_F:=\tfrac{1}{2}(\phi_{F}^+ + \phi_{F}^-)$
where $\phi^{\pm}(\bfx) = \lim_{t\to0^+} \phi(\bfx \pm t \bfn_F)$
for some chosen  normal unit vector $\bfn_F$ on $F$.
The element-wise constant stabilization parameter $\phi_u$ is as defined in \eqref{eq:cip-Nitsche_scaling}
and $\phi_{\beta},\phi_p$ are given as
\begin{align}
  \label{eq:cip-s_scalings}
  \phi_{\beta,T} = \phi_{p,T} =  h^2 \phi_{u,T}^{-1}.
\end{align}

Throughout this work, we use the notation $a \lesssim b$ for
$a\leqslant C b$ for some positive generic constant $C$ which varies with the context,
however, is always independent of the mesh size $h$ and the intersection of the mesh $\mcT_h$ by $\Gamma$.

For a (Lipschitz)-continuous $\bfbeta \in [W^{1,\infty}(\Omega)]^d \subseteq [C^{0,1}(\Omega)]^d$,
in the forthcoming numerical analysis we assume a piecewise constant approximation satisfying
\begin{equation}
\|\bfbeta - \bfbeta_h^0\|_{0,\infty,T} 
\lesssim h
\|\bfbeta\|_{1,\infty,T}
\quad
\text{and}
\quad
\| \bfbeta_h^0 \|_{0,\infty,T} \lesssim  \| \bfbeta \|_{0,\infty,T}
\quad \foralls T\in\mcT_h.
\label{eq:beta-approximation}
\end{equation}
%
Furthermore, in the unfitted mesh case we assume that there exists an extension $\bfbeta^\ast \in [W^{1,\infty}(\Oast)]^d$
from $\Omega$ to $\Oast$ satisfying
  $\| \bfbeta^{\ast} \|_{1,\infty,\Oast}
\lesssim
  \| \bfbeta \|_{1,\infty,\Omega}
$.
Similar to the preceding works by \cite{MassingSchottWall2016_CMAME_Arxiv_submit, BurmanFernandezHansbo2006},
we assume that the flow field $\bfbeta$ is sufficiently resolved by the mesh such that $\foralls T\in\mcT_h$
\begin{align}
\| \bfbeta \|_{0,\infty,T'} 
\lesssim
\| \bfbeta \|_{0,\infty,T} 
\lesssim
\| \bfbeta \|_{0,\infty,T'}
\quad \foralls T'\in\omega(T),
  \label{eq:beta-resolution-I}
\end{align}
where $\omega(T)$ denotes a local patch of elements neighboring~$T$.
As a result, the piecewise constant stabilization parameters are comparable locally in a neighborhood of elements, i.e.
\begin{align}
\label{eq:local_equivalence_stab_param_patch}
 \phi_T \sim \phi_{T'} \quad\forall T'  \in \omega(T) \quad\text{and}\quad
 \phi_F \sim \phi_{T} \quad \forall T \in \omega(F) \quad \text{ for } \phi \in \{\phi_u, \phi_{\beta}, \phi_p \}.
\end{align}

\begin{remark}
 Note that for facets $F\in\mcF_i$, which are intersected by the boundary $\Gamma$, the inner products
of all CIP stabilization operators \eqref{eq:cip-s_beta}--\eqref{eq:cip-s_p} have to be evaluated along the entire cut facets. 
\end{remark}

\subsection{Stabilizing Cut Elements -- The Role of Ghost Penalties}
\label{ssec:cutfem-oseen}

To strengthen the stability properties of the discrete formulation in the boundary zone for non-boundary-fitted meshes $\mcT_h$,
additional measures are required. So-called ghost-penalty stabilizations, comprised in the operator $G_h$,
are active in the boundary zone and augment the stabilized bilinear form $A_h+S_h$ to
account for small cut elements $|T \cap \Omega| \ll |T|,\; T\in \mcT_h$, in the vicinity of the
boundary $\pO$. For detailed elaborations on this concept see, e.g.,
the works \cite{Burman2010,BurmanHansbo2012,MassingLarsonLoggEtAl2014,MassingSchottWall2016_CMAME_Arxiv_submit}.

Ghost-penalty stabilizations extend the stability and approximation properties of the discrete scheme
to the entire active background mesh, i.e. from $\Omega$ to $\Oast_h$,
and thus give control of the discrete velocity and pressure solutions,
where they have no physical significance.
As a major advantage, the resulting scheme has improved optimality properties
with highly reduced sensitivity of the errors and guaranteed uniformly bounded conditioning
irrespective of how the boundary $\Gamma$ intersects the underlying mesh $\mcT_h$.
For this purpose, the different terms need to be designed such that inf-sup stability is guaranteed for all flow regimes
and (weak) consistency and optimality of the numerical scheme is maintained.

In the preceding work by \citet{MassingSchottWall2016_CMAME_Arxiv_submit}, a set of different ghost-penalty terms have been developed to stabilize the Oseen equations.
The suggested terms consist of CIP-type jump penalties of polynomial order~$k$ for velocity and pressure and are recalled in the following
\begin{align}
\label{eq:Gh-form-def}
  G_h(U_h,V_h) &= g_{\sigma}(\bfu_h, \bfv_h) + g_{\nu}(\bfu_h, \bfv_h) +  g_{\beta}(\bfu_h, \bfv_h) + g_{u}(\bfu_h, \bfv_h) + g_p(p_h, q_h)
\end{align}
with
\begin{align}
  \label{eq:ghost-penalty-sigma}
  g_{\sigma}(\bfu_h, \bfv_h)
  :=&
  \gamma_{\sigma}
  \sum_{F\in\Fast}
  \sum_{1\leqslant j \leqslant k}
  \sigma
  h^{2j+1}\langle\jump{\nablan^j \bfu_h},\jump{\nablan^j \bfv_h}\rangle_F,
  \\
  \label{eq:ghost-penalty-nu}
  g_{\nu}(\bfu_h, \bfv_h)
  :=&
  \gamma_{\nu}
  \sum_{F\in\Fast}
  \sum_{1\leqslant j \leqslant k}
  \nu
  h^{2j-1}\langle\jump{\nablan^j \bfu_h},\jump{\nablan^j \bfv_h}\rangle_F,
\\
  \label{eq:ghost-penalty-beta}
  g_{\beta}(\bfu_h, \bfv_h)
  :=&
  \gamma_{\beta}
  \sum_{F\in\Fast}
  \sum_{0\leqslant j \leqslant k-1}
   \phi_{\beta,F}
      h^{2j+1}
  \langle
  \jump{\bfbeta\cdot \nabla\nablan^j\bfu_h}
  ,\jump{\bfbeta\cdot \nabla\nablan^j\bfv_h}
      \rangle_F,
  \\
  \label{eq:ghost-penalty-u}
  g_{u}(\bfu_h, \bfv_h)
  :=&
  \gamma_{u}
  \sum_{F\in\Fast}
  \sum_{0\leqslant j \leqslant k-1}
 \phi_{u,F} h^{2j+1}
  \langle
  \jump{\div \nablan^j \bfu_h},
  \jump{\div \nablan^j \bfv_h}
  \rangle_F,
  \\
  \label{eq:ghost-penalty-p}
  g_p(p_h, q_h) :=&
  \gamma_{p}
  \sum_{F\in\Fast}
  \sum_{1\leqslant j \leqslant k}
 \phi_{p,F} h^{2j-1}
\langle\jump{\nablan^j p_h}, \jump{\nablan^j q_h}\rangle_F,
\end{align}
where the $j$-th normal derivative $\nablan^j v$ is given by 
$\nablan^j v = \sum_{| \alpha | = j}D^{\alpha} v(\bfx)
  \bfn^{\alpha}$ for multi-index \mbox{$\alpha = (\alpha_1, \ldots,
  \alpha_d)$}, $|\alpha| = \sum_{i} \alpha_i$ and $\bfn^{\alpha} =
  n_1^{\alpha_1} n_2^{\alpha_2} \cdots n_d^{\alpha_d}$.
In contrast to the related CIP stabilizations, these terms need to
control the entire discrete polynomial, this entails all the higher-order normal derivatives contained in the adjacent elements polynomials.
However, they are only evaluated along facets in the vicinity of the boundary, i.e. $\foralls F\in \mcF_\Gamma$.
The ghost-penalty terms \eqref{eq:ghost-penalty-beta}--\eqref{eq:ghost-penalty-p} extend control over the respective instabilities to the enlarged domain $\Omega_h^{\ast}$,
whereas the reactive and viscous ghost-penalties \eqref{eq:ghost-penalty-sigma}--\eqref{eq:ghost-penalty-nu} extend control over the scaled $L^2$- and $H^1$-norms
in the following fashion:
\begin{proposition}
 \label{prop:ghost-penalty-norm-equivalence}
   Let $\Omega$, $\Oast_h$ and $\mcF_{\Gamma}$ be defined as in
   Section~\ref{ssec:cutfem-spaces}
and the ghost penalty operators be given as in \eqref{eq:ghost-penalty-sigma}--\eqref{eq:ghost-penalty-p}. Then
  for scalar functions \mbox{$p_h\in\mcQ_h$} as well as for vector-valued equivalents \mbox{$\bfu_h \in \mcV_h$} 
  the following estimates hold
  \begin{alignat}{1}
    \label{eq:ghost_penalty:norm_equivalence_L2_reaction}
    \| \sigma^{\onehalf} \bfu_h \|_{\Oast_h}^2
    &\lesssim
    \| \sigma^{\onehalf}\bfu_h \|_{\Omega}^2
    +
    g_{\sigma}(\bfu_h,\bfu_h)
    \lesssim
    \| \sigma^{\onehalf}\bfu_h \|_{\Oast_h}^2,
\\
    \label{eq:ghost_penalty:norm_equivalence_H1_viscous}
    \| \nu^{\onehalf}\nabla \bfu_h \|_{\Oast_h}^2
    &\lesssim
    \| \nu^{\onehalf} \nabla \bfu_h
    \|_{\Omega}^2
    +
    g_{\nu}(\bfu_h,\bfu_h)
    \lesssim
    \| \nu^{\onehalf} \nabla \bfu_h \|_{\Oast_h}^2.
\end{alignat}
Furthermore, let the scaling functions~$\phi_\beta, \phi_u$ and $\phi_p$ be defined as in
\eqref{eq:cip-Nitsche_scaling} and \eqref{eq:cip-s_scalings}
and let \mbox{$\bfbeta_h^0\in[\mcX_h^{\mathrm{dc},0}]^d$} be a piecewise constant approximation
to~$\bfbeta$ on~$\mcT_h$, which satisfies the approximation properties specified in \eqref{eq:beta-approximation}.
Then the following estimates hold
\begin{alignat}{1}
    \label{eq:norm_equivalence_rsb_pressure}
    \Phi \| p_h \|_{\Oast_h}^2
    &\lesssim
    \Phi \|  p_h \|_{\Omega}^2
    + g_{p}(p_h,p_h),
\\
    \label{eq:norm_equivalence_rsb_div}
    \| \phi_u^{\onehalf} \nabla \cdot \bfu_h \|_{\Oast_h}^2
    &\lesssim
    \| \phi_u^{\onehalf} \nabla \cdot \bfu_h \|_{\Omega}^2
    + g_{u}(\bfu_h,\bfu_h),
\\
    \|
    {\phi}_{\beta}^{\onehalf}
    (\bfbeta_h^0 - \bfbeta) \cdot \nabla \bfu_h
    \|_{\Oast_h}^2
    &\lesssim
    \omega_h
    \bigl(
    \| \nu^{\onehalf}\nabla \bfu_h \|_{\Oast_h}^2
    +
    \| \sigma^{\onehalf} \bfu_h \|_{\Oast_h}^2
    \bigr),
    \label{eq:norm_rsb_beta_diff_const}
\\
    \| 
    \phi_{\beta}^{\onehalf} 
    (
    \bfbeta_h^0 \cdot \nabla \bfu_h + \nabla p_h
    )
    \|_{\Oast_h}^2
    &\lesssim
    \| 
    \phi_{\beta}^{\onehalf} 
    (
    \bfbeta \cdot \nabla \bfu_h + \nabla p_h
    )
    \|_{\Omega}^2
    + 
    g_{\beta}(\bfu_h, \bfu_h)
    +
    g_{p}(p_h, p_h)
+ 
    \omega_h
    \bigl(
    \| \nu^{\onehalf}\nabla \bfu_h \|_{\Oast_h}^2
    +
    \| \sigma^{\onehalf} \bfu_h \|_{\Oast_h}^2
    \bigr),
    \label{eq:norm_equivalence_rsb_beta}
\end{alignat}
with the non-dimensional scaling functions $\omega_h$ and $\Phi$ given as in \cite{MassingSchottWall2016_CMAME_Arxiv_submit} by
\begin{align}
 \label{eq:phi_p_definition}
 \omega_h &:= \frac{ h^2 |\bfbeta|_{1,\infty,\Omega}}{\nu + \sigma h^2},
\qquad\qquad
  \Phi^{-1} 
:= \sigma C_P^2 + \norm{\bfbeta}_{0,\infty,\Omega}C_P 
  + \nu 
  + \left(\frac{\norm{\bfbeta}_{0,\infty,\Omega}C_P}{\sqrt{\nu + \sigma C_P^2}}\right)^2.
\end{align}
Note that the hidden constants in \eqref{eq:norm_equivalence_rsb_pressure}--\eqref{eq:norm_equivalence_rsb_beta} depend only on the shape-regularity and the polynomial order, but not on the mesh or the location of~$\Gamma$
within~$\mcT_h$.
\end{proposition}

\begin{proof}
For detailed proofs of these estimates, the reader is referred to the preceding work by \citet{MassingSchottWall2016_CMAME_Arxiv_submit} (cf.~Lemma 5.4, Corollary 5.5 and 5.7 and Lemma 5.8)
and the references therein.
\end{proof}

\begin{remark}
 Note that the 
 function
 $\bfbeta$ occurring in the stabilization operators $s_\beta$ \eqref{eq:cip-s_beta} and $g_\beta$ \eqref{eq:ghost-penalty-beta}
 can be replaced by
a proper continuous interpolation $\bfbeta_h$ satisfying the assumption specified in \eqref{eq:beta-approximation}, 
without changing the final results of the stability and \apriori~estimates.
When solving the incompressible Navier-Stokes equations,
$\bfbeta_h$ is the finite element approximation of $\bfu$ from 
a previous time or iteration step.
\end{remark}

\begin{remark}
\label{rem:gbeta-simple}  
As suggested in \cite{MassingSchottWall2016_CMAME_Arxiv_submit}, the convection and the incompressibility related
stabilization forms $s_\beta, s_u $ and $g_\beta, g_u$, i.e.~\eqref{eq:cip-s_beta}, \eqref{eq:cip-s_u} and \eqref{eq:ghost-penalty-beta}, \eqref{eq:ghost-penalty-u},
can be replaced by single operators
  \begin{align}
    \overline{s}_{\beta}(\bfu_h, \bfv_h)
    &:= \gamma_{\beta} \sum_{F \in \mcF_i} 
      \overline{\phi}_{\beta} h (\jump{\nablan \bfu_h}, \jump{\nablan \bfv_h})_{F},
      \label{eq:sbeta-simple-def}
    \\
    \overline{g}_{\beta}(\bfu_h, \bfv_h)
    &:= \gamma_{\beta} \sum_{F \in \Fast} \sum_{1\leqslant j\leqslant k}
    \overline{\phi}_{\beta} h^{2j -1} (\jump{\nablan^j \bfu_h}, \jump{\nablan^j \bfv_h})_{F},
  \label{eq:gbeta-simple-def}
  \end{align}
  with
  $\overline{\phi}_{\beta} = \|\bfbeta \|_{0,\infty, F}^2\phi_{\beta}$,
which simplifies the implementation of the purposed method.
\end{remark}

\subsection{Final Discrete Formulation and Norms}
\label{ssec:discrete_form_and_norms}
The full stabilized cut finite element method for the Oseen problem with generalized Navier boundary conditions
then reads: find
$U_h=(\bfu_h, p_h) \in \mcW_h $ such that $\foralls V_h=(\bfv_h, q_h)\in \mcW_h$
\begin{align}
 A_h(U_h,V_h) + S_h(U_h,V_h) + G_h(U_h,V_h) = L_h(V_h),
 \label{eq:oseen-discrete-unfitted}
\end{align}
where $A_h,L_h$ are the discrete operators including the weak constraint enforcement of the boundary conditions
(see Section~\ref{ssec:weak_imposition_bcs_oseen}, \eqref{eq:Ah-form-def}--\eqref{eq:lh-form-def_tangential_2}),
$S_h(\cdot,\cdot)$ is the CIP operator to balance instabilities in the interior of the physical domain
(see Section~\ref{ssec:stabilized-oseen}, \eqref{eq:Sh-form-def}--\eqref{eq:cip-s_p})
and $G_h(\cdot, \cdot)$ the ghost-penalty operator, which extends stability control to the boundary zone when non-boundary-fitted meshes
are used for the approximation (see Section~\ref{ssec:cutfem-oseen}, \eqref{eq:Gh-form-def}--\eqref{eq:ghost-penalty-p}).

Following the works by \citet{MassingSchottWall2016_CMAME_Arxiv_submit} on a related cut finite element method with Dirichlet boundary conditions
and the work by \citet{JuntunenStenberg2009a} on a Nitsche-type method for Robin-type constraints,
for the subsequent numerical stability and convergence analysis we introduce the following (semi-)norms
according to our cut finite element method \eqref{eq:oseen-discrete-unfitted}.
For functions $U=(\bfu,p)$ with $\bfu\in H^1(\Omega)$ and $p\in L^2(\Omega)$ we define the mesh-dependent energy norms related to
our Nitsche-type formulation
\begin{align}
\label{eq:oseen-norm-u}
  \tn \bfu  \tn^2 &:=
 \| \sigma^{\onehalf} \bfu \|_{\Omega}^2
 + \| \nu^{\onehalf} \grad \bfu\|_{\Omega}^2
 + \| (\nu/(\gamma^n h))^{\onehalf} \bfu \cdot \bfn \|^2_{\GammaGNBC} 
 + \| ( \nu/(\varepsilon + \gamma^t h) )^{\onehalf} \bfu \boldsymbol{P}^t \|^2_{\GammaGNBC} 
\nonumber
 \\
   & \quad
+ \| |\bfbeta \cdot \bfn |^{\onehalf} \bfu \|_{\Gamma}^2
+ \| (\phi_u/(\gamma^n h) )^{\onehalf} \bfu \cdot \bfn \|^2_{\GammaGNBC}
,
\\
\label{eq:oseen-norm-p}
\tn  p \tn_{\phi}^2
&:=
\| \phi^{-\onehalf} p \|_{\Omega}^2
.
\intertext{\noindent Throughout the stability analysis control over discrete functions is required on unfitted meshes $\mcT_h$,
which can be achieved thanks to the continuous interior and ghost penalty operators
for velocity and pressure.
For discrete functions $\bfu_h\in\mcV_h\subset H^1(\mcT_h)$ and $p_h\in \mcQ_h\subset L^2(\mcT_h)$
the following norms are used
}
\label{eq:oseen-norm-unfitted-u}
\tn \bfu_h \tn_h^2
&:=
\tn \bfu_h \tn^2
+
|\bfu_h|_h^2
,
\\
\label{eq:oseen-norm-unfitted-p}
\tn  p_h \tn_{h,\phi}^2
&:=
\tn  p_h \tn_{\phi}^2
+
| p_h |_h^2
,
\intertext{
with a piecewise constant scaling function~$\phi$ 
and semi-norms which are defined by the stabilization operators
}
| \bfu_h |_h^2
&:=
  s_{\beta}(\bfu_h, \bfu_h)
+ s_u(\bfu_h, \bfu_h)
+ g_{\sigma}(\bfu_h, \bfu_h)
+ g_{\nu}(\bfu_h, \bfu_h)
+ g_{\beta}(\bfu_h, \bfu_h)
+ g_u(\bfu_h, \bfu_h)
,
\\
| p_h |_h^2
&:=
  s_p(p_h,p_h)
+ g_p(p_h,p_h)
.
\end{align}
For the stability analysis we will also utilize the augmented natural energy norm for $U_h\in\mcV_h\times\mcQ_h$
\begin{align}
  \tn U_h \tn_h^2 
  := | U_h |_h^2 + \| \phi_u^{\onehalf} \nabla \cdot \bfu_h \|_{\Omega}^2
  + \dfrac{1}{1 + \omega_h} \|\phi_{\beta}^{\onehalf}(\bfbeta\cdot \nabla \bfu_h + \nabla p_h) \|_{\Omega}^2
  + \Phi\| p_h\|_{\Omega}^2,
  \label{eq:oseen-norm-up}
\end{align}
and a semi-norm which is defined as
\begin{align}
  |U_h|_h^2 := |(\bfu_h, p_h)|_h^2 = \tn \bfu_h \tn_h^2 + |p_h|^2_h = \tn \bfu_h \tn^2 + |\bfu_h|^2_h + |p_h|^2_h.
  \label{eq:Ah-semi-norm}
\end{align}
Note that it holds that $\Phi \lesssim \phi_u^{-1}$, which allows us to create a lower bound of the locally scaled pressure norms.
Moreover, $C_P$~denotes the so-called Poincar\'e constant as defined in \eqref{eq:Poincare-II} in Section~\ref{sec:interpolation-est},
which scales as the diameter of~$\Omega$.

For the \apriori~error analysis, additional control over boundary fluxes on $\Gamma$ and the divergence of the velocity is desired.
For this purpose, for sufficiently regular functions $U=(\bfu,p)\in H^2(\Omega)\times H^1(\Omega)$ we define
\begin{align}
\label{eq:oseen-norm-fluxes-u}
\tn \bfu \tn_{\ast}^2
&:=
\tn \bfu \tn^2
+
\| (\nu h)^{\onehalf} \grad \bfu \cdot \bfn \|_{\Gamma}^2
,
\\
\label{eq:oseen-norm-fluxes-p}
\tn  p \tn_{\ast,\phi}^2
&:=
\tn p \tn_{\phi}^2
+ \| \phi^{-\onehalf} h^{\onehalf} p \|_{\Gamma}^2,  
\\
\label{eq:oseen-norm-fluxes-up}
\tn U \tn_{\ast}^2
&:=
\tn \bfu \tn_{\ast}^2
+
\| \phi_u^{\onehalf} \nabla \cdot \bfu \|_{\Omega}^2
+
\tn p  \tn_{\ast,\Phi^{-1}}^2.
\end{align}
From the inverse estimate \eqref{eq:inverse-estimates} (see Section~\ref{sec:interpolation-est}) and the
norm equivalences from Proposition~\ref{prop:ghost-penalty-norm-equivalence},
discrete functions satisfy
\begin{alignat}{2}
\label{eq:oseen-norm-u-relation}
 \tn \bfu_h  \tn_{\ast}^2
&\lesssim
\tn \bfu_h \tn^2
+
g_{\nu}(\bfu_h,\bfu_h)
\lesssim
\tn \bfu_h \tn_h^2
\quad &&\foralls \bfu_h \in \mcV_h,
\\
\label{eq:oseen-norm-p-relation}
 \tn p_h  \tn_{\ast,\Phi^{-1}}^2
&\lesssim
\tn p_h \tn_{\Phi^{-1}}^2
+
g_p(p_h,p_h)
\lesssim
\tn p_h \tn_{h,\Phi^{-1}}^2
\quad &&\foralls p_h \in \mcQ_h, \\
\label{eq:oseen-norm-U-relation}
\tn U_h  \tn_{\ast}^2
&\lesssim
\tn U_h \tn_{h}^2
&&\foralls U_h \in \mcW_h
.
\end{alignat}

\section{Preliminary Estimates}
\label{sec:interpolation-est}
In this section, we collect important and useful estimates, which will be used frequently throughout the stability and \apriori~error analysis
of the proposed cut finite element method \eqref{eq:oseen-discrete-unfitted}
in Sections~\ref{sec:stability-properties} and~\ref{sec:apriori-analysis}.
First, we collect trace inequalities and inverse estimates and comment on the
Korn-type inequality for the strain-rate tensor in combination with Nitsche boundary terms to enforce a generalized Navier boundary condition.
Finally, we introduce suitable interpolation error estimates, which will be used to establish \apriori~error estimates.

\subsection{Useful Inequalities and Estimates}
For discrete functions $v_h\in \mcX_h$ the following well-known generalized inverse and trace inequalities
hold for elements~$T$ which are arbitrarily intersected by the boundary~$\Gamma$:
\begin{alignat}{3}
  \norm{D^j v_h}_{T}
 + h^{\onehalf} \norm{\partial_{\bfn}^j v_h}_{\partial T}
 + h^{\onehalf} \norm{\partial_{\bfn}^j v_h}_{\Gamma \cap T}
  & \lesssim h^{i-j} \norm{D^i v_h}_{T} & & \quad \foralls T \in
  \mcT_h, \quad 0\leqslant i\leqslant j.
  \label{eq:inverse-estimates}
\end{alignat}
%
For functions $v \in H^1(\Oast_h)$, the following trace inequalities are valid
  \begin{align}
    \label{eq:trace-inequality}
    \norm{v}_{\partial T}
  + \norm{v}_{\Gamma \cap T}
    &\lesssim
    h^{-1/2} \norm{v}_{T} +
    h^{1/2}  \norm{\nabla v}_{T}
    \quad \foralls T \in \mcT_h,
  \end{align}
as shown in~\cite{HansboHansbo2002,Burman2016}.
  Finally, assuming that $\meas(\GammaGNBC^\varepsilon)>0$, we recall the generalized Poincar\'e inequality 
for $[H^1(\Omega)]^d$ functions with non-vanishing boundary trace from \cite{Brenner2003}
  \begin{alignat}{2}
    \| \bfv \|_{0,\Omega}
    &\lesssim C_P (\| \nabla \bfv \|_{0,\Omega} + \|\bfv\|_{\GammaGNBC^{\varepsilon}}) \quad \foralls \bfv_h\in [H^1(\Omega)]^d.
    \label{eq:Poincare-II}
  \end{alignat}

In the following, a generalized Korn-type inequality is presented and it is shown how to control
the kernel of the strain-rate-deformation tensor $\bfepsilon(\cdot)$ with the help of Nitsche-type penalty terms.
\begin{proposition}\label{prop:norm_kernel_strain_rate}
  Let a semi-norm on $[H^1(\Omega)]^d$ be defined as
$\|\bfu\|_{\GammaGNBC^{\varepsilon}}^2:= \int_{\GammaGNBC^\varepsilon}{ \bfu^2} \ds$,
where $\|\varepsilon\|_{\infty,\GammaGNBC^\varepsilon} \leqslant c_\varepsilon < \infty$.
Let us assume that $\meas(\GammaGNBC^\varepsilon)>0$, then $\|\cdot\|_{\GammaGNBC^{\varepsilon}}$ defines a norm on the space of rigid body motions
\begin{equation}\label{eq:kernel_strain_rate_tensor_H1}
 RM(\Omega) := \{ \bfu\in [H^1(\Omega)]^d ~|~ \bfu(\bfx) := \bfc + \boldsymbol{W} \bfx,~\bfc\in\R^d, \boldsymbol{W}\in S_d~\foralls \bfx\in\Omega\},
\end{equation}
where $S_d$ is the space of anti-symmetric $d\times d$ matrices, i.e. $\boldsymbol{W}=-\boldsymbol{W}^T$.
Note, that $RM(\Omega)$ is the kernel of the symmetric strain-rate-deformation tensor $\bfepsilon(\cdot)$ in $[H^1(\Omega)]^d$
with $d\in\{2,3\}$.
\end{proposition}
\begin{proof}
For a more explicit description of why
$RM(\Omega)$ 
is the kernel of the tensor $\bfepsilon(\cdot)$ in $[H^1(\Omega)]^d$,
the reader is referred to, e.g. \cite{Mardal2005, BrennerScott2008}.
The claim regarding the norm property follows directly from the fact that $\meas(\GammaGNBC^\varepsilon) >0$ and from the uniqueness of the trivial solution
$(\bfc,\boldsymbol{W})=(\bfzero,\bfzero)\in\R^d\times S_d$
of the linear system $\bfc + \boldsymbol{W} \bfx = \bfzero$ on~$\GammaGNBC^\varepsilon$.
\end{proof}

\begin{theorem}[Korn-type inequality]
\label{thm:Korn_type_inequality}
Let $\meas(\GammaGNBC^\varepsilon)>0$ and $\|\cdot\|_{\GammaGNBC^{\varepsilon}}$ be as defined in Proposition~\ref{prop:norm_kernel_strain_rate},
then there exists a constant $C_K$ such that
 \begin{align}\label{eq:Korn_type_inequality}
  \| \nabla \bfu \|_{\Omega} \leqslant \| \bfu \|_{1,\Omega} \leqslant C_K( \| \bfepsilon(\bfu) \|_{\Omega} + \|\bfu\|_{\GammaGNBC^{\varepsilon}}) \quad \foralls \bfu\in [H^1(\Omega)]^d.
 \end{align}
\end{theorem}

\begin{proof}
 The proof follows the technique proposed in \cite{Boiveau2014} and \cite{BrennerScott2008} for Korn-type inequalities
and is therefore only sketched in the following.
Let $[\hat{H}^1(\Omega)]^d:=\{\bfv \in [H^1(\Omega)]^d:|\int_\Omega \rot \bfv \dx|=0 \wedge |\int_\Omega\bfv\dx|=0 \}$, then it holds $[H^1(\Omega)]^d = [\hat{H}^1(\Omega)]^d \oplus RM(\Omega)$
such that for each $\bfu\in [H^1(\Omega)]^d$ there exists a unique pair $(\bfz,\bfw)\in [\hat{H}^1(\Omega)]^d \oplus RM(\Omega)$
with $\bfu=\bfz+\bfw$ satisfying $\|\bfz\|_{1,\Omega}+\|\bfw\|_{1,\Omega} \lesssim \|\bfu\|_{1,\Omega}$ (open mapping theorem by \citet{Lax2002functional}).
Following \cite{Boiveau2014} the claim can be proven via contradiction. Assuming non-boundedness in \eqref{eq:Korn_type_inequality}, there exists a sequence
$\{\bfu_n\}\subseteq [H^1(\Omega)]^d$ with $\|\bfu_n\|_{1,\Omega}=1$ and $\|\bfepsilon(\bfu_n)\|_{\Omega} + \|\bfu_n\|_{\GammaGNBC^{\varepsilon}}<\frac{1}{n}$.
Splitting $\bfu_n=\bfz_n+\bfw_n \in [\hat{H}^1(\Omega)]^d\oplus RM(\Omega)$ for each $n$ it holds that
$\|\bfepsilon(\bfz_n)\|_\Omega = \|\bfepsilon(\bfu_n)\|_\Omega < \frac{1}{n}$, since $\bfepsilon(\bfw_n) = \bfzero$.
Applying the second Korn's inequality, see e.g. \cite{Brenner2004}, we obtain
\begin{equation}
 \|\bfz_n\|_{1,\Omega} \lesssim \|\bfepsilon(\bfz_n)\|_{\Omega} + |\int_\Omega \rot \bfz_n \dx| + |\int_\Omega \bfz_n \dx| = \|\bfepsilon(\bfz_n)\|_{\Omega} \rightarrow 0 \text{ in } [H^1(\Omega)]^d.
\end{equation}
Since $\bfz_n+\bfw_n$ is a bounded sequence and $RM(\Omega)$ is finite dimensional, $\{\bfw_n\}$ is a bounded sequence and as such there exists a convergent subsequence
for which holds $\bfu=\lim_{n_k\to\infty}{\bfw_{n_k}}$ with $\bfu\in RM(\Omega)$ since $\lim_{n_k\to\infty}\bfz_{n_k}=\bfzero$ in $[H^1(\Omega)]^d$.
Note that by assumption $\|\bfu\|_{\GammaGNBC^{\varepsilon}}=0$ and $\|\bfu\|_{1,\Omega}=1$.
Since $\|\cdot\|_{\GammaGNBC^{\varepsilon}}$ defines a norm on $RM(\Omega)$ whenever $\meas(\GammaGNBC^{\varepsilon})>0$
(see Proposition~\ref{prop:norm_kernel_strain_rate}), it follows that $\bfu=\bfzero\in RM(\Omega)$ and therefore
$\bfu=\bfzero\in [H^1(\Omega)]^d$, which contradicts $\|\bfu\|_{1,\Omega}=1$.
As a result, the assumption of unboundedness of \eqref{eq:Korn_type_inequality} is refuted, which proves the claim.
\end{proof}

The subsequent corollary states that the Nitsche penalty terms occurring in our Nitsche-type cut finite element formulation
\eqref{eq:oseen-discrete-unfitted} are sufficient to
control the rigid body motions in $[H^1(\Omega)]^d$, which remain undetermined by the strain-rate-deformation tensor $\bfepsilon(\cdot)$.

\begin{corollary}
\label{cor:modified_korns_ineq}
Let the domain be bounded, i.e. $\diam(\Omega)<\infty$, the Nitsche penalty parameters $\gamma^t,\gamma^n < \infty$ and
assume that normal and tangential velocities are constrained by Nitsche-type penalty terms
on a non-vanishing part of the boundary~$\GammaGNBC^\varepsilon$, i.e. $\meas(\GammaGNBC^\varepsilon)>0$, then
  \begin{align}
    \| \nabla \bfu_h \|_{\Omega}^2
    \leqslant
   \|\bfu_h \|_{1,\Omega}^2
    \lesssim
  \| \bfepsilon(\bfu_h) \|_{\Omega}^2
  + \| \left( {\gamma^n h} \right)^{-\onehalf} \bfu_h \cdot \bfn \|_{\GammaGNBC}^2
  + \| \left( {\varepsilon + \gamma^t h} \right)^{-\onehalf} \bfu_h \boldsymbol{P}^t \|_{\GammaGNBC^{\varepsilon}}^2.
  \end{align}
\end{corollary}
\begin{proof}
The proof follows immediately from applying the Korn-type inequality deduced in Theorem~\ref{thm:Korn_type_inequality}, the definition
of $\GammaGNBC^{\epsilon}$ and the specified assumptions in the corollary
\begin{align}
\label{eq:modified_korns_ineq}
  & \| \bfepsilon(\bfu_h) \|_{\Omega}^2
  + \| \left( {\gamma^n h} \right)^{-\onehalf} \bfu_h \cdot \bfn \|_{\GammaGNBC}^2
  + \| \left({\varepsilon + \gamma^t h} \right)^{-\onehalf} \bfu_h \boldsymbol{P}^t \|_{\GammaGNBC^{\varepsilon}}^2 \nonumber
\\
& \quad
  \gtrsim
  \| \bfepsilon(\bfu_h) \|_{\Omega}^2
  + (\gamma^n \diam(\Omega))^{-1} \|\bfu_h \cdot \bfn \|_{\Gamma}^2
  + ( \|\varepsilon\|_{\infty,\GammaGNBC^{\varepsilon}} + \gamma^t \diam(\Omega) )^{-1} \| \bfu_h \boldsymbol{P}^t \|_{\GammaGNBC^{\varepsilon}}^2
\\
& \quad
  \gtrsim
  \| \bfepsilon(\bfu_h) \|_{\Omega}^2
  + \|  \bfu_h  \|_{\GammaGNBC^{\varepsilon}}^2
\\
& \quad
  \gtrsim
  \|\bfu_h \|_{1,\Omega}^2
  \geqslant \|\nabla \bfu_h \|_{\Omega}^2.
\end{align}
\end{proof}
\enlargethispage{0.2cm}

\subsection{Interpolation Operators}
\label{ssec:interpolation}
Since the finite element approximation space is defined on the enlarged domain $\Omega_h^{\ast}$,
we first comment on the construction of an appropriate interpolation operator
$L^2(\Omega) \to \mcX_h$.
From \cite{Stein1970} it is well known that for the Sobolev spaces $W^{m,q}(\Omega)$,
$0\leqslant m < \infty$, $1 \leqslant q \leqslant \infty$,
a linear extension operator can be defined
\begin{equation}
  \label{eq:extension-operator-boundedness}
  E: W^{m,q}({\Omega}) \rightarrow W^{m,q}(\Oast) \qquad \text{with} \qquad  \norm{E v}_{m,q,\Oast} \lesssim \norm{v}_{m,q,\Omega},
\end{equation}
and we write $v^\ast := Ev$.
Following the analysis provided by \citet{MassingSchottWall2016_CMAME_Arxiv_submit}, 
let $\pi_h$ denote the \emph{Cl\'ement operator}, see for instance \cite{ErnGuermond2004},
then for $u \in H^s(\Omega)$ we define its ``fictitious domain'' extension
$\pi_h^\ast: H^s(\Omega) \to \mcX_h$ by $\pi_h^\ast u := \pi_h(u^\ast)$.
For some fixed Lipschitz-domain $\Oast$ satisfying $\Oast_h \subseteq \Oast$ for $h \lesssim 1$,
then the following interpolation estimates hold for functions $v \in H^r(\Oast)$
and its fictitious domain variant
\begin{alignat}{3}
\| v - \pi_h v \|_{s,T} 
& \lesssim
  h^{t-s}| v |_{t,\omega(T)},
&&\quad 0\leqslant s \leqslant t \leqslant m \quad &\foralls T\in
  \mcT_h,
  \label{eq:interpest0}
\\
  \| v^\ast - \pi_h^{\ast} v \|_{s,\mcT_h} 
  & \lesssim
  h^{t-s}\| v \|_{t,\Omega},
  & &\quad 0\leqslant s \leqslant t \leqslant m,
  \label{eq:interpest0-ast}
\end{alignat}
owing to the boundedness of the extension operator~\eqref{eq:extension-operator-boundedness},
with $s, t \in \NN$, $m=\min\{r,k+1\}$, $k$ the interpolation order of $\mcX_h$
and $\omega(T)$ the set of elements in $\mcT_h$ sharing at least
one vertex with $T$.
Throughout the analysis we write $\boldsymbol{\pi}_h^{\ast}$
for the Cl\'ement interpolant of vector-valued functions~$\bfv$ and $\Pi_h^{\ast}$
for functions in a product space.

\section{Continuity and Stability Estimates}
\label{sec:stability-properties}

In this section, we establish stability properties of the proposed cut finite element method \eqref{eq:oseen-discrete-unfitted}.
Our presentation is closely related to the analysis presented in the preceding work by \citet{MassingSchottWall2016_CMAME_Arxiv_submit}
and mainly differs in the analysis of the boundary terms related to the weak enforcement of general Navier boundary conditions.
The major stability result relies on a modified coercivity estimate for the total bilinear form $A_h + S_h + G_h$.
Since fluid instabilities and the extension of the finite element method to unfitted meshes can be treated almost equivalently to \citet{MassingSchottWall2016_CMAME_Arxiv_submit},
we only recall the required statements from the latter publication without presenting proofs.

After deriving continuity estimates for parts of the stabilized bilinear form in Lemma~\ref{lem:continuity}, we start by
proving a coercivity estimate for our stabilized formulation w.r.t. the semi-norm $|U_h|_h$ on $\mcW_h$ in Lemma~\ref{lem:coercivity_ah}.
We then recall how to recover required control over three additional (semi-)norms
\begin{equation}
\| \phi_{u}^{\onehalf} \nabla\cdot \bfu_h \|_{\Omega}, \quad
\| \phi_{\beta}^{\onehalf}(\bfbeta \cdot \nabla \bfu_h + \nabla p_h)\|_{\Omega} \quad \text{and} \quad
\Phi^{\onehalf}\|p_h\|_{\Omega},
\end{equation}
in Lemmas~\ref{lem:divergence-stability},~\ref{lem:convect-pressure-stab} and \ref{lem:pressure-stability}
with the help of CIP and ghost-penalty stabilizations as introduced in Sections~\ref{ssec:stabilized-oseen} and~\ref{ssec:cutfem-oseen}
using techniques provided in \cite{Burman2007a,MassingSchottWall2016_CMAME_Arxiv_submit}.
By combining the above mentioned estimates, a global inf-sup stability on $\mcW_h$ is shown in Theorem~\ref{thm:inf-sup_condition_total}.

\begin{lemma}[Continuity Estimates]
\label{lem:continuity}
For an arbitrary choice of functions $\bfu,\bfv\in [H^2(\Omega)]^d$, $p\in H^1(\Omega)$ and discrete functions $\bfu_h,\bfv_h \in \mcV_h$, $p_h\in\mcQ_h$ the following continuity estimates
hold provided that $\gamma^n,\gamma^t<\infty$
\begin{alignat}{2}
\label{eq:continuity_ah_1}
 a_h(\bfu+\bfu_h, \bfv+\bfv_h) - (\bfbeta\cdot\nabla (\bfu+\bfu_h), (\bfv+\bfv_h))_\Omega &\lesssim \tn \bfu+\bfu_h\tn_{\ast  } \tn \bfv+\bfv_h\tn_{\ast}
 ,
\\
\label{eq:continuity_ah_2}
 a_h(\bfu+\bfu_h, \bfv+\bfv_h) + ((\bfu+\bfu_h), \bfbeta\cdot\nabla (\bfv+\bfv_h))_\Omega &\lesssim \tn \bfu+\bfu_h\tn_{\ast  } \tn \bfv+\bfv_h\tn_{\ast}
 ,
\\
\label{eq:continuity_ah_3}
 a_h(\bfu_h, \bfv_h) - (\bfbeta\cdot\nabla \bfu_h, \bfv_h)_{\Omega}&\gtrsim - \tn \bfu_h\tn_h \tn \bfv_h \tn_h
 ,
\\
\label{eq:continuity_ah_4}
 a_h(\bfu_h, \bfv_h) + (\bfu_h, \bfbeta\cdot\nabla \bfv_h)_{\Omega}  &\gtrsim - \tn \bfu_h\tn_h \tn \bfv_h \tn_h
 ,
\\
\label{eq:continuity_bh_1}
 |b_h(p+p_h,\bfv+\bfv_h)|      &\lesssim \tn p+p_h \tn_{\ast,\phi_u} (\tn \bfv+\bfv_h \tn_{\ast} + \| \phi_u^{\onehalf}\nabla\cdot(\bfv+\bfv_h )\|_{\Omega})
 ,
\\
\label{eq:continuity_bh_3}
 |b_h(p+p_h,\bfv_h)|      &\lesssim \tn p+p_h \tn_{\ast,\phi_u} (\tn \bfv_h \tn_{h} + \|\phi_u^{\onehalf} \nabla \cdot \bfv_h\|_{\Omega})
 .
\end{alignat}
\end{lemma}
\begin{proof}
We start by proving the continuity estimate for $a_h$ with neglected advective bulk term as in \eqref{eq:continuity_ah_1}.
The proof is straightforward as it follows directly from
applying the Cauchy-Schwarz inequality for each term.
For the bulk terms we have
\begin{align}
   &|(\sigma (\bfu + \bfu_h), (\bfv + \bfv_h))_{\Omega}| + |(2\nu \bfepsilon(\bfu + \bfu_h), \bfepsilon(\bfv + \bfv_h))_{\Omega}|
\nonumber\\
   &\qquad\lesssim
   (\| \sigma^{\onehalf}(\bfu + \bfu_h) \|_\Omega + \| \nu^{\onehalf} \nabla(\bfu + \bfu_h)\|_\Omega )
   (\| \sigma^{\onehalf}(\bfv + \bfv_h) \|_\Omega + \| \nu^{\onehalf} \nabla(\bfv + \bfv_h)\|_\Omega )
   \\
   & \qquad\lesssim
    ||| \bfu+\bfu_h |||_{\ast}  ||| \bfv+\bfv_h |||_{\ast}.
\end{align}
It remains to estimate the boundary terms. By analogy, the estimate for all symmetric viscous and advective inflow Nitsche penalty terms follows directly
from applying Cauchy-Schwarz inequality and the definition of $||| \cdot |||_{\ast}$.
The non-symmetric boundary terms in $a_h$ can be estimated as
\begin{align}
&| \langle ( \frac{\varepsilon}{\varepsilon + \gamma^t h}  - 1 )  (2 \nu \bfepsilon(\bfu+\bfu_h)\bfn )  \boldsymbol{P}^t  , (\bfv+\bfv_h) \rangle_{\GammaGNBC} |
   \nonumber
   \\
   &\qquad\lesssim
   2(\gamma^t)^{\onehalf} || (\frac{\gamma^t h}{\varepsilon + \gamma^t h})^{\onehalf} (\nu h)^{\onehalf} (\nabla(\bfu+\bfu_h)\cdot\bfn ) ||_{\GammaGNBC}|| (\frac{\nu}{\varepsilon + \gamma^t h})^{\onehalf} (\bfv+\bfv_h)\boldsymbol{P}^t ||_{\GammaGNBC}
   \\
   &\qquad\lesssim
   \tn \bfu+\bfu_h \tn_{\ast} \tn \bfv+\bfv_h \tn_{\ast},
\intertext{which holds if $\gamma^t < \infty$, since $\frac{\gamma^t h}{\varepsilon + \gamma^t h}\leqslant 1$.
Similarly, utilizing that $\frac{\gamma^t \varepsilon}{\varepsilon + \gamma^t h}\leqslant \gamma^t$ we have}
  &|\zeta_u \langle \frac{ \gamma^t h }{\varepsilon + \gamma^t h} \varepsilon(2 \nu \bfepsilon(\bfu+\bfu_h) \bfn)  \boldsymbol{P}^t
  , 2 \bfepsilon(\bfv+\bfv_h) \bfn \rangle_{\GammaGNBC}|
  \nonumber\\
   &\qquad\lesssim
   4(\frac{ \gamma^t \varepsilon}{\varepsilon + \gamma^t h})|| (\nu h)^{\onehalf} (\nabla(\bfu+\bfu_h)\cdot\bfn ) ||_{\GammaGNBC}
   || (\nu h)^{\onehalf} (\nabla(\bfv+\bfv_h)\cdot\bfn)  ||_{\GammaGNBC}   
   \\
   &\qquad\lesssim
   \gamma^t \tn \bfu+\bfu_h \tn_{\ast} \tn \bfv+\bfv_h \tn_{\ast}.
\end{align}
All remaining boundary terms in $a_h$ can be estimated analogously, which yields
\begin{align}
 | a_h(\bfu+\bfu_h, \bfv+\bfv_h)- (\bfbeta\cdot\nabla (\bfu+\bfu_h), (\bfv+\bfv_h))_\Omega| \lesssim \tn \bfu+\bfu_h\tn_{\ast  } \tn \bfv+\bfv_h\tn_{\ast}
\label{eq:continuity_tmp_1}
\end{align}
and as a consequence \eqref{eq:continuity_ah_1}.
Estimate \eqref{eq:continuity_ah_2} can easily be deduced by integrating the advective bulk term by parts
and applying the Cauchy-Schwarz inequality for the resulting boundary term
\begin{align}
 & |a_h(\bfu+\bfu_h, \bfv+\bfv_h) + ((\bfu+\bfu_h), \bfbeta\cdot\nabla (\bfv+\bfv_h))_\Omega|
\nonumber
\\
&\qquad=
 |a_h(\bfu+\bfu_h, \bfv+\bfv_h) - (\bfbeta\cdot\nabla (\bfu+\bfu_h), (\bfv+\bfv_h))_\Omega
+ \langle (\bfbeta\cdot\bfn) (\bfu+\bfu_h), (\bfv+\bfv_h) \rangle_\Gamma |\\
&\qquad\lesssim
 |a_h(\bfu+\bfu_h, \bfv+\bfv_h)
- (\bfbeta\cdot\nabla (\bfu+\bfu_h), (\bfv+\bfv_h))_\Omega|
+ \tn \bfu+\bfu_h \tn_{\ast} \tn \bfv+\bfv_h \tn_{\ast}.
\label{eq:continuity_tmp_2}
\end{align}
Setting $\bfv=\bfu=\bfzero$ in \eqref{eq:continuity_tmp_1}--\eqref{eq:continuity_tmp_2} and applying relation \eqref{eq:oseen-norm-u-relation}, i.e.
$\tn \bfv_h \tn_{\ast} \lesssim \tn \bfv_h \tn_{h}$, gives
\begin{align}
 | a_h(\bfu_h, \bfv_h) - (\bfbeta\cdot\nabla\bfu_h, \bfv_h)_\Omega| &\lesssim \tn \bfu_h \tn_{h} \tn \bfv_h \tn_{h} , \\
 | a_h(\bfu_h, \bfv_h) + (\bfu_h, \bfbeta\cdot\nabla\bfv_h)_\Omega| &\lesssim \tn \bfu_h \tn_{h} \tn \bfv_h \tn_{h} ,
\end{align}
which leads to \eqref{eq:continuity_ah_3} and \eqref{eq:continuity_ah_4}.

Similar estimates can be established for $b_h$ by applying the Cauchy-Schwarz inequality
\begin{align}
 |b_h(p+p_h,\bfv+\bfv_h)|
&\leqslant
|(p+p_h, \nabla\cdot(\bfv+\bfv_h))_{\Omega}|
+
|\langle p+p_h, (\bfv+\bfv_h)\cdot\bfn\rangle_{\GammaGNBC}|
\\
&\lesssim
(\| \phi_u^{-\onehalf} (p+p_h) \|_{\Omega} + \| \phi_u^{-\onehalf} h^{\onehalf}(p+p_h) \|_{\GammaGNBC})
(\|\phi_u^{\onehalf}\nabla\cdot(\bfv+\bfv_h)\|_{\Omega} + \| (\phi_u/h)^{\onehalf}(\bfv+\bfv_h)\cdot\bfn\|_{\GammaGNBC} )
\\
&\lesssim
\tn p+p_h \tn_{\ast,\phi_u}
(\tn \bfv+\bfv_h \tn_{\ast} + \|\phi_u^{\onehalf}\nabla\cdot(\bfv+\bfv_h)\|_{\Omega} ),
\end{align}
which proves \eqref{eq:continuity_bh_1}.
Choosing $\bfv=\bfzero$ and using \eqref{eq:oseen-norm-u-relation} yields estimate \eqref{eq:continuity_bh_3}.
\end{proof}

The next lemma shows how the Nitsche-type boundary terms,
which have been consistently added to enforce the general Navier boundary conditions,
guarantee stability of the bilinear form with respect to a semi-norm on $\mcW_h$
and how ghost-penalty terms extend stability to the enlarged domain $\Omega_h^{\ast}$.

\begin{lemma}[Coercivity estimate]
The stabilized cut finite element formulation is coercive, i.e.
  \begin{equation}
    | U_h |_h^2
   \lesssim
   A_h(U_h,U_h) + S_h(U_h,U_h) + G_h(U_h,U_h) \quad \foralls U_h = (\bfu_h,p_h) \in \mcW_h
   \label{eq:coercivity_ah}
  \end{equation}
  holds whenever the CIP and ghost penalty stability parameters $\gamma_\nu,\gamma_\sigma,\gamma_\beta,\gamma_u,\gamma_p$ are chosen
  large enough.
  For $\zeta_u=1$, the Nitsche penalty parameters need to be chosen small enough, i.e. $0<\gamma^n,\gamma^t \leqslant C < \infty$,
where $C$ depends on the shape and polynomial order of the elements (see inverse estimate \eqref{eq:inverse-estimates}), however, it is independent of how the boundary intersects the element.
  By contrast, for $\zeta_u=-1$, the bilinear form is coercive for any choice $0<\gamma^n,\gamma^t < \infty$.
Note that $\meas{(\Gamma^{\varepsilon})}>0$ is assumed.
  \label{lem:coercivity_ah}
\end{lemma}

\begin{proof}
Starting from the definition of $A_h$, see \eqref{eq:Ah-form-def}, we have
  \begin{align}
   A_h(U_h,U_h)
  &= \normLtwo{\sigmahalf \bfu_h}{\Omega}^2
     + \normLtwo{(2\nu)^\onehalf\bfepsilon(\bfu_h)}{\Omega}^2
     + (\bfbeta\cdot\nabla\bfu_h, \bfu_h)_{\Omega}
     - \langle (\bfbeta\cdot\bfn)\bfu_h,  \bfu_h\rangle_{\GammaIn}
     + \| \left( \frac{ \phi_u}{\gamma^n h}  \right)^{\onehalf} \bfu_h \cdot \bfn \|^2_{\GammaGNBC} 
  \nonumber
\\
  &  \quad
       + \| \left( \frac{\nu}{\gamma^n h} \right)^{\onehalf} \bfu_h \cdot \bfn \|^2_{\GammaGNBC}
     - (1+\zeta_u) \langle (2 \nu \bfepsilon(\bfu_h)\bfn) \cdot \bfn, \bfu_h \cdot \bfn \rangle_{\GammaGNBC}
   \nonumber
\\
  &  \quad
     + \| \left( \frac{\nu}{\varepsilon + \gamma^t h} \right)^{\onehalf} \bfu_h \boldsymbol{P}^t \|^2_{\GammaGNBC} 
     - (1+\zeta_u) \langle \frac{\gamma^t h }{\varepsilon + \gamma^t h} 2 \nu \bfepsilon(\bfu_h)\bfn, \bfu_h \boldsymbol{P}^t \rangle_{\GammaGNBC}
     - \zeta_u \| \left( \frac{4 \gamma^t h  \varepsilon}{\varepsilon + \gamma^t h}  \right)^{\onehalf} \nu^{\onehalf} (  \bfepsilon(\bfu_h) \bfn )\boldsymbol{P}^t \|^2_{\GammaGNBC}  . 
  \label{eq:coercivity_ah_Ah}
  \end{align}
  Integration by parts for the advective term together with continuity of $\bfbeta$ yields
  \begin{align}
   (\bfbeta\cdot\nabla\bfu_h, \bfu_h)_{\Omega} = \frac{1}{2}\langle
    (\bfbeta\cdot\bfn)\bfu_h,\bfu_h \rangle_{\Gamma} -
    \frac{1}{2}((\div\bfbeta)\bfu_h, \bfu_h)_{\Omega}.
  \end{align}
  Using the assumption $\div\bfbeta = 0$, with the help of the advective inflow control imposed on
$\Gamma_{\mathrm{in}} \subseteq  \Gamma_D \subseteq \GammaGNBC$, the boundary control can be recovered at the entire boundary $\Gamma$ as
  \begin{align}\label{eq:coercivity_adv}
   (\bfbeta\cdot\nabla\bfu_h, \bfu_h)_{\Omega}
       - \langle (\bfbeta\cdot\bfn)\bfu_h,  \bfu_h\rangle_{\GammaIn}
   &= \frac{1}{2}\langle(\bfbeta\cdot\bfn)\bfu_h,\bfu_h\rangle_{\Gamma}
       - \langle (\bfbeta\cdot\bfn)\bfu_h,  \bfu_h\rangle_{\GammaIn}
    = \frac{1}{2}\normLtwo{|\bfbeta\cdot\bfn|^\onehalf \bfu_h}{\Gamma}^2.
  \end{align}
For the viscous boundary terms, the following estimates can be established.
  By applying a $\delta$-scaled Cauchy-Schwarz inequality, a trace inequality \eqref{eq:inverse-estimates}
and the viscous norm equivalence from Proposition~\ref{prop:ghost-penalty-norm-equivalence} for cut meshes, the following estimate can be established
 \begin{align}
 &(1+\zeta_u)  |\langle (2 \nu \bfepsilon(\bfu_h)\bfn) \cdot \bfn, \bfu_h \cdot \bfn \rangle_{\GammaGNBC}|
   \nonumber
\\
   & \quad\lesssim \frac{4}{2\delta_1}
   (1+\zeta_u)\gamma^n
   \| (\nu h )^\onehalf\bfepsilon(\bfu_h)\bfn
   \|_{\GammaGNBC}^2  
   + \frac{\delta_1}{2}(1+\zeta_u)\| (\nu/(\gamma^n h))^\onehalf \bfu_h \cdot \bfn \|_{\GammaGNBC}^2 
\\
     &\quad\lesssim \frac{4}{2\delta_1} (1+\zeta_u)\gamma^n
     \bigl(
     \| \nuhalf\grad\bfu_h\|_{\Omega}^2  + 
     g_{\nu}(\bfu_h, \bfu_h) 
     \bigr)
     +
     \frac{\delta_1}{2} (1+\zeta_u)\| (\nu/(\gamma^n h))^\onehalf \bfu_h  \cdot \bfn\|_{\GammaGNBC}^2.
\label{eq:stress-boundary-estimate-I}
  \end{align}
By analogy, the following boundary term from \eqref{eq:coercivity_ah_Ah} can be estimated as
\begin{align}
 &(1+\zeta_u) |\langle   \frac{\gamma^t h }{\varepsilon + \gamma^t h} 2 \nu \bfepsilon(\bfu_h)\bfn, \bfu_h \boldsymbol{P}^t \rangle_{\GammaGNBC}|
 \nonumber
\\
     &\quad\lesssim\frac{4}{2\delta_2}  (1+\zeta_u)\gamma^t(\gamma^t h)/(\varepsilon + \gamma^t h)
     \bigl(
     \| \nuhalf\grad\bfu_h\|_{\Omega}^2  + 
     g_{\nu}(\bfu_h, \bfu_h) 
     \bigr)
     +
     \frac{\delta_2}{2} (1+\zeta_u)\|  (\nu/(\varepsilon + \gamma^t h))^\onehalf \bfu_h  \boldsymbol{P}^t\|_{\GammaGNBC}^2.
\label{eq:stress-boundary-estimate-II}
\end{align}
The last term in \eqref{eq:coercivity_ah_Ah} can be bounded by applying a trace inequality, such that
\begin{align}
 \zeta_u \| \left( \frac{4  \gamma^t h  \varepsilon}{\varepsilon + \gamma^t h}  \right)^{\onehalf} \nu^{\onehalf}( \bfepsilon(\bfu_h) \bfn )\boldsymbol{P}^t \|^2_{\GammaGNBC}  
  \lesssim
  4 \zeta_u \gamma^t \left( \frac{\varepsilon}{\varepsilon + \gamma^t h}  \right)
     \bigl(
     \| \nuhalf\grad\bfu_h\|_{\Omega}^2  + 
     g_{\nu}(\bfu_h, \bfu_h) 
     \bigr).
\label{eq:stress-boundary-estimate-III}
\end{align}
From the Corollary~\ref{cor:modified_korns_ineq} a part of the remaining viscous terms in \eqref{eq:coercivity_ah_Ah} can be estimated as
  \begin{align}
\| \nu^\onehalf\bfepsilon(\bfu_h) \|_{\Omega}^2
  + \frac{1}{2} \| \left( \frac{\nu}{\gamma^n h} \right)^\onehalf \bfu_h \cdot \bfn \|_{\GammaGNBC}^2
  + \frac{1}{2} \| \left( \frac{\nu}{\varepsilon + \gamma^t h} \right)^\onehalf \bfu_h \boldsymbol{P}^t \|_{\GammaGNBC^{\varepsilon}}^2
  \gtrsim
  \| \nu^\onehalf \nabla \bfu_h \|_{\Omega}^2,
  \end{align}
and by combining the previous inequalities \eqref{eq:coercivity_ah_Ah}, \eqref{eq:coercivity_adv}, \eqref{eq:stress-boundary-estimate-I},
\eqref{eq:stress-boundary-estimate-II}, \eqref{eq:stress-boundary-estimate-III} with \eqref{eq:modified_korns_ineq},
the following estimate can be done
  \begin{align}
&A_h(U_h,U_h) + g_\nu(\bfu_h,\bfu_h)
\nonumber\\
& \quad \gtrsim
  \normLtwo{\sigmahalf \bfu_h}{\Omega}^2 +
  \frac{1}{2}\normLtwo{|\bfbeta\cdot\bfn|^\onehalf \bfu_h}{\Gamma}^2
\nonumber\\
  & \qquad
  +\left(  1  - \frac{4 }{2\delta_1} \left( 1 +\zeta_u \right)\gamma^n  - \frac{4}{2 \delta_2} \left(1+\zeta_u\right) \gamma^t\left( \frac{  \gamma^t h }{\varepsilon + \gamma^t h} \right)   -  4 \zeta_u \gamma^t \left( \frac{  \varepsilon}{\varepsilon + \gamma^t h}  \right) \right)
        \left(  \normLtwo{\nu^\onehalf\nabla \bfu_h}{\Omega}^2 + g_\nu(\bfu_h,\bfu_h)  \right)
\nonumber  \\
  &  \qquad  
  +     \left( \frac{1}{2} - \frac{\delta_1}{2} \left( 1 +\zeta_u \right) \right)
            \|  (\nu/(\gamma^n h))^{\onehalf} \bfu_h \cdot \bfn \|^2_{\GammaGNBC}
\nonumber  \\
  &  \qquad
     +  \left(  \frac{1}{2} -\frac{\delta_2}{2 } \left(1+\zeta_u\right)   \right)
           \| \left(\nu/(\varepsilon + \gamma^t h) \right)^{\onehalf} \bfu_h \boldsymbol{P}^t \|^2_{\GammaGNBC}.
  \end{align}
At this stage it is trivial to see that if an adjoint-inconsistent approach, i.e. $\zeta_u = -1$, is chosen, stability is guaranteed
for any positive choice of $\gamma^n,\gamma^t<\infty$.
More delicate is to show stability in the adjoint-consistent case, i.e. $\zeta_u= 1$, which is considered from here on.
From the last two rows we have an upper bound with $\delta_1,\delta_2 \leqslant c < 1$.
The second row can be rewritten as
\begin{align}
& \left(  1  - \frac{4 }{2\delta_1} \left( 1 +\zeta_u \right)\gamma^n  - \frac{4}{2 \delta_2} \left(1+\zeta_u\right) \gamma^t \left( \frac{  \gamma^t h }{\varepsilon + \gamma^t h} \right)   -  4  \zeta_u \gamma^t \left( \frac{  \varepsilon}{\varepsilon + \gamma^t h}  \right) \right)
\nonumber 
\\
& \qquad= \frac{1}{\varepsilon + \gamma^t h} 
   \left( 
      \varepsilon \left[ 1 - \frac{4}{\delta_1}\gamma^n -  4 \gamma^t \right]
       + \gamma^t h  \left[ 1 - \frac{4}{\delta_1}\gamma^n -  \frac{4}{\delta_2} \gamma^t\right]
   \right).
\end{align}
Choosing, for instance, $\delta_1=\delta_2=1/2$,
 stability of the adjoint-consistent Nitsche-type method can be guaranteed by bounding
the stabilization parameters $\gamma^n,\gamma^t$ by
$\gamma^n\leqslant \delta_1/8  = 1/16$ and
$\gamma^t\leqslant \delta_2/8  = 1/16$.
Note that the parameters also need to be lower bounded,
 i.e. $\gamma^n,\gamma^t>0$, 
 not to lose contributions of the bulk measures contained in the energy-type semi-norm $| U_h |_h$
and to prevent ill-conditioning. 
Finally, the claim follows by adding all other stabilization terms in $S_h$ and $G_h$, not already included in the analysis.
\end{proof}

The next three lemmas are recalled from the work by \citet{MassingSchottWall2016_CMAME_Arxiv_submit} and show how
the CIP, the ghost-penalty and the Nitsche-related stabilization terms can be used to recover
the missing velocity and pressure semi-norm parts which are included in $\tn V_h \tn_h$.

First, we start by recovering control over the incompressibility constraint with the help of the stabilization terms $s_u$ and $g_u$.
This is required for low Reynolds-numbers when the viscous $H^1$-semi-norm control $\|\nu^{\onehalf}\bfepsilon(\bfu_h)\|_\Omega$ vanishes.
\begin{lemma}
  \label{lem:divergence-stability}
There is a constant $c_1 > 0$ such that
for each $\bfu_h \in \mcV_h$ there exists a $q_h \in \mcQ_h$ satisfying
 \begin{align}
   - b_h(q_h, \bfu_h) 
   &\gtrsim
   \| \phi_u^{\onehalf} \nabla \cdot \bfu_h \|_{\Omega}^2
   - 
   c_1\Bigl(
   s_u(\bfu_h,\bfu_h)
   + g_u(\bfu_h,\bfu_h)
   + \| h^{-\onehalf}\phi_u^{\onehalf} \bfu_h\cdot \bfn \|_{\GammaGNBC}^2
   \Bigr)
  \label{eq:divergence-stability}
 \end{align}
and the stability estimate
\begin{align}
  \label{eq:divergence-stability-qh_stab_bound}
    \tn q_h \tn_{h,\phi}^2 = \|\phi^{-\onehalf} q_h\|_{\Omega}^2 + |q_h|_h^2
    &\lesssim 
    \| \phi_u^{\onehalf} \nabla \cdot \bfu_h \|_{\mcT_h}^2
    \lesssim
    \| \phi_u^{\onehalf} \nabla \cdot \bfu_h \|_{\Omega}^2
    + g_u(\bfu_h,\bfu_h)
 \end{align}
 with $\phi\in\{\phi_u,\Phi^{-1}\}$ whenever the CIP and ghost-penalty parameter
 $\gamma_u$ is chosen large.
\end{lemma}

\begin{proof}
 A detailed proof can be found in \citet{MassingSchottWall2016_CMAME_Arxiv_submit} (cf. Lemma 6.3).
\end{proof}

The next lemma shows how additional control over a mixed semi-norm of the
form $\| \phi_{\beta}^{\onehalf}(\bfbeta \cdot \nabla \bfu_h + \nabla
p_h)\|_{\Omega}^2$ can be recovered with the help of the CIP stabilizations $s_\beta,s_p$ and the ghost penalty stabilizations.

\begin{lemma}
  \label{lem:convect-pressure-stab}
There is a constant $c_2 > 0$ such that for each $U_h=(\bfu_h,p_h)\in \mcW_h$
we can construct a $\bfv_h \in V_h$ satisfying
 \begin{align}
   (\bfbeta \cdot \nabla \bfu_h + \nabla p_h, \bfv_h)_{\Omega}
   &\gtrsim
   \| \phi_{\beta}^{\onehalf}(\bfbeta \cdot \nabla \bfu_h + \nabla p_h)\|_{\Omega}^2
   - c_2 (s_\beta(\bfu_h,\bfu_h) + s_p(p_h,p_h) + \omega_h \|\nu^{\onehalf}\nabla\bfu_h\|_{\Omega_h^\ast}^2 + \|\sigma^{\onehalf}\bfu_h\|_{\Omega_h^\ast}^2) \\
   &\gtrsim
   \| \phi_{\beta}^{\onehalf}(\bfbeta \cdot \nabla \bfu_h + \nabla p_h)\|_{\Omega}^2
   - c_2 (1 + \omega_h) | U_h |_h^2
  \label{eq:convect-pressure-stab}
\end{align}
and the stability estimate
\begin{align}
   \| h^{-\onehalf}\phi_u^{\onehalf} \bfv_h\cdot\bfn\|_{\Gamma}^2
+ \tn \bfv_h \tn_h^2
+ \| \phi_{\beta}^{\onehalf}\bfbeta \cdot \nabla \bfv_h\|_{\Omega}^2
+ \| \phi_u^{\onehalf} \nabla \cdot \bfv_h \|_{\Omega}^2
   \lesssim
   \| \phi_{\beta}^{\onehalf}(\bfbeta \cdot \nabla \bfu_h + \nabla p_h)\|_{\Omega}^2
   + (1 + \omega_h) | U_h |_h^2,
   \label{eq:convect-pressure-stab-vh_stab_bound_2}
 \end{align}
  whenever the stability parameters $\gamma_\beta,\gamma_p, \gamma_\nu,\gamma_\sigma$ are chosen sufficiently large and the Nitsche penalty parameter $\gamma^n$ is chosen sufficiently small.
\end{lemma}
\begin{proof}
A detailed proof of this estimate has been presented in \citet{MassingSchottWall2016_CMAME_Arxiv_submit} (cf. Lemma 6.4). 
\end{proof}

Next, we recall two estimates from \cite{MassingSchottWall2016_CMAME_Arxiv_submit} which will be useful in deriving the final stability estimate in Theorem~\ref{thm:inf-sup_condition_total}.
\begin{lemma}
  \label{lem:Phi_p_bounds}
Let $\bfu_h,\bfv_h\in \mcV_h$, then the following estimates hold
  \begin{align}
    \tn \bfv_h \tn_h^2
+
(1+\omega_h)^{-1}\|\phi_\beta^{\onehalf} \bfbeta\cdot\nabla\bfv_h \|_{\Omega}^2
+
\|\phi_u^{\onehalf} \nabla \cdot \bfv_h\|_{\Omega}^2
    &\lesssim
    \bigl(\nu + \| \bfbeta \|_{0,\infty, \Omega}h + \sigma C_P^2
    \bigr)
    \bigl(
    \| \nabla \bfv_h \|_{\mcT_h}^2 + \| h^{-\onehalf} \bfv_h \|_{\Gamma}^2
    \bigr)
    \\
    &
    \lesssim
    \Phi^{-1} 
    \bigl(
    \| \nabla \bfv_h \|_{\mcT_h}^2 + \| h^{-\onehalf} \bfv_h \|_{\Gamma}^2
    \bigr),
    \label{eq:Phi_p_bounds_1}
  \end{align}
  \begin{align}
    |( \bfu_h, \bfbeta \cdot \nabla \bfv_h)_{\Omega}|
    &\lesssim
    \tn \bfu_h \tn_h
    \dfrac{\|\bfbeta\|_{0,\infty,\Omega} C_P}{\sqrt{\nu + \sigma C_P^2}}
    \| \nabla \bfv_h \|_{\Omega}
    \lesssim
    \tn \bfu_h \tn_h
    \Phi^{-\onehalf} 
    \| \nabla \bfv_h \|_{\Omega}.
    \label{eq:Phi_p_bounds_2}
  \end{align}
\end{lemma}
\begin{proof}
A detailed proof can be found in \citet{MassingSchottWall2016_CMAME_Arxiv_submit} (cf. Lemma 6.5).
\end{proof}

Finally, a stabilized inf-sup condition for the operator $b_h(p_h, \bfv_h)$ holds, provided that the pressure stabilization operator~$s_p$ is added,
which gives the desired $L^2$-pressure norm control.
\begin{lemma}
There is a constant $c_3 > 0$ such that
for each pressure $p_h \in \mcQ_h$ there exists a velocity field $\bfv_h \in \mcV_h$ satisfying
 \begin{align}
   b_h(p_h, \bfv_h) 
   &\gtrsim
   \Phi \| p_h \|_{\Omega}^2
   - c_3
s_p(p_h, p_h)
  \label{eq:pressure-stability}
 \end{align}
and the stability estimate
\begin{align}
  \label{eq:pressure-stability-vh_stab_bound}
    \tn \bfv_h \tn_h^2
+
(1+\omega_h)^{-1}\|\phi_\beta^{\onehalf} \bfbeta\cdot\nabla\bfv_h\|_{\Omega}^2
+
\|\phi_u^{\onehalf} \nabla \cdot \bfv_h\|_{\Omega}^2
    &\lesssim
    \Phi^{-1} (
    \| \nabla \bfv_h \|_{\mcT_h}^2
    +
    \| h^{-\onehalf} \bfv_h \|_{\Gamma}^2
    )
    \lesssim
    \Phi\| p_h \|_{\Omega}^2 + g_p(p_h,p_h),
 \end{align}
  whenever the stability parameters 
  $\gamma_\nu,\gamma_\sigma,\gamma_\beta,\gamma_u,\gamma_p$ are chosen
  large enough. 
  \label{lem:pressure-stability}
\end{lemma}

\begin{proof}
A detailed proof of this modified inf-sup condition has been presented in \citet{MassingSchottWall2016_CMAME_Arxiv_submit} (cf. Lemma 6.6).
\end{proof}

Collecting the previous lemmas, the final inf-sup stability estimate of our stabilized cut finite element method
$A_h+S_h+G_h$, see \eqref{eq:oseen-discrete-unfitted}, can be stated with respect to the energy norm $\tn U_h \tn_h^2$.
The subsequent theorem ensures existence and uniqueness of a discrete velocity and pressure solution.
\begin{theorem}
\label{thm:inf-sup_condition_total}
 Let $U_h=(\bfu_h,p_h)\in \mcW_h$.
Then, under the assumptions of Lemma~\ref{lem:coercivity_ah} on the Nitsche penalty parameter~$\gamma^t$ depending on $\zeta_u\in\{-1,1\}$,
on $\gamma^n$ being sufficiently small as stated in Lemmas~\ref{lem:coercivity_ah} and~\ref{lem:convect-pressure-stab} and that CIP and ghost penalty stability parameters
$\gamma_\nu,\gamma_\sigma,\gamma_\beta,\gamma_u,\gamma_p$ are chosen large enough (see Lemmas~\ref{lem:coercivity_ah}--\ref{lem:pressure-stability}),
the cut finite element method \eqref{eq:oseen-discrete-unfitted} is inf-sup stable 
 \begin{equation}
\label{eq:inf-sup_condition_Ah_weak_norm}
\tn U_h \tn_h
\lesssim
\sup_{V_h\in \mcW_h \setminus \{0\}}
\dfrac{
A_h(U_h,V_h) + S_h(U_h,V_h)
+ G_h(U_h,V_h)
}
{\tn V_h \tn_h }.
\end{equation}
Note that the hidden stability constant is independent of the mesh size $h$, the slip-length coefficient~$\varepsilon\in [o,\infty]$
and in case of unfitted meshes independent of the position of the boundary relative to the background mesh.
\end{theorem}

\begin{proof}
Since the proof follows the procedure proposed by \citet{MassingSchottWall2016_CMAME_Arxiv_submit}, the major steps are only sketched in the following.
For a given $U_h\in\mcW_h$ we construct a test function $V_h\in\mcW_h$
based on estimates derived in Lemmas~\ref{lem:coercivity_ah}, \ref{lem:divergence-stability}, \ref{lem:convect-pressure-stab}
and~\ref{lem:pressure-stability}.
\\
\noindent{\bf Step 1.}
Choosing the test function $V_h^1 := (\bfzero,q_h^1)$ with $q_h^1$ from Lemma~\ref{lem:divergence-stability}
to recover divergence control and utilizing its stability bounds \eqref{eq:divergence-stability} and \eqref{eq:divergence-stability-qh_stab_bound} we have
\begin{align}
  (A_h
  + S_h
  + G_h)(U_h,V_h^1)
  &= -b_h(q_h^1, \bfu_h) + s_p(p_h,q_h^1) + g_p(p_h, q_h^1)
  \\
  &\gtrsim
  \| \phi_u^{\onehalf} \nabla \cdot \bfu_h \|_{\Omega}^2
   - 
   c_1 \bigl(
  |\bfu_h|_h^2
   + \| h^{-\onehalf}\phi_u^{\onehalf} \bfu_h\cdot \bfn \|_{\GammaGNBC}^2
   \bigr)
  -\delta^{-1}
  |p_h|_h^2
  -\delta
  \|\phi_u^{\onehalf}\nabla \cdot \bfu_h\|_{\mcT_h}^2
\\
  &\gtrsim
  (1-\delta)\| \phi_u^{\onehalf} \nabla \cdot \bfu_h \|_{\Omega}^2
   - C_1(\delta) |U_h|_h^2
\end{align}
for any positive $\gamma^n \leqslant C < \infty$. This shows the need for the boundary-normal Nitsche penalty term at $\GammaGNBC$
for convective and reactive dominant flows as well as the need for the pressure and divergence CIP and ghost penalty terms.

\noindent{\bf Step 2.}
Inserting $V_h^2 := (\bfv_h^2,0)$ with $\bfv_h^2$, from Lemma~\ref{lem:convect-pressure-stab} into the formulation $A_h + S_h + G_h$,
followed by applying the continuity estimate \eqref{eq:continuity_ah_3} for $a_h$
and the stability bounds \eqref{eq:convect-pressure-stab} and \eqref{eq:convect-pressure-stab-vh_stab_bound_2}
yields the following
\begin{align}
(A_h+S_h+G_h)(U_h, V_h^2)
&= a_h(\bfu_h, \bfv_h^2) + b_h(p_h, \bfv_h^2) + (S_h+G_h)(U_h, V_h^2)
\\
&\gtrsim 
- \delta^{-1}\tn \bfu_h \tn_h^2
- \delta \tn \bfv_h^2 \tn_h^2
+ \| \phi_{\beta}^{\onehalf}(\bfbeta \cdot \nabla \bfu_h + \nabla p_h)\|_{\Omega}^2
   - c_2(1 + \omega_h) | U_h |_h^2
\\
&\gtrsim 
(1-\delta)\| \phi_{\beta}^{\onehalf}(\bfbeta \cdot \nabla \bfu_h + \nabla p_h)\|_{\Omega}^2
  - C_2(\delta)(1 + \omega_h) | U_h |_h^2.
\label{eq:A_h-step-2}
\end{align}

\noindent{\bf Step 3.}
The $L^2$-pressure norm term can be constructed by testing with
$V_h^3 := (\bfv_h^3,0)$, where $\bfv_h^3$ is now chosen as in
Lemma~\ref{lem:pressure-stability}.
Utilizing the continuity estimate \eqref{eq:continuity_ah_4} for $a_h$ and
making use of the estimates~\eqref{eq:Phi_p_bounds_1}, \eqref{eq:Phi_p_bounds_2}, and
the stability bound~\eqref{eq:pressure-stability-vh_stab_bound} for the chosen test function
allows us to deduce that
\begin{align}
(A_h+S_h+G_h)(U_h, V_h^3) 
&= a_h(\bfu_h, \bfv_h^3) + b_h(p_h, \bfv_h^3) + (S_h+G_h)(U_h, V_h^3)
\\
&\gtrsim 
- \tn \bfu_h \tn_h
\tn \bfv_h^3 \tn_h
- (\bfu_h, \bfbeta\cdot\nabla \bfv_h^3)_{\Omega}
+ 
  \Phi \| p_h \|_{\Omega}^2 - c_3 s_p(p_h, p_h)
  \\
  &\gtrsim
  - \delta^{-1}\tn \bfu_h \tn_h^2
  - \delta \Phi^{-1}(\|\nabla \bfv_h^3\|_{\mcT_h}^2 + \|h^{-\onehalf}
  \bfv_h^3\|_{\Gamma}^2)
  - \delta^{-1} \tn \bfu_h \tn_h^2 
  - \delta \Phi^{-1} \| \nabla \bfv_h^3 \|_{\Omega}^2
  \nonumber
  \\
  &\phantom{\gtrsim}\quad
  + 
  \Phi \| p_h \|_{\Omega}^2 - c_3 s_p(p_h, p_h)
  \\
  &\gtrsim
  (1-2\delta) \Phi \| p_h \|_{\Omega}^2
  -2 \delta^{-1} \tn \bfu_h \tn_h^2 - c_3s_p(p_h,p_h) - 2\delta
  g_p(p_h,p_h)
  \\
&\gtrsim 
  \Phi \| p_h \|_{\Omega}^2  - C_3(\delta) | U_h |_{h}^2.
\end{align}

\noindent{\bf Step 4.}
From the coercivity estimate in Lemma~\ref{lem:coercivity_ah}
we obtain the positive semi-norm term $|U_h|_h$ with $V_h^4 := U_h$
\begin{align}
  A_h(U_h, V_h^4) + S_h(U_h, V_h^4) + G_h(U_h, V_h^4) \gtrsim |U_h|_h^2.
\end{align}
\noindent{\bf Step 5.}
It remains to combine Step 1--4 by choosing $\delta$ sufficiently small and defining the final test function
\begin{equation}
  V_h^5 := \eta (V_h^1 + (1+\omega_h)^{-1}V_h^2 + V_h^3) + V_h^4
\end{equation}
for a given~$U_h$.
Choosing \mbox{$\eta>0$} sufficiently small for some \mbox{$2\eta~\sim (C_1(\delta)+C_2(\delta)+C_3(\delta))^{-1}$}
allows us to gain control over all desired norm parts in $\tn U_h \tn_h$ and at the same time to absorb
the defective \mbox{$|U_h|_{h}$}-contribution, which stem from testing with~\mbox{$V_h^1,V_h^2,V_h^3$}.
Consequently,
\begin{align}
(A_h+S_h+G_h)(U_h, V_h^5)
  &\gtrsim
  (1 - \eta(C_1(\delta) + C_2(\delta) + C_3(\delta)))| U_h |_{h}^2 \nonumber\\
  &\quad
  +
  \eta \left(\|\phi_u^{\onehalf}\nabla\cdot\bfu_h\|_{\Omega}^2
  +
  \dfrac{1}{1 + \omega_h} 
  \| \phi_{\beta}^{\onehalf}(\bfbeta \cdot \nabla \bfu_h + \nabla p_h)\|_{\Omega}^2
  + \Phi \|p_h\|_{\Omega}^2 \right)
\gtrsim \tn U_h \tn_h^2.
\end{align}
The inf-sup stability estimate can be concluded as
\begin{align}
(A_h+S_h+G_h)(U_h, V_h^5)
&\gtrsim \tn U_h \tn_h \tn V_h^5 \tn_h,
\end{align}
after dividing by \mbox{$\tn V_h^5 \tn_h$} and choosing the supremum over \mbox{$V_h\in\mcW_h\backslash \{0\}$}.
To conclude the proof, it remains to prove that \mbox{$\tn V_h^5 \tn_h \lesssim \tn U_h \tn_h$}.
Note that
\begin{align}
 \tn V_h^5 \tn_h \leqslant \tn U_h \tn_h + \eta \tn V_h^1 \tn_h + \frac{\eta}{1+\omega_h} \tn V_h^2 \tn_h + \eta \tn V_h^3 \tn_h.
\end{align}
Thanks to the stability estimate~\eqref{eq:convect-pressure-stab-vh_stab_bound_2} and
norm definitions~\eqref{eq:oseen-norm-up} and \eqref{eq:Ah-semi-norm} it holds that
\begin{align}
  \tn V_h^2 \tn_h^2
  &= \tn \bfv_h^2 \tn_h^2 + g_p(0,0) + s_p(0,0) + (1 + \omega_h)^{-1} \|\phi_{\beta}^{\onehalf}\bfbeta\cdot \nabla \bfv_h^2 \|_{\Omega}^2 + \|\phi_u^{\onehalf} \nabla \cdot \bfv_h^2 \|_{\Omega}^2\\
  &\lesssim (1 + (1 + \omega_h)^{-1}) (   \|\phi_{\beta}^{\onehalf}(\bfbeta\cdot \nabla \bfu_h + \nabla p_h) \|_{\Omega}^2 + (1+\omega_h) | U_h |_{h}^2   )\\
  &\lesssim (1 + \omega_h) \tn U_h \tn_h^2.
\end{align}
Similarly, the stability bound \eqref{eq:divergence-stability-qh_stab_bound}
for $q_h^1$ implies that \mbox{$\tn V_h^1 \tn_{h}^2 = \Phi\|q_h^1\|_{\Omega}^2 + |q_h^1|_{h}^2 \lesssim \tn U_h \tn_{h}^2$}
and from \eqref{eq:pressure-stability-vh_stab_bound} it is obtained that
\mbox{$\tn V_h^3 \tn_h^2 = \tn \bfv_h^3 \tn_h^2 + (1+\omega_h)^{-1}\|\phi_\beta^{\onehalf} \bfbeta\cdot\nabla\bfv_h^3\|_{\Omega}^2 + \|\phi_u^{\onehalf} \nabla \cdot \bfv_h^3\|_{\Omega}^2 \lesssim \tn U_h \tn_{h}^2$}.
As a result it holds 
\begin{equation}
  \tn V_h^5 \tn_h \lesssim (1+ \eta + \eta(1 + \omega_h)^{-\onehalf} + \eta) \tn U_h \tn_h \lesssim \tn U_h \tn_h,
\end{equation}
which concludes the proof.
\end{proof}

\section{A Priori Error Estimates}
\label{sec:apriori-analysis}
This section is devoted to the \apriori~error analysis of our cut finite element method.
We derive an energy-norm error estimate for the discrete velocity and pressure solution
in Theorem~\ref{thm:apriori-estimate}
and use a standard duality technique to establish optimal error convergence in the velocity $L^2$-norm
for the adjoint-consistent formulation with dominating viscous flow effects
in Theorem~\ref{thm:apriori-estimate-L2}.

\subsection{Consistency and Interpolation Error Estimates}
We start by showing that the discrete formulation~\eqref{eq:oseen-discrete-unfitted}
satisfies a weakened form of the
Galerkin orthogonality.
\begin{lemma}[Weakened Galerkin Orthogonality]
\label{lem:weak-orthogonality}
  Assume that the solution $U = (\bfu,p)$ of the
  formulation~\eqref{eq:oseen-problem-momentum}--\eqref{eq:oseen-problem-boundary-GNBC-tangential} is sufficiently regular, i.e. in
  $[H^2(\Omega)]^d \times H^1(\Omega)$,
  and let
  $U_h = (\bfu_h, p_h) \in \mcV_h \times \mcQ_h$ 
  be the cut finite element solution to the
  discrete weak formulation~\eqref{eq:oseen-discrete-unfitted}.
  Then, the error $U-U_h$ satisfies a weak Galerkin orthogonality property
  \begin{align}
    A_h(U - U_h, V_h) = S_h(U_h, V_h) + G_h(U_h, V_h) \quad \foralls V_h \in \mcV_h\times \mcQ_h .
    \label{eq:weak-orthogonality}
  \end{align}
\end{lemma}
\begin{proof}
  The proof follows standard techniques. Multiplying the problem formulation \eqref{eq:oseen-problem-momentum}--\eqref{eq:oseen-problem-compressible}
with test functions $V_h=(\bfv_h,q_h)\in \mcV_h\times \mcQ_h$, integrating over $\Omega$, performing integration by parts
and using the fact that all additional Nitsche-related terms vanish for $U$ satisfying the boundary condition,
yields
  \begin{alignat}{2}
   \textbf{I} &= A_h(U,V_h) - L_h(V_h) = 0 \quad &&\foralls V_h \in \mcV_h\times Q_h, \\
   \textbf{II} &= A_h(U_h,V_h) + S_h(U_h,V_h) + G_h(U_h,V_h) - L_h(V_h) = 0 \quad &&\foralls V_h \in \mcV_h\times \mcQ_h,
 \end{alignat}
where the last equation holds for the discrete solution $U_h\in\mcV_h\times\mcQ_h$. The combination $\textbf{I}-\textbf{II}$ yields the claim.
\end{proof}

The subsequent lemma recalls the well known weak consistency property of continuous interior penalty and ghost penalty stabilization operators $S_h,G_h$,
and as such they do not deteriorate the optimality of our cut finite element scheme.
\begin{lemma}[Weak Consistency]
\label{lem:weak-consistency}
For all functions $(\bfu, p) \in [H^r(\Omega)]^d \times H^s(\Omega)$ there holds
\begin{align}
  S_h(\Piast U, \Piast U)
  +
  G_h(\Piast U, \Piast U)
  &\lesssim
  ( \nu + \|\bfbeta\|_{0,\infty,\Omega}h +  \sigma h^{2})
   h^{2r_{u} -2}\| \bfu \|_{r_{u}, \Omega}^2
  \nonumber
  \\
  &\phantom{\lesssim}\quad
     +
   \max_{T \in \mcT_h}
   \left\{
   (\nu + \|\bfbeta \|_{0,\infty, T} h + \sigma h^2)^{-1}
   \right\}
   h^{2 s_p}
   \| p \|_{s_p, \Omega}^2,
\end{align}
where $k$ is the polynomial degree of the respective Cl\'ement interpolants $\Piast U = (\bfpiast \bfu,\piast p)$ for the velocity and pressure
and $r_{u} := \min\{r, k+1\}$ and $s_p := \min\{s, k+1\}$.
\end{lemma}
\begin{proof}
A detailed proof of the weak consistency of the proposed continuous interior penalty and ghost penalty operators $S_h$ and $G_h$ has been given in \cite{MassingSchottWall2016_CMAME_Arxiv_submit} (see Lemma 7.2)
and the references therein. 
\end{proof}
The next lemma ensures that
the interpolation error between a continuous solution and its Cl\'ement interpolation
converges with optimal rates.

\begin{lemma}[Interpolation Estimates]
\label{lem:interpolation-estimate}
Assume that $(\bfu,p) \in [H^r(\Omega)]^d \times H^s(\Omega)$
and let $r_{u} := \min\{r, k+1\} \geqslant 2$, $s_p := \min\{s,k+1\}\geqslant 1$
where $k$ is the polynomial degree of the approximation spaces
for the velocity and pressure. Then
\begin{align}
\tn \bfu  - \bfpiast \bfu \tn_{\ast}
+ \|\phi_u^{\onehalf} \nabla\cdot(\bfu  - \bfpiast \bfu)\|_{\Omega}
&\lesssim
(\nu +  \|\bfbeta\|_{0,\infty,\Omega} h + \sigma h^2)^{\onehalf}
h^{r_u-1}\|\bfu \|_{r_u,\Omega},
\label{eq:interpolation-estimate-u}
\\
\tn p  - \piast p \tn_{\ast,\phi} 
&\lesssim
\max_{T\in\mcT_h}
\left\{
(\nu + \|\bfbeta\|_{0,\infty,T}h + \sigma h^2)^{-1}
\right\}^{\onehalf}
h^{s_p}\|p \|_{s_p, \Omega}
\label{eq:interpolation-estimate-p}
\end{align}
with $\phi\in\{\phi_u,\Phi^{-1}\}$ in the pressure estimate.
\end{lemma}

\begin{proof}
  The proof follows the techniques proposed in \cite{MassingSchottWall2016_CMAME_Arxiv_submit} (see Lemma 7.4).
By applying the interpolation estimate~\eqref{eq:interpest0-ast}, the viscous, the reactive and the divergence bulk error measures can be estimated as
\begin{align}
\label{eq:interpolation-estimate-u-tmp}
 &\| \nu^{\onehalf} \nabla(\bfu^{\ast} - \bfpiast \bfu) \|_{\Omega}^2
+
\|\sigma^{\onehalf} (\bfu^{\ast} - \bfpiast \bfu) \|_{\Omega}^2
+
\| \phi_u^{\onehalf} \nabla\cdot(\bfu^{\ast} - \bfpiast \bfu) \|_{\Omega}^2
\lesssim
(\nu + \|\bfbeta\|_{0,\infty,\Omega}h + \sigma h^2) h^{2(r_u-1)}\|\bfu \|_{r_u,\Omega}^2
\intertext{and similarly for the pressure}
& \| \phi_u^{-\onehalf} (p^{\ast} - \piast p) \|_{\Omega}^2
\lesssim
\sum_{T\in \mcT_h}{
(\nu + \|\bfbeta\|_{0,\infty,T}h + \sigma h^2)^{-1}
h^{2s_p} 
 \| p^{\ast} \|_{s_p, \omega(T)}^2
}
\lesssim
\max_{T\in\mcT_h}
\left\{
(\nu + \|\bfbeta\|_{0,\infty,T}h + \sigma h^2)^{-1}
\right\}
h^{2s_p}\|p \|_{s_p, \Omega}^2
.
\end{align}
Considering the normal and tangential boundary semi-norms with $\gamma^t,\gamma^n \geqslant c >0$ and $\varepsilon \geqslant 0$,
the claim follows by combining the interpolation estimate~\eqref{eq:interpest0-ast},
the trace inequality~\eqref{eq:trace-inequality} and the definition of the scaling function $\phi_u$ such that
\begin{align}
\|((\nu + \phi_u)/(\gamma^n h))^{\onehalf} (\bfu^{\ast} - \bfpiast \bfu)\boldsymbol{P}^n \|_{\GammaGNBC}^2
&\lesssim
\|((\nu + \phi_u)/h)^{\onehalf} (\bfu^{\ast} - \bfpiast \bfu) \|_{\GammaGNBC}^2
\\
&\lesssim
(\nu +  \|\bfbeta\|_{0,\infty,\Omega}h + \sigma h^2)
\bigl(
h^{-2}\| \bfu^{\ast} - \bfpiast \bfu \|_{\mcT_h}^2
+\|\nabla(\bfu^{\ast} - \bfpiast \bfu)\|_{\mcT_h}^2
\bigr)
\\
&\lesssim
(\nu +  \|\bfbeta\|_{0,\infty,\Omega}h + \sigma h^2) h^{2(r_u-1)}\|\bfu \|_{r_u,\Omega}^2,
\\
\|(\nu /(\varepsilon + \gamma^t h))^{\onehalf} (\bfu^{\ast} - \bfpiast \bfu)\boldsymbol{P}^t \|_{\GammaGNBC}^2
&\lesssim
\|(\nu /h)^{\onehalf} (\bfu^{\ast} - \bfpiast \bfu) \|_{\GammaGNBC}^2
\\
&\lesssim
\nu
\bigl(
h^{-2}\| \bfu^{\ast} - \bfpiast \bfu \|_{\mcT_h}^2
+\|\nabla(\bfu^{\ast} - \bfpiast \bfu)\|_{\mcT_h}^2
\bigr)
\\
&\lesssim
\nu  h^{2(r_u-1)}\|\bfu \|_{r_u,\Omega}^2.
\end{align}
The remaining boundary terms can be estimated in a similar fashion,
\begin{align}
\| |\bfbeta\cdot\bfn|^{\onehalf} (\bfu^{\ast} - \bfpiast \bfu) \|_{\Gamma}^2
&\lesssim
\|\bfbeta\|_{0,\infty,\Omega}
(h^{-1}\| \bfu^{\ast} - \bfpiast \bfu \|_{\mcT_h}^2
+ h \|\nabla(\bfu^{\ast} - \bfpiast \bfu)\|_{\mcT_h}^2)
\lesssim
(\|\bfbeta\|_{0,\infty,\Omega}h) h^{2(r_u-1)}\|\bfu \|_{r_u,\Omega}^2,
\\
 \| (\nu h)^{\onehalf}\nabla(\bfu^{\ast}-\bfpiast\bfu)\cdot\bfn \|_{\Gamma}^2
&\lesssim
\nu \| \nabla(\bfu^{\ast}-\bfpiast\bfu) \|_{\mcT_h}^2
+
\nu h^2\| D^2(\bfu^{\ast}-\bfpiast\bfu) \|_{\mcT_h}^2
\lesssim
\nu h^{2(r_u-1)}\|\bfu\|_{r_u,\Omega}^2,
\\
\label{eq:interpolation-estimate-tmp_pressure_flux}
 \| \phi_u^{-\onehalf} h^{\onehalf}(p^{\ast}-\piast p)\|_{\Gamma}^2
&\lesssim
 \| \phi_u^{-\onehalf} (p^{\ast}-\piast p) \|_{\mcT_h}^2
+
 h^2\| \phi_u^{-\onehalf} \nabla(p^{\ast}-\piast p) \|_{\mcT_h}^2
\lesssim
\max_{T\in\mcT_h}\left\{\phi_u^{-1}\right\} h^{2 s_p} \|p\|_{s_p,\Omega}^2.
\end{align}
Collecting all estimates and noting that $\Phi \lesssim \phi_u^{-1}$ yields the claim.
\end{proof}

\subsection{A Priori Error Estimates}

Subsequently, the main \apriori~estimate for the velocity and pressure errors w.r.t a natural energy norm is stated.

\begin{theorem}[Energy norm error estimate]
  \label{thm:apriori-estimate}
Assume that the continuous solution of the Oseen problem~\eqref{eq:oseen-problem-momentum}--\eqref{eq:oseen-problem-boundary-GNBC-tangential}
resides in $U = (\bfu, p) \in [H^r(\Omega)]^d \times H^s(\Omega)$ and let
$U_h = (\bfu_h,p_h) \in \mcV_h \times \mcQ_h$ be the discrete solution
of problem~\eqref{eq:oseen-discrete-unfitted}.
Let the energy type norm be defined as in \eqref{eq:oseen-norm-fluxes-up}, then
\begin{align}
    \tn U -  U_h \tn_{\ast}
    &\lesssim
    (1+\omega_h)^{\onehalf}
    \bigl(
    \nu + \|\bfbeta \|_{0,\infty,\Omega} h + \sigma h^2
    \bigr)^{\onehalf}
    h^{r_{u} - 1}
    \| \bfu \|_{r_{u},\Omega}
      +
   \max_{T\in \mcT_h}
   \left\{(\nu + \|\bfbeta \|_{0,\infty, \Omega} h + \sigma h^2)^{-1}
   \right\}
   ^{\onehalf}
   h^{ s_p}
   \| p \|_{s_p, \Omega},
    \label{eq:apriori-estimate}
  \end{align} 
where $r_{u} := \min\{r, k+1\}\geqslant 2$ and $s_p := \min\{s, k+1\}\geqslant 1$.
Note that the hidden constants are independent of $h$, are bounded with respect to the slip length coefficient $\varepsilon\in[0,\infty]$
and, owing to the ghost penalty stabilization terms $G_h$, independent of how the boundary intersects the mesh $\mcT_h$.
The (hidden) scaling functions $\omega_h,\Phi$
are as defined in \eqref{eq:phi_p_definition}.
\end{theorem}
\begin{proof}
We first split the total discretization error into a discrete error and an interpolation part by applying \eqref{eq:oseen-norm-U-relation}
and the norm definitions in Section~\ref{ssec:discrete_form_and_norms}
\begin{align}
\tn U - U_h \tn_{\ast}
& \lesssim
\tn U - \Piast U \tn_{\ast}
+
\tn \Piast U - U_h \tn_{\ast}
\lesssim
\tn U - \Piast U \tn_{\ast}
+
\tn \Piast U - U_h \tn_{h}
.
  \label{eq:apriori_split_velocity}
\end{align}
As the term $ \tn  U - \Piast U \tn_{\ast}$ is readily estimated from the interpolation estimates~\eqref{eq:interpolation-estimate-u} and \eqref{eq:interpolation-estimate-p}, only the discrete error $\tn\Piast U-U_h \tn_h$ is considered from here on. 
From the inf-sup condition~\eqref{eq:inf-sup_condition_Ah_weak_norm} there exists a $\tn V_h \tn_h = 1$ such that
\begin{align}
  \tn \Piast U - U_h \tn_h
  &\lesssim
  A_h(\Piast U - U_h, V_h) 
  + S_h(\Piast U - U_h, V_h) 
  + G_h(\Piast U - U_h, V_h)
  \label{eq:discrete-error-est}
  \\
  &= 
  A_h(\Piast U - U, V_h) 
  + S_h(\Piast U, V_h) 
  + G_h(\Piast U, V_h),
  \label{eq:discrete-error-est-I}
\end{align}
where the last step follows from applying the weak Galerkin orthogonality~\eqref{eq:weak-orthogonality}.
By applying a Cauchy-Schwarz inequality on the stabilization terms $S_h(\Piast U, V_h)$ and $G_h(\Piast U, V_h)$ and utilizing the results from Lemma~\ref{lem:weak-consistency} the following estimate can be made
\begin{align}
\label{eq:consistency_error}
  S_h(\Piast U, V_h) + G_h(\Piast U, V_h)
  &\lesssim
  ( \nu + \|\bfbeta\|_{0,\infty,\Omega}h +  \sigma h^{2})^{\onehalf}
   h^{r_{u} -1}\| \bfu \|_{r_{u}, \Omega}
    +
   \max_{T \in \mcT_h}
   \left\{
   (\nu + \|\bfbeta \|_{0,\infty, T} h + \sigma h^2)^{-1}
   \right\}^{\onehalf}
   h^{ s_p}
   \| p \|_{s_p, \Omega}.
\end{align}
The term $A_h(\Piast U - U, V_h)$ can be estimated by integrating $b_h(q_h, \bfpiast \bfu - \bfu)$ by parts and
applying the continuity estimates \eqref{eq:continuity_ah_2} and \eqref{eq:continuity_bh_3} to
$a_h( \bfpiast \bfu - \bfu, \bfv_h)$ and $b_h(\piast p - p, \bfv_h)$ respectively, such that
\begin{align}
 &
  A_h(\Piast U - U, V_h) \nonumber\\
 &\qquad
=
 a_h( \bfpiast \bfu - \bfu, \bfv_h)
 + b_h(\piast p - p, \bfv_h)
 - b_h(q_h, \bfpiast \bfu - \bfu)
 \\
&\qquad
\lesssim
\tn  \bfpiast \bfu - \bfu \tn_{\ast}\tn \bfv_h \tn_{h} - (\bfpiast \bfu - \bfu, \bfbeta \cdot \nabla \bfv_h)_{\Omega}
 + \tn  \piast p - p \tn_{\ast,\phi_u} ( \tn \bfv_h \tn_{h} + \| \phi_u^{\onehalf} \nabla \cdot \bfv_h \|_{\Omega})
- (\bfpiast \bfu - \bfu, \nabla q_h)_{\Omega}
\\
&\qquad
\lesssim
(\tn  \bfpiast \bfu - \bfu \tn_{\ast} + \tn  \piast p - p \tn_{\ast,\phi_u} ) (\tn \bfv_h \tn_{h} + \| \phi_u^{\onehalf} \nabla \cdot \bfv_h \|_{\Omega})
- (\bfpiast \bfu - \bfu, \bfbeta \cdot \nabla \bfv_h + \nabla q_h)_{\Omega}.
\end{align}
Thanks to the interpolation estimate Lemma~\ref{lem:interpolation-estimate} and keeping in mind that $\tn \bfv_h \tn_{h} + \| \phi_u^{\onehalf} \nabla \cdot \bfv_h \|_{\Omega} \lesssim \tn V_h \tn_h = 1$,
it only remains to estimate the last term.
Its estimate follows readily from applying a Cauchy-Schwarz inequality, using \eqref{eq:interpest0} and the definition of~$\phi_\beta$ to arrive at
\begin{align}
  |(\bfpiast \bfu - \bfu, \bfbeta \cdot \nabla \bfv_h + \nabla q_h)| &\lesssim
(1+\omega_h)^{\onehalf} \| \phi^{-\onehalf}_{\beta}(\bfpiast \bfu - \bfu) \|_{\Omega}
\cdot
(1+\omega_h)^{-\onehalf} \| \phi_{\beta}^{\onehalf} (\bfbeta \cdot \nabla \bfv_h + \nabla q_h)\|_{\Omega}
\\
&\lesssim
(1+\omega_h)^{\onehalf} (\nu + \|\bfbeta \|_{0,\infty,\Omega}h + \sigma h^2)^{\onehalf}
 h^{r_u - 1}\| \bfu \|_{r_u,\Omega} 
\tn V_h \tn_h,
\end{align}
which concludes the proof of the \apriori~error estimate \eqref{eq:apriori-estimate}.
\end{proof}

\subsection{$L^2$-Optimal Estimate for Flows with Large Viscosity}


We now proceed to deduce an optimal $L^2$-error estimate for the velocity
of an adjoint-consistent ($\zeta_u=1$) formulation 
for flows dominated by viscous forces, i.e. from here on we assume
\begin{align}
\label{eq:assumption_low_RE}
  \nu \geqslant \|\bfbeta\|_{0,\infty,T} h_T + \sigma h_T^2 \quad \foralls T \in \mcT_h.
\end{align}
The proof follows the standard Aubin--Nitsche duality argument, see e.g. \cite{Aubin1967, Nitsche1968}, and requires the
established optimal energy-type error estimate from Theorem~\ref{thm:apriori-estimate}.
A similar estimate for the Oseen problem has been derived for boundary-fitted meshes by \citet{BurmanFernandezHansbo2006}.
%
We introduce the following dual problem to \eqref{eq:oseen-problem-momentum}--\eqref{eq:oseen-problem-boundary-GNBC-tangential}:
find adjoint velocity and pressure $(\bfw,r)$ such that
\begin{alignat}{2}
  \label{eq:oseen-problem-momentum-dual}
  \sigma \bfw +(-\bfbeta)\cdot\nabla\bfw - \nabla\cdot(2\nu\bfepsilon(\bfw)) + \nabla (-r) 
&= \tilde{\bff} \quad &&\text{ in } \Omega,
  \\
  \div\bfw &= 0 \quad &&\text{ in } \Omega,
  \label{eq:oseen-problem-compressible-dual}
  \\
  \bfw \boldsymbol{P}^n &= \bfzero
  \quad &&\text{ on } \GammaGNBC,
    \label{eq:oseen-problem-boundary-GNBC-normal-dual}\\
  \left( \varepsilon 2 \nu \bfepsilon(\bfw) \bfn + (\nu + \varepsilon \bfbeta \cdot \bfn) \bfw \right) \boldsymbol{P}^t &= \bfzero
  \quad &&\text{ on } \GammaGNBC,
  \label{eq:oseen-problem-boundary-GNBC-tangential-dual}
\end{alignat}
for which we assume additional elliptic regularity and that the solution belongs to $ (\bfw,r) \in [H^2(\Omega)]^d \times H^1(\Omega)$
so that it satisfies
\begin{align}
  \label{eq:estimate_dual_solution}
 \nu\|\bfw\|_{2,\Omega} + \|r\|_{1,\Omega} \leqslant C \|\tilde{\bff}\|_{0,\Omega}, 
\end{align}
provided that the boundaries are sufficiently smooth, see e.g. \cite{RaviartGirault1986, Roos2008, Quarteroni2009}.
Note that due to the homogeneous boundary conditions in \eqref{eq:oseen-problem-boundary-GNBC-normal-dual}--\eqref{eq:oseen-problem-boundary-GNBC-tangential-dual}
the estimate \eqref{eq:estimate_dual_solution} is independent of the boundary data.
However, the dimensionless constant $C$ depends on the
physical parameters $\nu,\sigma,\bfbeta$ and the domain $\Omega$.
By choosing $\tilde{\bff} = \bfu-\bfu_h\in L^2(\Omega)$ as the right hand side of the dual momentum equation,
the desired error quantity bounds the dual solution $(\bfw,r)$ in \eqref{eq:estimate_dual_solution}.
Note that the dual advective velocity is set as the negative advective velocity field of the primal problem
\eqref{eq:oseen-problem-momentum}--\eqref{eq:oseen-problem-boundary-GNBC-tangential}.
As a result, inflow and outflow parts of the boundary of primal and dual problems swap, respectively.
Similarly, the dual pressure solution $r$ changes the sign compared to the pressure solution $p$ of the primal problem.
For further explanations on the dual problem, see e.g. the textbook \cite{Quarteroni2009}.
Furthermore, it is assumed that $\bfbeta \cdot \bfn = 0$ on $\GammaGNBC\setminus \Gamma_{\mathrm{D}}$, which simplifies \eqref{eq:oseen-problem-boundary-GNBC-tangential-dual} in the following.


\begin{theorem}[Velocity $L^2$-error estimate]
  \label{thm:apriori-estimate-L2}
Let $U = (\bfu, p) \in [H^r(\Omega)]^d \times H^s(\Omega)$ be the continuous solution to the Oseen problem~\eqref{eq:oseen-problem-momentum}--\eqref{eq:oseen-problem-boundary-GNBC-tangential}
and $U_h = (\bfu_h,p_h) \in \mcV_h \times \mcQ_h$ be the discrete solution
of problem~\eqref{eq:oseen-discrete-unfitted}.
Under previously specified assumptions, for an adjoint-consistent Nitsche-type formulation ($\zeta_u=1$)
we have
\begin{align}
    \| \bfu -  \bfu_h \|_{0,\Omega}
    \lesssim
    h^{r_{u}}
    \| \bfu \|_{r_{u},\Omega}
   + h^{ s_p+1}
   \nu^{-1}\| p \|_{s_p, \Omega},
    \label{eq:apriori-estimate-L2}
  \end{align}
provided that viscous forces dominate the flow as assumed in \eqref{eq:assumption_low_RE}
with $r_{u} := \min\{r, k+1\}$ and $s_p := \min\{s, k+1\}$.
Note that the hidden constant is independent of $h$
and 
independent of how the boundary intersects the mesh $\mcT_h$.
However, the constant depends on the physical parameters.
\end{theorem}

\begin{proof}
We need to estimate the desired velocity $L^2$-error $\|\bfu-\bfu_h\|_{0,\Omega}$.
For this purpose, we choose $\tilde{\bff}=\bfu-\bfu_h$ as right hand side of the dual momentum equation \eqref{eq:oseen-problem-momentum-dual}
and multiply \eqref{eq:oseen-problem-momentum-dual}--\eqref{eq:oseen-problem-compressible-dual} with test functions $\bfv:=\bfu-\bfu_h$ and $-q:=-(p-p_h)$, respectively.
After integrating by parts, using $\div \bfbeta =0$ and the relation
\begin{align}
  \langle ((-\bfbeta)\cdot\bfn)\bfw,\bfv \rangle_{\Gamma}
= \bf0
\end{align}
as $\bfw=\bf0$ on $\Gamma_D$ and $\bfbeta\cdot\bfn= 0$ on $\GammaGNBC\setminus \Gamma_{\mathrm{D}}$, we obtain
\begin{align}
 \| \bfu-\bfu_h \|_{0,\Omega}^2
  &= (\tilde{\bff},\bfv)_{\Omega}
  = (\sigma \bfw,\bfv)_{\Omega} + ((-\bfbeta)\cdot\nabla\bfw, \bfv)_{\Omega} - (\nabla\cdot(2\nu\bfepsilon(\bfw)),\bfv)_{\Omega} + (\nabla (-r), \bfv)_{\Omega} - (\div \bfw, q)_{\Omega}
\\
  &= (\sigma \bfw,\bfv)_{\Omega}
+ (\bfw, \bfbeta\cdot\nabla\bfv)_{\Omega}
+ (\bfepsilon(\bfw),2\nu\bfepsilon(\bfv))_{\Omega} - \langle 2\nu\bfepsilon(\bfw)\bfn,\bfv\rangle_{\GammaGNBC}
+ ( r , \nabla\cdot \bfv)_{\Omega} - \langle r ,\bfv \cdot\bfn \rangle_{\GammaGNBC}
- (\div \bfw, q)_{\Omega}
\\
  &= a(\bfv,\bfw)
- \langle 2\nu\bfepsilon(\bfw)\bfn,\bfv\rangle_{\GammaGNBC}
-b_h(r,\bfv)
+ b(q,\bfw).
\label{eq:dual_estimate_1_last_line}
\end{align}
Using the boundary conditions for the normal and tangential directions \eqref{eq:oseen-problem-boundary-GNBC-normal-dual}--\eqref{eq:oseen-problem-boundary-GNBC-tangential-dual} with the assumption that $\bfbeta \cdot \bfn = 0$ on $\GammaGNBC\setminus \Gamma_{\mathrm{D}}$,
which then reduce to $\bfw\cdot\bfn = 0$ on $\GammaGNBC$ and $(\varepsilon(2\nu\bfepsilon(\bfw)\bfn) + \nu\bfw)\boldsymbol{P}^t = \bfzero $ on $\GammaGNBC$,
the following terms can be consistently added to \eqref{eq:dual_estimate_1_last_line}
\begin{align}
 0 &=
\langle \bfw\cdot\bfn , q \rangle_{\GammaGNBC}
- \langle \bfw\cdot\bfn, (2\nu\bfepsilon(\bfv)\bfn)\cdot\bfn \rangle_{\GammaGNBC}
+ \langle \frac{\nu}{\gamma^n h}  \bfw\cdot\bfn , \bfv\cdot\bfn \rangle_{\GammaGNBC} 
+ \langle \frac{\phi_u}{\gamma^n h} \bfw\cdot\bfn , \bfv\cdot\bfn \rangle_{\GammaGNBC}
= I + II + III + IV
,
\\
 0 &=
  \langle \frac{1}{\varepsilon + \gamma^t h} (\varepsilon (2\nu\bfepsilon(\bfw)\bfn) + \nu\bfw)\boldsymbol{P}^t, \bfv \rangle_{\GammaGNBC}
- \langle \frac{\gamma^t h }{\varepsilon + \gamma^t h} (\varepsilon (2\nu\bfepsilon(\bfw)\bfn) + \nu\bfw)\boldsymbol{P}^t, 2\bfepsilon(\bfv)\bfn \rangle_{\GammaGNBC}
= V + VI + VII + VIII.
\end{align}
Continuing in \eqref{eq:dual_estimate_1_last_line}, we note that $b(q,\bfw) + I = b_h(q, \bfw)$.
The viscous term can be split into directional parts such that
\begin{align}
 - \langle 2\nu\bfepsilon(\bfw)\bfn,\bfv\rangle_{\GammaGNBC} + V
&= -\langle (2\nu\bfepsilon(\bfw)\bfn)\cdot\bfn , \bfv\cdot\bfn \rangle_{\GammaGNBC}
-\langle \frac{\gamma^t h}{\varepsilon + \gamma^t h} (2\nu\bfepsilon(\bfw)\bfn)\boldsymbol{P}^t, \bfv \rangle_{\GammaGNBC}
\intertext{and}
 II + VIII
&=
- \langle \bfw\cdot\bfn, (2\nu\bfepsilon(\bfv)\bfn)\cdot\bfn \rangle_{\GammaGNBC}
- \langle \frac{(\varepsilon+\gamma^t h)-\varepsilon }{\varepsilon + \gamma^t h} \nu\bfw\boldsymbol{P}^t, 2\bfepsilon(\bfv)\bfn \rangle_{\GammaGNBC}
\\
&= - \langle \bfw, 2\nu\bfepsilon(\bfv)\bfn \rangle_{\GammaGNBC}
+\langle \frac{\varepsilon}{\varepsilon + \gamma^t h} \nu\bfw\boldsymbol{P}^t, 2\bfepsilon(\bfv)\bfn \rangle_{\GammaGNBC}.
\label{eq:dual_estimate_2_last_line}
\end{align}
Collecting all terms from \eqref{eq:dual_estimate_1_last_line}--\eqref{eq:dual_estimate_2_last_line}
and defining $W:=(\bfw,r)$, the $L^2$-error can be expressed as
\begin{align}
 \| \bfu-\bfu_h \|_{0,\Omega}^2
&= a_h(\bfu-\bfu_h,\bfw) +  b_h(p-p_h, \bfw) - b_h(r, \bfu-\bfu_h)
= A_h(U-U_h,W),
\end{align}
i.e. it can be expressed in terms of the discrete bilinear operator \eqref{eq:Ah-form-def}
associated to the primal problem \eqref{eq:oseen-problem-momentum}--\eqref{eq:oseen-problem-boundary-GNBC-tangential},
where the sufficiently smooth solution $W$ of the dual problem now takes the role of the test function.

Note that under the assumption of dominating viscous forces \eqref{eq:assumption_low_RE},
it holds for the advective term occurring in the continuity estimate of $a_h$ in \eqref{eq:continuity_ah_1} that
\begin{align}
 |(\bfbeta\cdot\nabla(\bfu-\bfu_h), \bfw-\bfpiast\bfw)_{\Omega}|
&\lesssim
\nu^{-\onehalf}\|\nu^{\onehalf}\nabla(\bfu-\bfu_h) \|_{\Omega}
\max_{T\in\mcT_h}
\{
\bigg(\underbrace{\frac{\|\bfbeta\|_{0,\infty,T}h}{\nu}}_{\lesssim 1}\bigg)
\}
\cdot
\nu h^{-1} \| \bfw-\bfpiast\bfw\|_{\Omega}
\\
&
\lesssim
\nu^{-\onehalf} \tn \bfu-\bfu_h \tn  \cdot h  \nu  \|\bfw \|_{2,\Omega},
\end{align}
where the interpolation estimate \eqref{eq:interpest0-ast} for the Cl\'ement interpolant was used in the last step.
The $L^2$-velocity error can then be further estimated by using the weak Galerkin orthogonality from Lemma~\ref{lem:weak-orthogonality}
with $\Piast W\in \mcV_h\times \mcQ_h$,
the continuity of $a_h,b_h$ provided in Lemma~\ref{lem:continuity} (equations \eqref{eq:continuity_ah_1} and \eqref{eq:continuity_bh_1}) and
by applying a Cauchy Schwarz inequality to the remaining stabilization operators
\begin{align}
 \| \bfu-\bfu_h \|_{0,\Omega}^2
&= A_h(U-U_h,W-\Piast W) + S_h(U_h,\Piast W) + G_h(U_h,\Piast W))
\\
&= a_h(\bfu-\bfu_h,\bfw-\bfpiast \bfw) +  b_h(p-p_h, \bfw-\bfpiast \bfw) - b_h(r-\piast r, \bfu-\bfu_h)
+ (S_h+G_h)(U_h,\Piast W)
\\
&\lesssim
\tn \bfu-\bfu_h \tn_{\ast} \tn \bfw-\bfpiast\bfw \tn_{\ast}
+ |(\bfbeta\cdot\nabla(\bfu-\bfu_h), \bfw-\bfpiast\bfw)_{\Omega}|
+ \tn U-U_h \tn_{\ast} \tn r-\piast r \tn_{\ast,\phi_u}
\nonumber
\\
&\qquad
+ \tn p-p_h \tn_{\ast,\Phi^{-1}} (\nu\Phi)^{-\onehalf} (\| \nu^{\onehalf}\nabla(\bfw-\bfpiast\bfw) \|_{\Omega}
+ \| (\nu/h)^{\onehalf} (\bfw-\bfpiast\bfw)\cdot\bfn \|_{\Gamma} )
\nonumber\\
&\qquad
+ (S_h+G_h)(U_h,U_h)^{\onehalf} \cdot (S_h+G_h)(\Piast W,\Piast W)^{\onehalf}
\\
&\lesssim
\tn U-U_h \tn_{\ast} \cdot (\tn\bfw-\bfpiast \bfw \tn_{\ast} + \tn r-\piast r \tn_{\ast,\phi_u})
+ (S_h+G_h)(U_h,U_h)^{\onehalf} \cdot (S_h+G_h)(\Piast W,\Piast W)^{\onehalf}
\\
&\lesssim \nu^{-\onehalf}(\tn U-U_h \tn_{\ast}  + (S_h+G_h)(U_h-\Piast U,U_h-\Piast U)^{\onehalf} + (S_h+G_h)(\Piast U,\Piast U)^{\onehalf}  )
\nonumber \\
&\qquad
\cdot \nu^{\onehalf}( \tn\bfw-\bfpiast \bfw \tn_{\ast} + \tn r-\piast r \tn_{\ast,\phi_u} + (S_h+G_h)(\Piast W,\Piast W)^{\onehalf})
\\
&\lesssim \nu^{-\onehalf}(\tn U-U_h \tn_{\ast}  + \tn U_h-\Piast U \tn_h + (S_h+G_h)(\Piast U,\Piast U)^{\onehalf}  )
\cdot h (\nu \| \bfw \|_{2,\Omega} + \|r\|_{1,\Omega})
\label{eq:l2-estimate-last-2}
\\
&\lesssim (1+\omega_h)^{\onehalf}\nu^{-\onehalf}( \nu^{\onehalf} h^{r_{u}-1} \| \bfu \|_{r_{u},\Omega} + \nu^{-\onehalf} h^{s_p}\| p \|_{s_p, \Omega})
\cdot h \cdot \|\bfu-\bfu_h\|_{0,\Omega}
\label{eq:l2-estimate-last-3}
\\
&\lesssim ( h^{r_{u}} \| \bfu \|_{r_{u},\Omega} + h^{s_p+1} \nu^{-1}\| p \|_{s_p, \Omega}) \cdot \|\bfu-\bfu_h\|_{0,\Omega}.
\label{eq:l2-estimate-last-line}
\end{align}
In line \eqref{eq:l2-estimate-last-2} we use the energy-norm \apriori~error estimate from Theorem~\ref{thm:apriori-estimate}
in \eqref{eq:apriori-estimate} and 
\eqref{eq:discrete-error-est}
for $(\bfu,p)\in [H^{r_u}(\Omega)]^d\times H^{s_p}(\Omega)$ under the assumption of dominant viscous effects \eqref{eq:assumption_low_RE}.
Note that for $\phi\in\{\phi_u,\Phi^{-1}\}$ it holds that $\nu/\phi\lesssim 1$.
Thanks to the interpolation estimate Lemma~\ref{lem:interpolation-estimate},
in \eqref{eq:l2-estimate-last-2} we gain the desired additional power of $h$ for $(\bfw,r)\in [H^2(\Omega)]^d\times H^1(\Omega)$
by estimating 
\begin{align}
\nu^{\onehalf}(\tn\bfw-\bfpiast \bfw \tn_{\ast} + \tn r-\piast r \tn_{\ast,\phi} + (S_h+G_h)(\Piast W,\Piast W)^{\onehalf})
\lesssim \nu^{\onehalf} h (\nu^{\onehalf}\|\bfw\|_{2,\Omega} + \nu^{-\onehalf}\|r\|_{1,\Omega})
\lesssim h \|\bfu-\bfu_h\|_{0,\Omega},
\end{align}
together with the boundedness \eqref{eq:estimate_dual_solution}
of the solution to the dual problem \eqref{eq:oseen-problem-momentum-dual}--\eqref{eq:oseen-problem-boundary-GNBC-tangential-dual}.
Note that for $S_h$ and $G_h$ the weak consistency estimates from Lemma~\ref{lem:weak-consistency} hold.
Finally, the claim follows after dividing by $\|\bfu-\bfu_h\|_{0,\Omega}$ in \eqref{eq:l2-estimate-last-line}.
\end{proof}

\section{Numerical Example}
\label{sec:numexamples}

To ensure the validity of the proposed method and the theoretical results presented, a numerical example is conducted. 
The error estimates obtained in the \apriori~error analysis, as summarized in the Theorems~\ref{thm:apriori-estimate} and \ref{thm:apriori-estimate-L2}, are validated by the results of a 2D box-flow case.

The stabilization parameters used for the following simulations are taken by large from \citet{MassingSchottWall2016_CMAME_Arxiv_submit} and \citet{SchottWall2014} and are repeated here for completeness.
The CIP-stabilization terms \eqref{eq:cip-s_beta}--\eqref{eq:cip-s_p} are set as $\gamma_\beta = \gamma_p=0.01$ and $\gamma_u=0.05\gamma_\beta$, as suggested in \cite{Burman2007}.
For the convective, pressure and velocity ghost penalty stabilizations \eqref{eq:ghost-penalty-beta}--\eqref{eq:ghost-penalty-p} the same parameters as for the CIP-stabilization are used.
In the case of the viscous \eqref{eq:ghost-penalty-nu} and (pseudo-) reactive \eqref{eq:ghost-penalty-sigma} ghost-penalty a value of $\gamma_\nu=0.05$ and $\gamma_\sigma = 0.005$ is prescribed, respectively. 
The second order terms of the ghost-penalties (i.e. for $j=2$ in the equations \eqref{eq:ghost-penalty-sigma},\eqref{eq:ghost-penalty-nu} and \eqref{eq:ghost-penalty-p} and $j=1$ in \eqref{eq:ghost-penalty-beta} and \eqref{eq:ghost-penalty-u}) are scaled by an extra $0.05$, as without it too strong enforcement of these terms were observed, ruining the solution.
Higher order terms of the ghost-penalties, i.e.~$j>2$ are neglected as their influence on the solution and stability is negligible for linear and quadratic elements.
Note that for simulations with higher-order approximations, i.e. $k>1$,
the simplified variant $\overline{g}_{\beta}$ \eqref{eq:gbeta-simple-def}
for the convective and divergence ghost penalty terms $g_\beta,g_u$ (see \eqref{eq:ghost-penalty-beta} and \eqref{eq:ghost-penalty-u})
and the related continuous interior penalty stabilizations $s_\beta,s_u$ (see \eqref{eq:cip-s_beta} and \eqref{eq:cip-s_u})
are used.

The different flow regimes appearing in $\phi_u,\phi_{\beta},\phi_p$ (see \eqref{eq:cip-Nitsche_scaling} and \eqref{eq:cip-s_scalings}) are
weighted as $\nu + c_u (\|\bfbeta\|_{0,\infty,T}h) + c_{\sigma} (\sigma h^2)$ with
$c_u = 1/6$ and $c_\sigma= 1/12$ as suggested in \cite{SchottRasthoferGravemeierWall2015}. If nothing else is mentioned, then the simulations are conducted with an adjoint-consistent Nitsche's method (i.e. $\zeta_u=1$) and with penalty parameters of $1/\gamma^n = 1/\gamma^t = 10.0$. Furthermore, as there is a need to integrate cut elements of non-regular shapes, standard integration rules can not be applied. To overcome this issue, the integration rules proposed by \citet{SudhakarMoitinhoWall2014} are used.

All simulations presented in this publication have been performed
using the parallel finite element software environment “Bavarian
Advanced Computational Initiative” (BACI), see \cite{WalletBaciCommittee2017}.

\subsection{Problem Setup -- 2D Box Flow}
This numerical example is inspired from the example done for Stokes flow in \cite{UrquizaGaronFarinas2014}, but made somewhat more complex and extended to the Oseen equations. 
We use the technique of manufactured solution to create the example.
A divergence-free velocity field ($\div \bfu = 0$) and a suitable pressure field are chosen which, when put into the Oseen equations with $\bfbeta=\bfu$, generate an associated volumetric body force~$\bff$.
By providing the volumetric body force and appropriate boundary conditions the solution is known and error studies can be carried out for the example.
The domain~$\Omega$ used here is a square $\{ (x,y), -1 < x < 1, -1 < y < 1 \}$ and the following choice for the solution field is

\begin{align}
u_1(x,y)&=0.75 y^3 (1 - x^4) + 1.25 y (1 - x^2) ,
\label{eq:oseen-box-flow-analytic-solution_u1}
\\ 
u_2(x,y)&=-0.75 x^3 (1 - y^4) - 1.25 x (1 - y^2),
\label{eq:oseen-box-flow-analytic-solution_u2}
\\ 
p(x,y)&= \left( \sum_{i=0}^{2} \frac{(3x)^{2i+1}}{(2i+1)!} \right) \left( \sum_{j=0}^{3} \frac{(3y)^{2j}}{(2j)!} \right),
\label{eq:oseen-box-flow-analytic-solution_p}
\end{align}
as visualized in Figure~\ref{fig:analytic-solution-of-box-flow}.
The geometry here is chosen such that we have no geometric approximation error by the meshes, thus avoiding issues stemming from the Babu{\v{s}}ka paradox \cite{Babuska1963_DEA}.
It can easily be seen that the chosen analytical velocity field ($\bfu = (u_1, u_2)$) satisfies $\div \bfu = 0$.
The volume force field $\bff$ is chosen in accordance with the given fields \eqref{eq:oseen-box-flow-analytic-solution_u1}--\eqref{eq:oseen-box-flow-analytic-solution_p} such that the Oseen equations \eqref{eq:oseen-problem-momentum}--\eqref{eq:oseen-problem-compressible} are satisfied.
In the discrete case, the force field~$\bff_h$ and the advective velocity~$\bfbeta_h$ are given by the nodal interpolations of their continuous counterparts.

\begin{figure}[ht]
  \centering
  \subfloat{\includegraphics[trim=200 0 200 0, clip, width=0.48\textwidth]{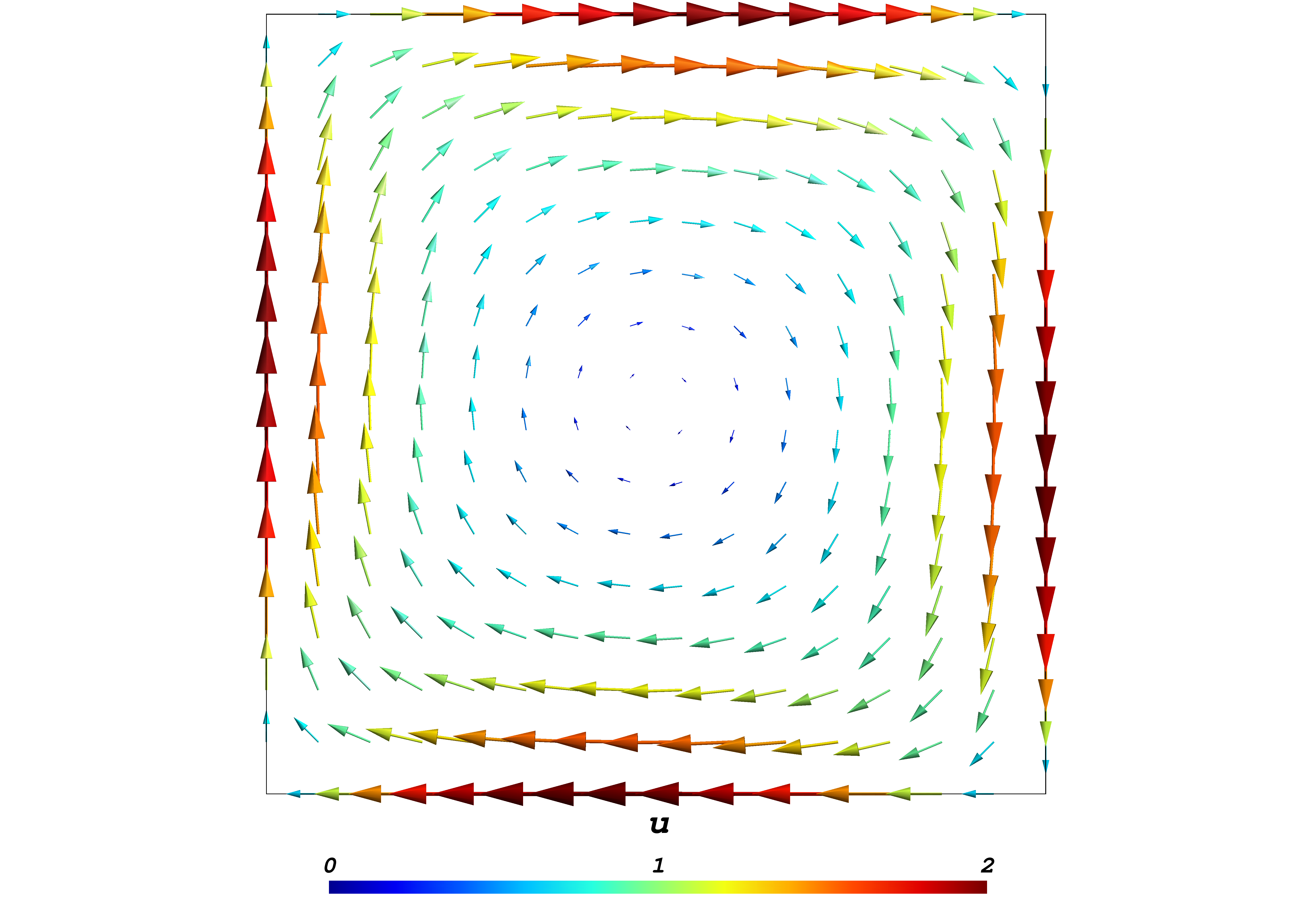}}
  \subfloat{\includegraphics[trim=200 0 200 0, clip, width=0.48\textwidth]{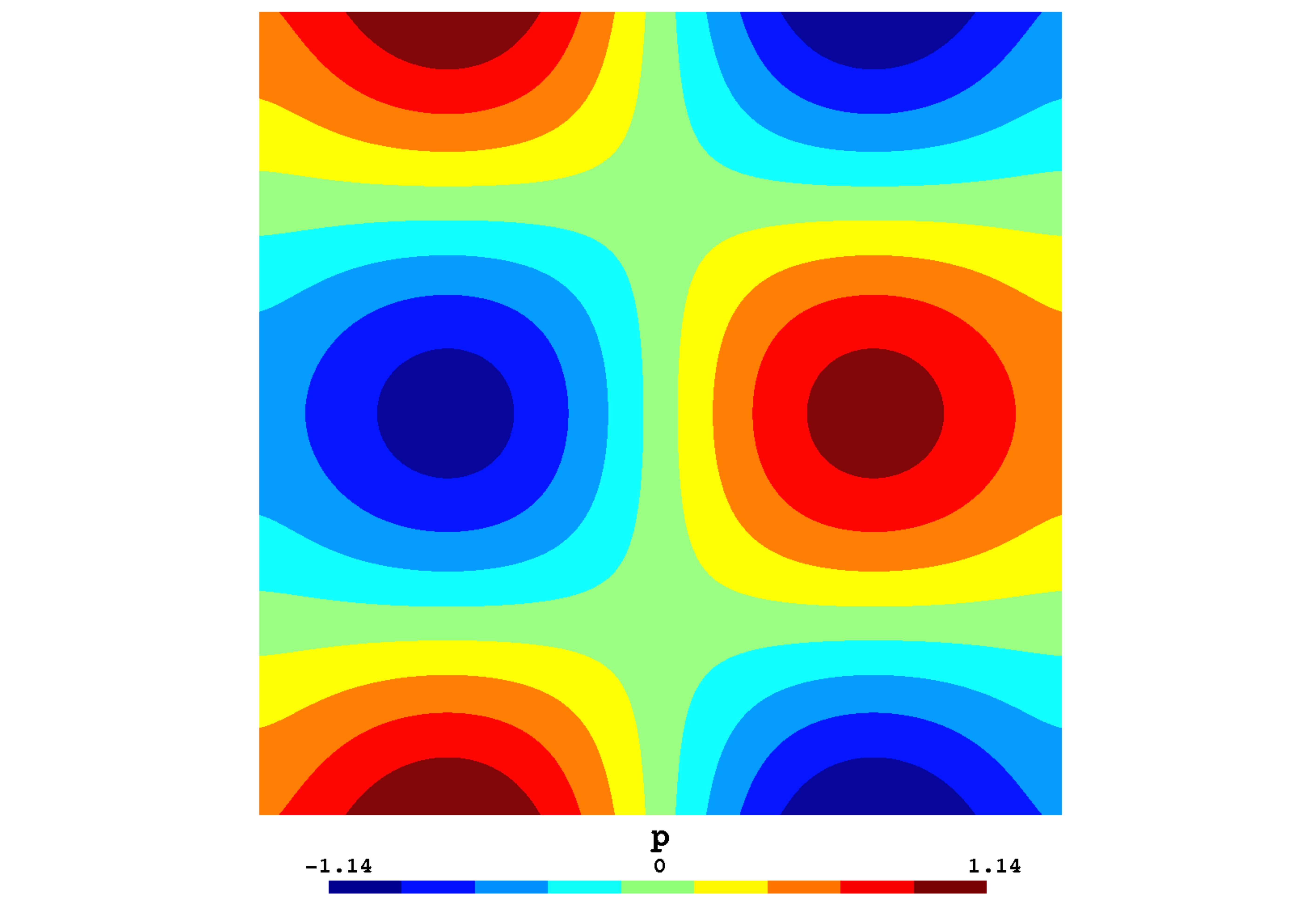}}
  \caption{The analytic velocity (left) and pressure (right) solution to the two-dimensional box-flow problem.}
  \label{fig:analytic-solution-of-box-flow}
\end{figure}

The boundary conditions applied to this problem are the same as introduced in the Oseen problem \eqref{eq:oseen-problem-boundary-GNBC-normal}--\eqref{eq:oseen-problem-boundary-GNBC-tangential}.
The velocity at the boundary is prescribed as $\bfg = \bfu$, and the traction at the boundary as $\bfh = 2 \nu \bfepsilon(\bfu)$. It should be mentioned that $\bfu \boldsymbol{P}^n = \bfzero$ on $\Gamma$, which satisfies the conditions assumed throughout the analysis.
Furthermore, with this choice of $\bfg$ and $\bfh$, the prescribed solution is independent of the choice of slip length $\varepsilon$.
The boundary condition is imposed on $\Gamma$ by the method \eqref{eq:oseen-discrete-unfitted} introduced in Section~\ref{ssec:weak_imposition_bcs_oseen}.

The background mesh $\widehat{\mcT}_h$  covers a rectangular domain $[-1.6,1.6]^2$ which is rotated by the angle $\theta = 0.25 \pi$ around the origin to ensure a non-trivial computational mesh $\mcT_h$, see Figure~\ref{fig:computational-domain-of-box-flow}.
It consists of $N \times N$ equally sized square elements equipped with either bi-linear $\mcQ^1$ or bi-quadratic $\mcQ^2$ equal-order approximations
for velocity and pressure, i.e. $\mcV^k_{h} \times \mcQ^k_h$ where $k \in \{ 1,2 \} $.
As a consequence the element length is $h=3.2/N$.
The parameters for the Oseen problem are chosen as $\nu=1.0$ and $\sigma=1.0$.
In the case of convection dominated flow (i.e. for small $\nu$), the tangential components disappear as they scale with~$\nu$,
and effectively the same formulation as presented and studied in \cite{MassingSchottWall2016_CMAME_Arxiv_submit} is retained.
For this reason, we focus on investigating viscous flows here.
As the pressure is only defined up to a constant, the constant pressure mode is filtered out by imposing $\int_{\Omega} p_h dx = 0$
during the solution process.

\begin{figure}[ht]
  \centering
  \subfloat{\includegraphics[trim=5 5 5 5, clip, width=0.4\textwidth]{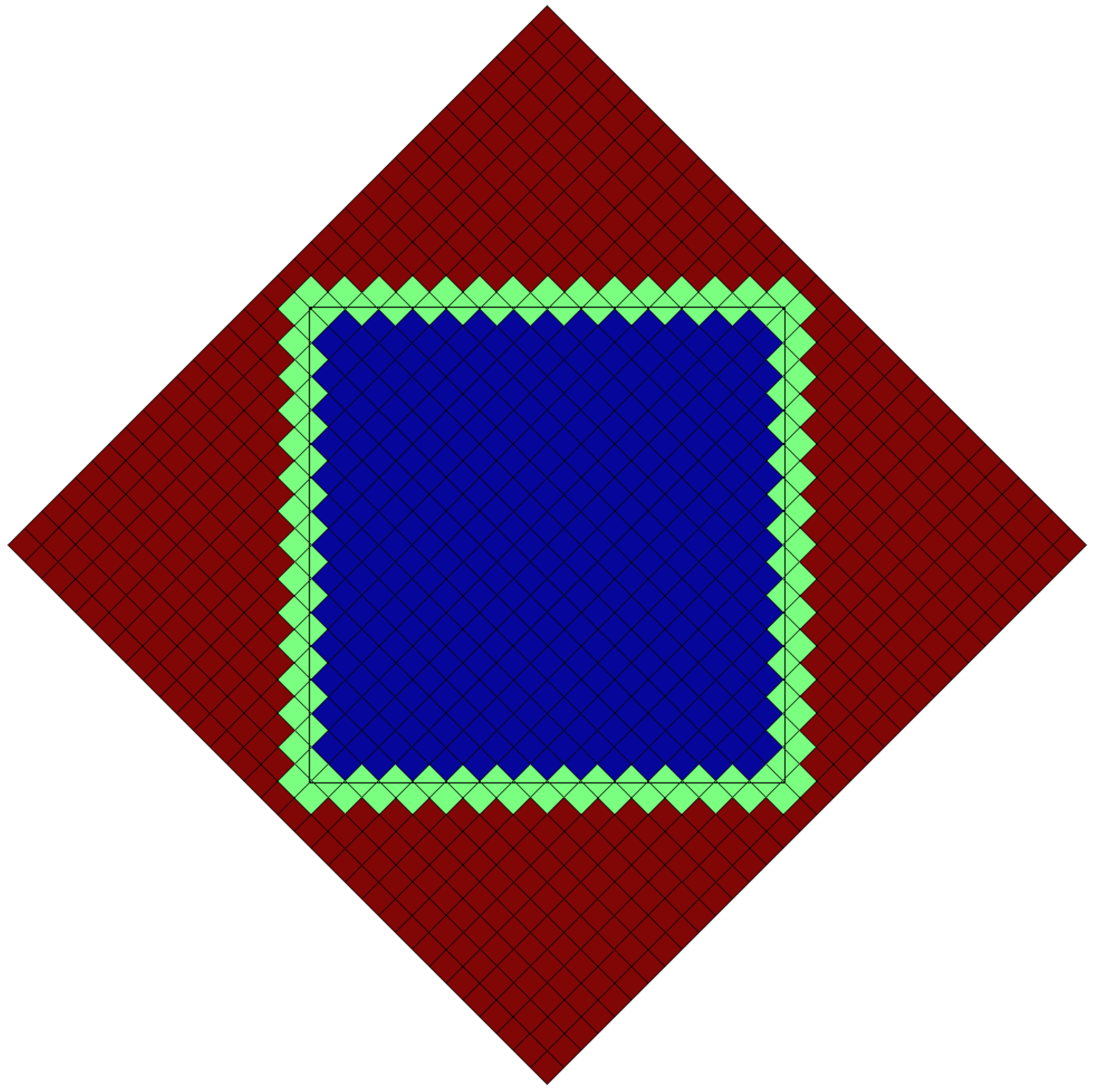}}
  \caption{This figure shows the background mesh $\widehat{\mcT}_h$ rotated at $0.25$ radians. The active computational mesh~$\mcT_h$ is covered by the green and blue domains. In the green domain, the elements are cut and ghost-penalties are applied on their facets.
The red domain indicates the inactive part of the background mesh.}
  \label{fig:computational-domain-of-box-flow}
\end{figure}

\subsection{Mesh Refinement Study}
To verify the results from the \apriori~error analysis, mesh refinement studies are conducted for both linear and quadratic interpolations.
In the case of the linear $\mcQ^1$ elements, a series of mesh sizes are chosen with $N \in [ 8; 512]$ and for the quadratic $\mcQ^2$ case $N \in [ 8; 224]$. To demonstrate the robustness to the choice of slip length, the convergence studies are conducted with a choice of three different $\varepsilon \in [10^{-10}; 1.0 ;10^{10} ]$.

In Figure~\ref{fig:mesh-convergence-study-split-hex8} the $L^2$-errors of $\bfu_h$, $\nabla \bfu_h$ and $p_h$ are presented for the linear $\mcQ^1$ approximations evaluated in $\Omega$ and on $\Gamma$.
The convergence plots verify the optimal convergence $\mathcal{O}(h^2)$ for the velocity $L^2$-error (see Theorem~\ref{thm:apriori-estimate-L2}) for all choices of $\varepsilon$.
Furthermore, the errors for the different choices of the various slip lengths remain of comparable size throughout the mesh refinement.
The convergence of the velocity gradient and pressure, as seen from the error at the boundary, converges with at least first order, in accordance with theory
stated in Theorem~\ref{thm:apriori-estimate}. 



\begin{figure}[ht!]
  \centering
  \subfloat{\includegraphics[trim=1 1 44 32, clip, width=0.33\textwidth]{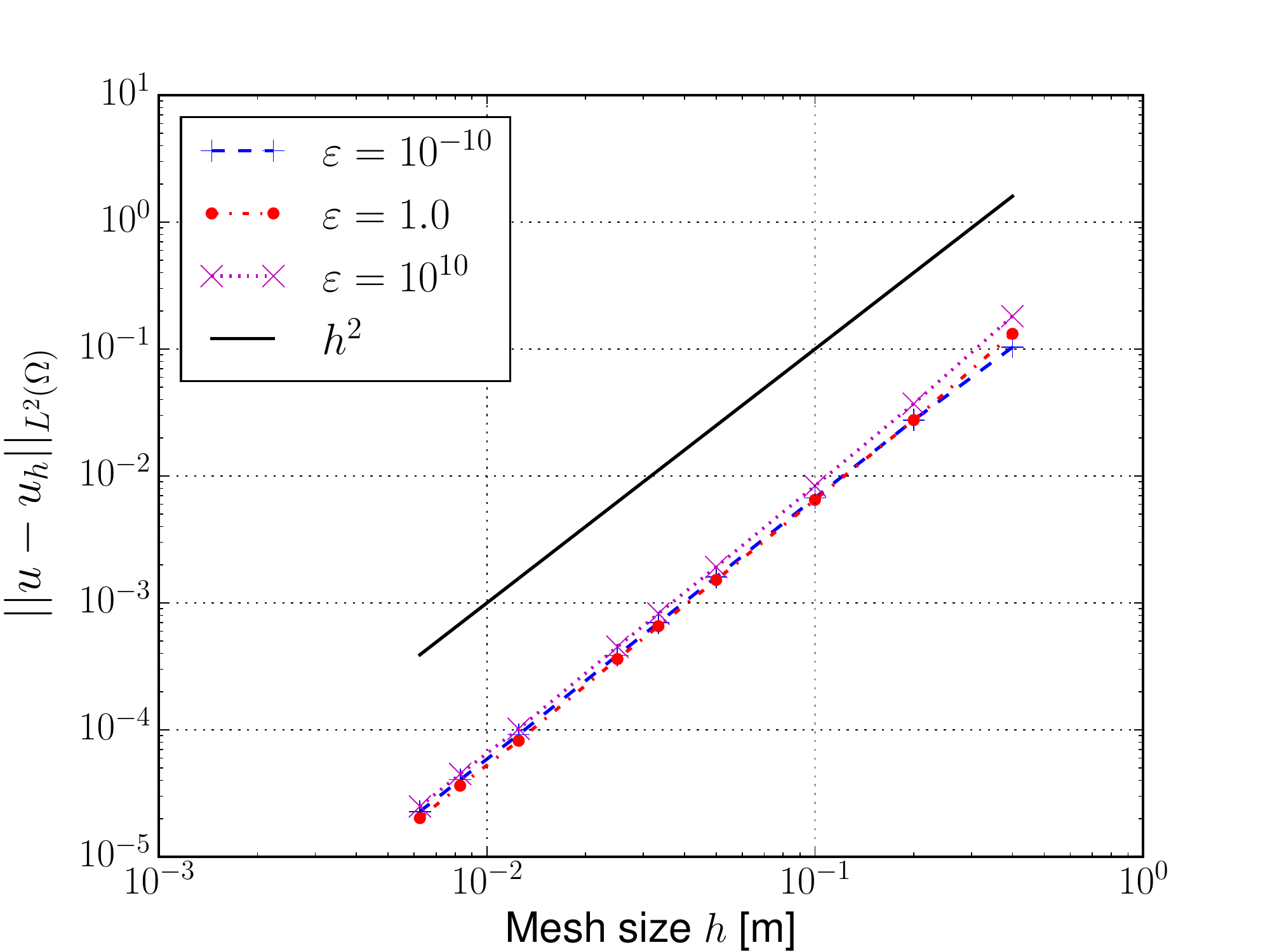}}
  \subfloat{\includegraphics[trim=1 1 44 32, clip, width=0.33\textwidth]{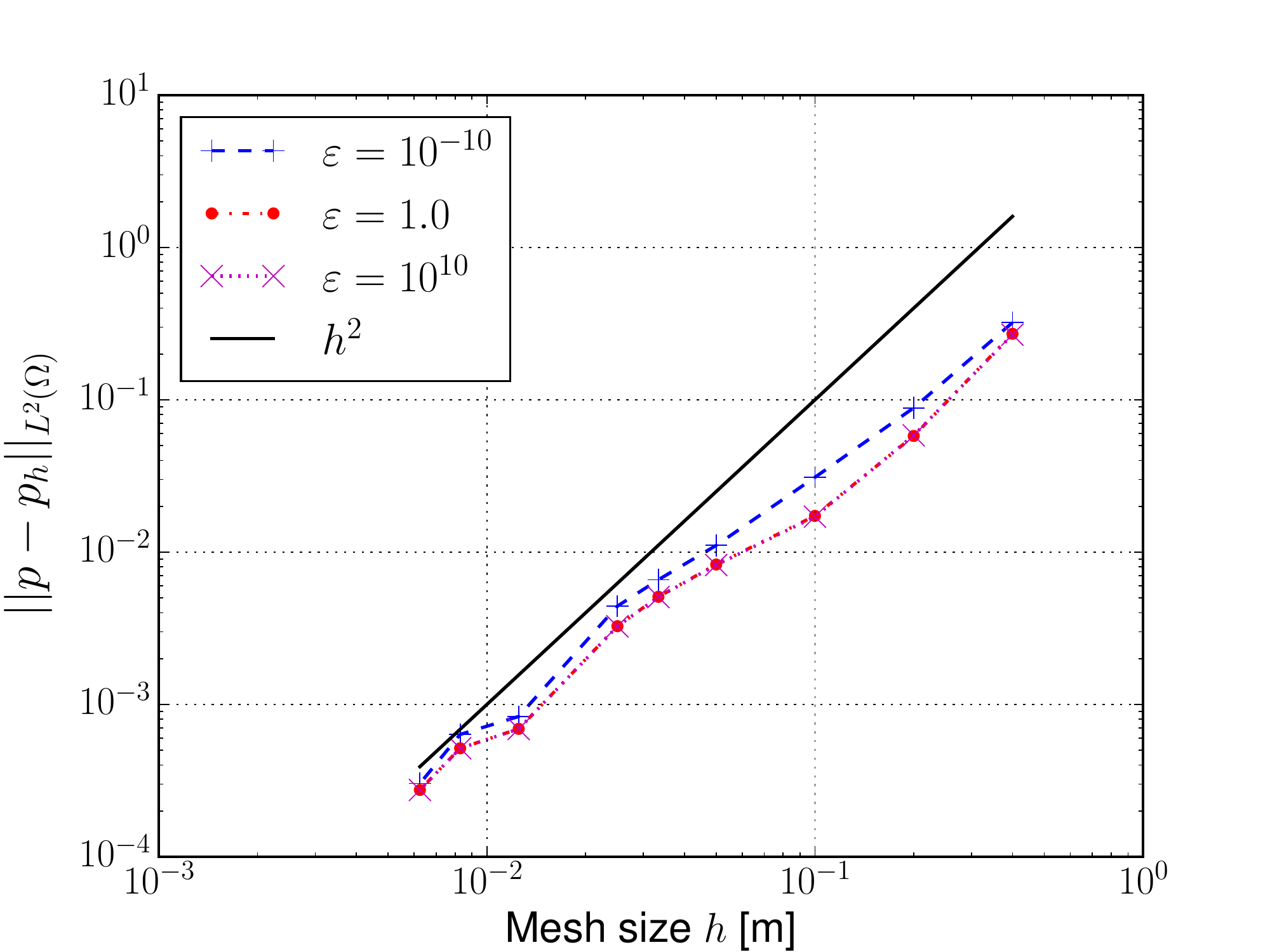}}
  \subfloat{\includegraphics[trim=1 1 44 32, clip, width=0.33\textwidth]{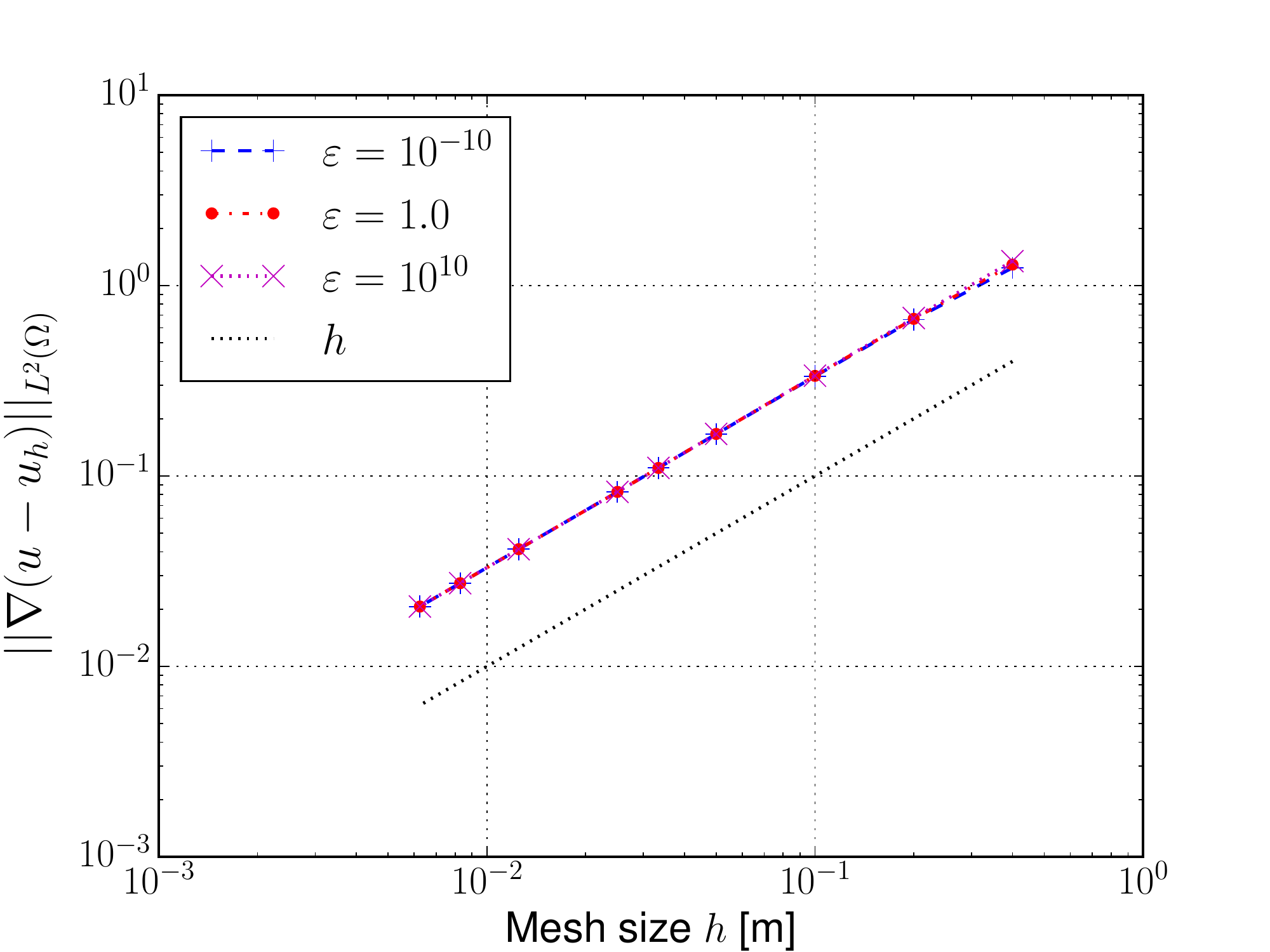}}\\
  \subfloat{\includegraphics[trim=1 1 44 32, clip, width=0.33\textwidth]{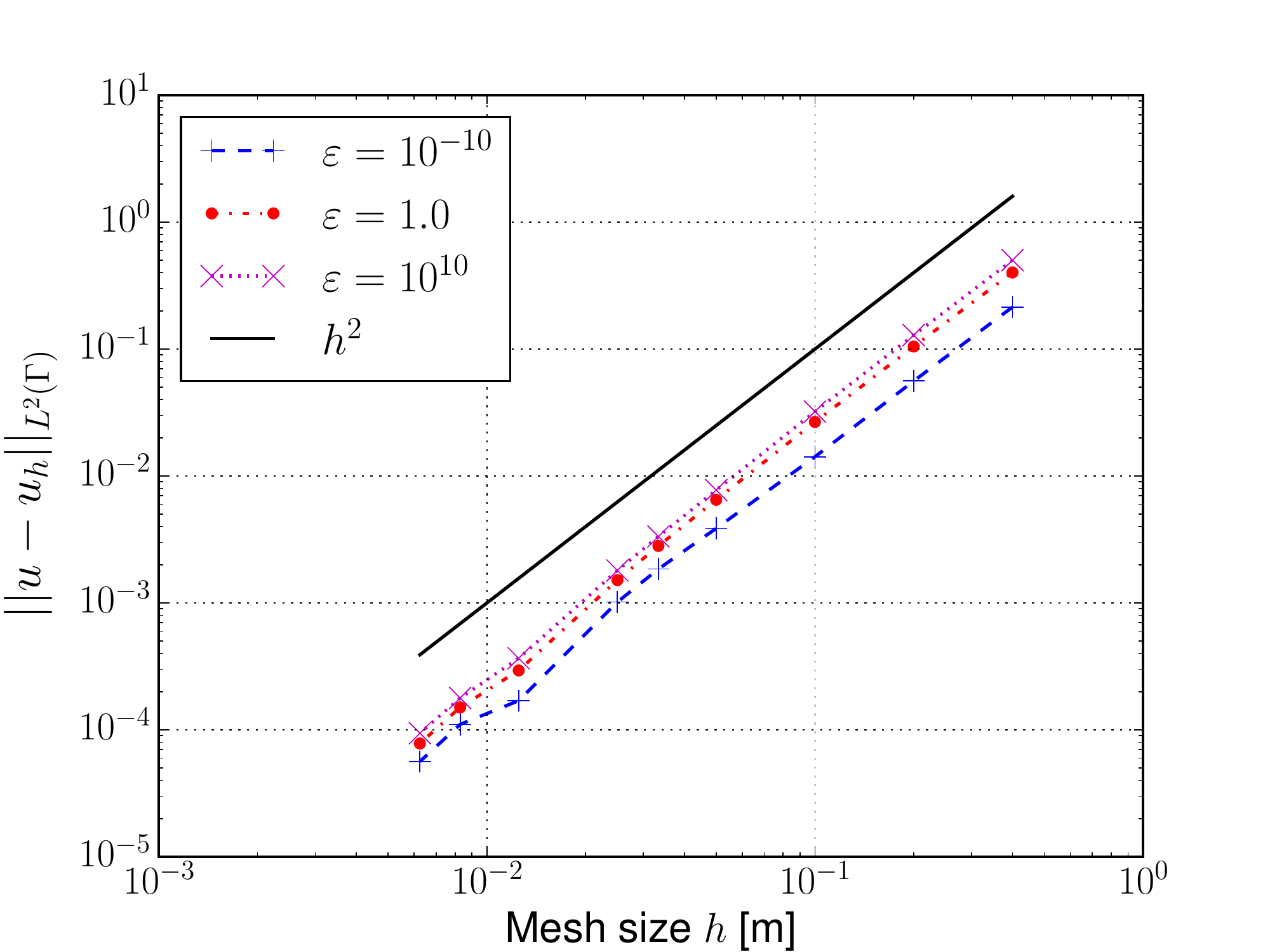}}
  \subfloat{\includegraphics[trim=1 1 44 32, clip, width=0.33\textwidth]{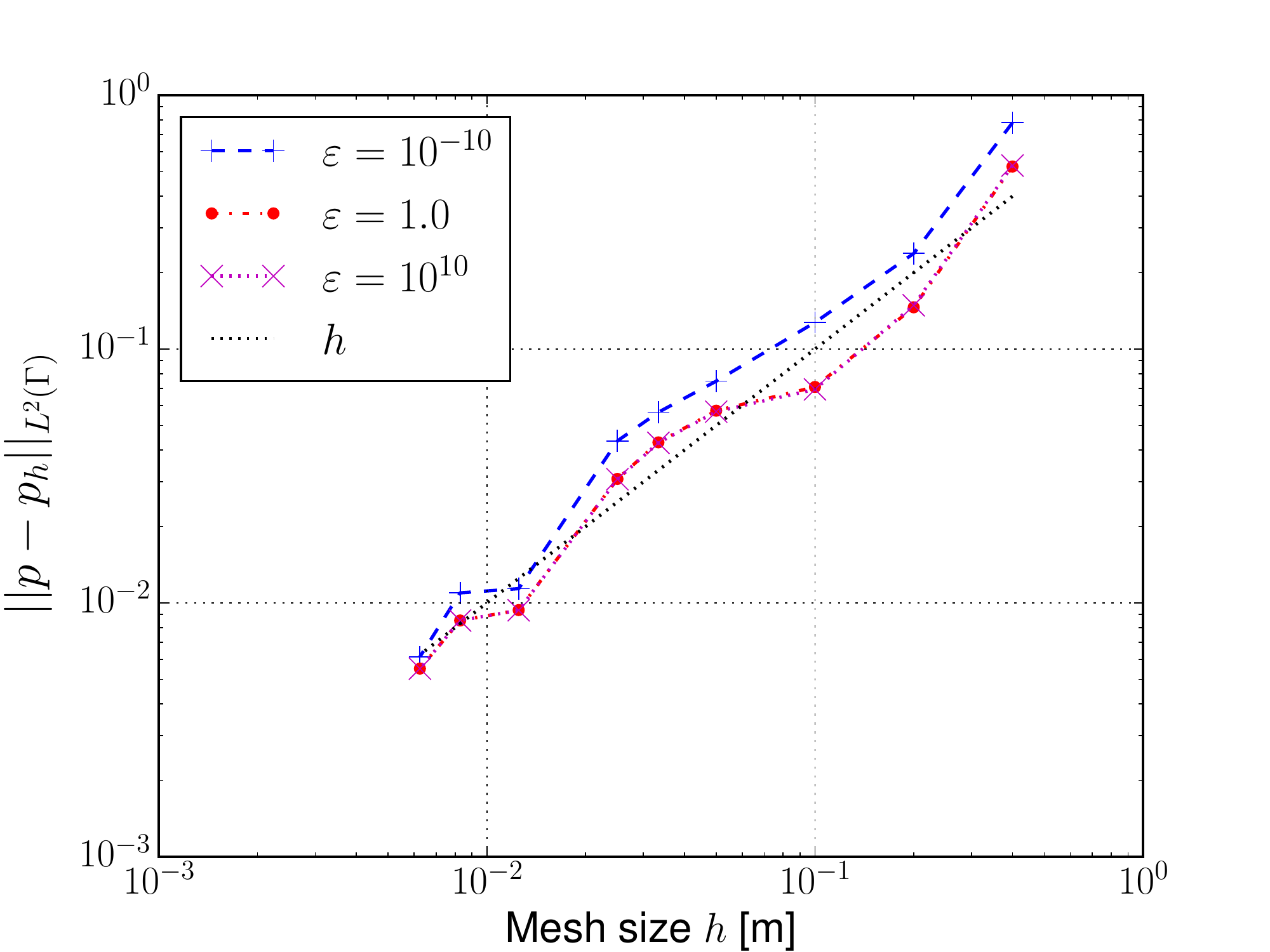}}
  \subfloat{\includegraphics[trim=1 1 44 32, clip, width=0.33\textwidth]{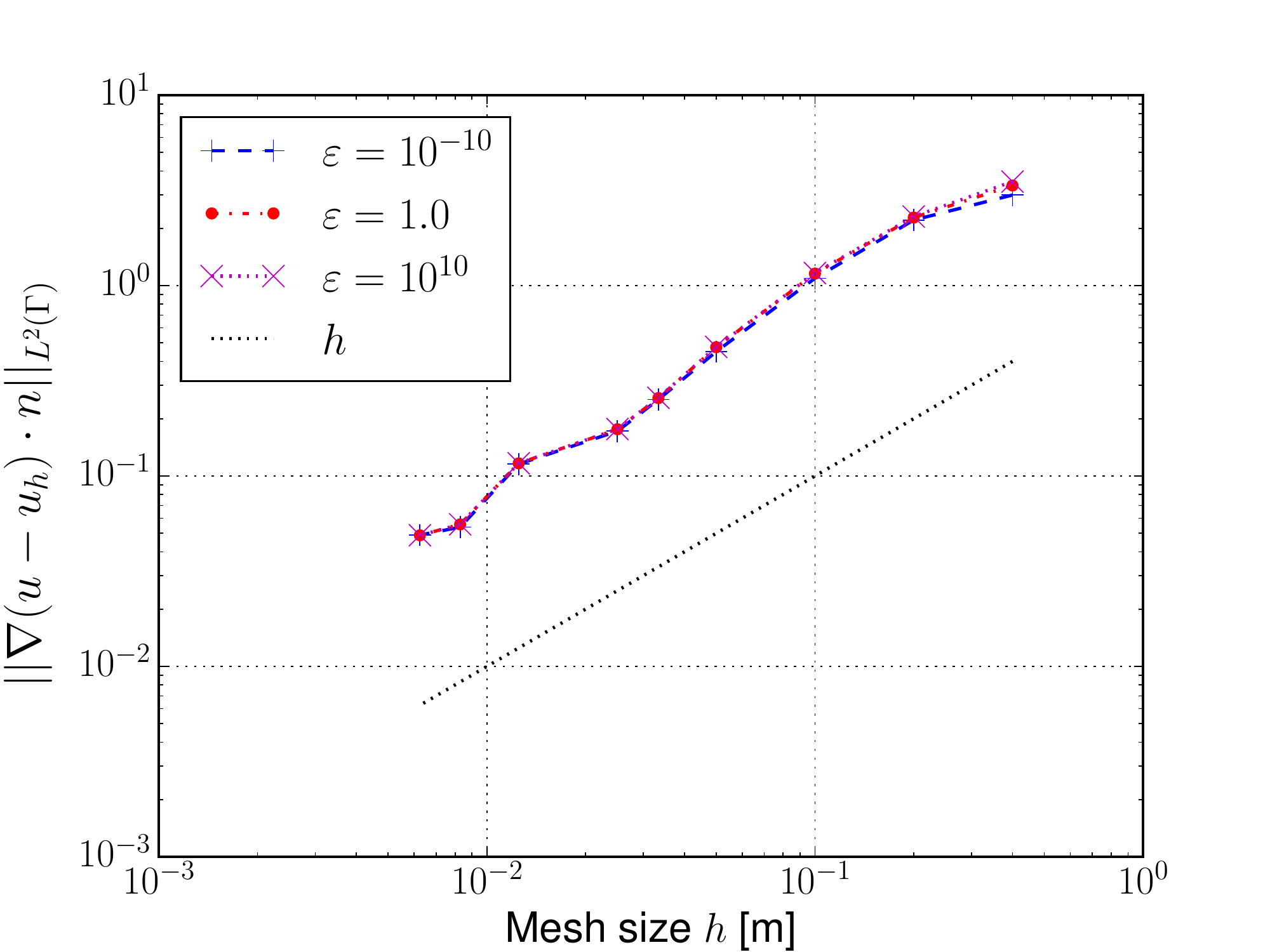}}
  \caption{Error convergence study for an adjoint-consistent Nitsche's method with $\mcQ^1$ elements: bulk errors (top row) and boundary errors (bottom row) for velocity, pressure and velocity gradient (from left to right).}
  \label{fig:mesh-convergence-study-split-hex8}
\end{figure}

For the quadratic $\mcQ^2$ interpolation case the same error norms as for the linear case are studied, as shown in Figure~\ref{fig:mesh-convergence-study-split-hex27}.
Similar conclusions can be drawn as for the linear case. An optimal convergence $\mathcal{O}(h^3)$ for the velocity $L^2$-error and an order of at least $\mathcal{O}(h^2)$ for the velocity gradient and pressure are observed for all~$\varepsilon$, in accordance with theory. This verifies the applicability of the method for higher order elements as well. Also notable is that the quadratically interpolated elements demonstrate less sensitivity to the cut position than the linear elements.




\begin{figure}[ht!]
  \centering
  \subfloat{\includegraphics[trim=1 1 44 32, clip, width=0.33\textwidth]{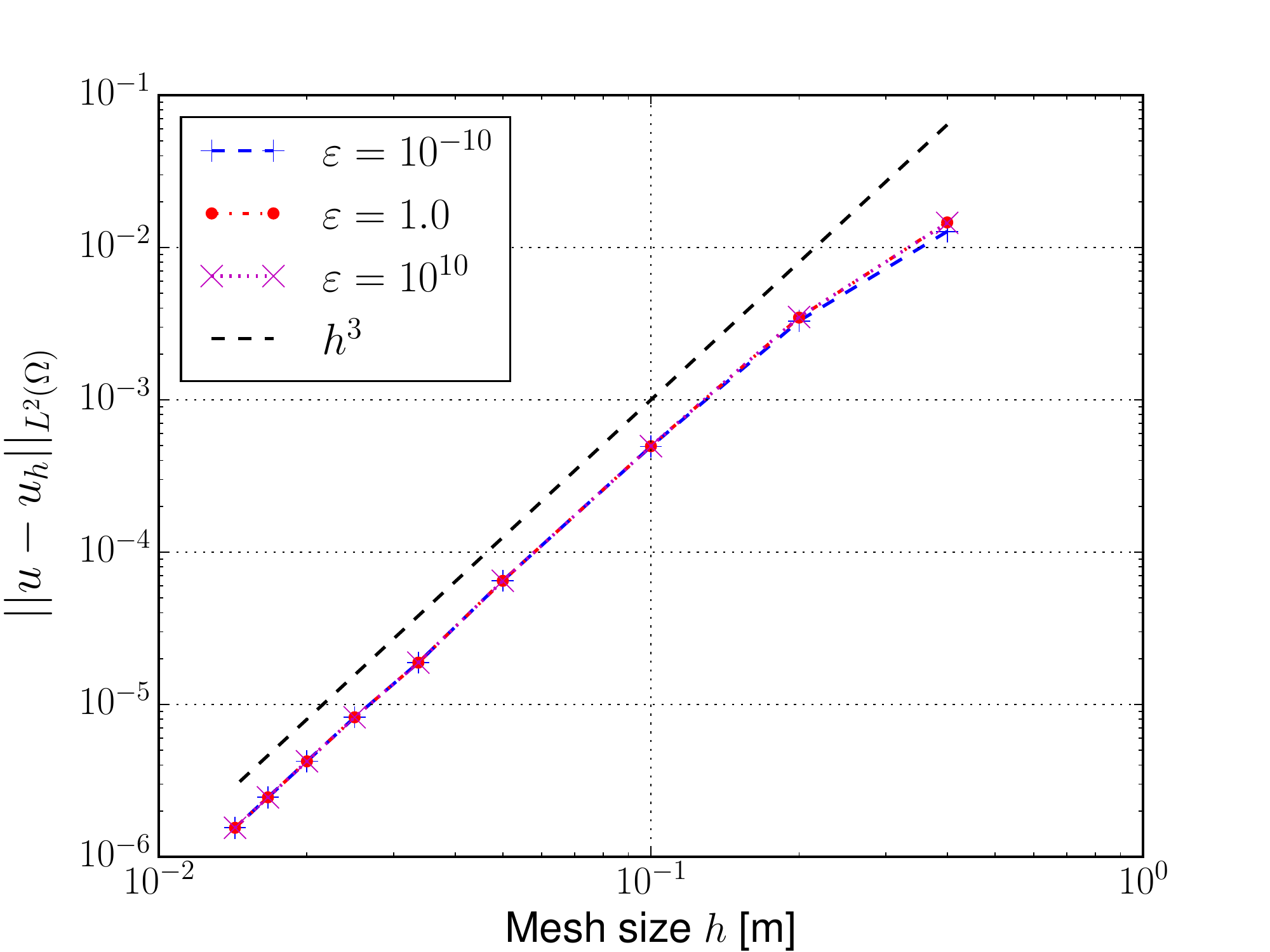}}
  \subfloat{\includegraphics[trim=1 1 44 32, clip, width=0.33\textwidth]{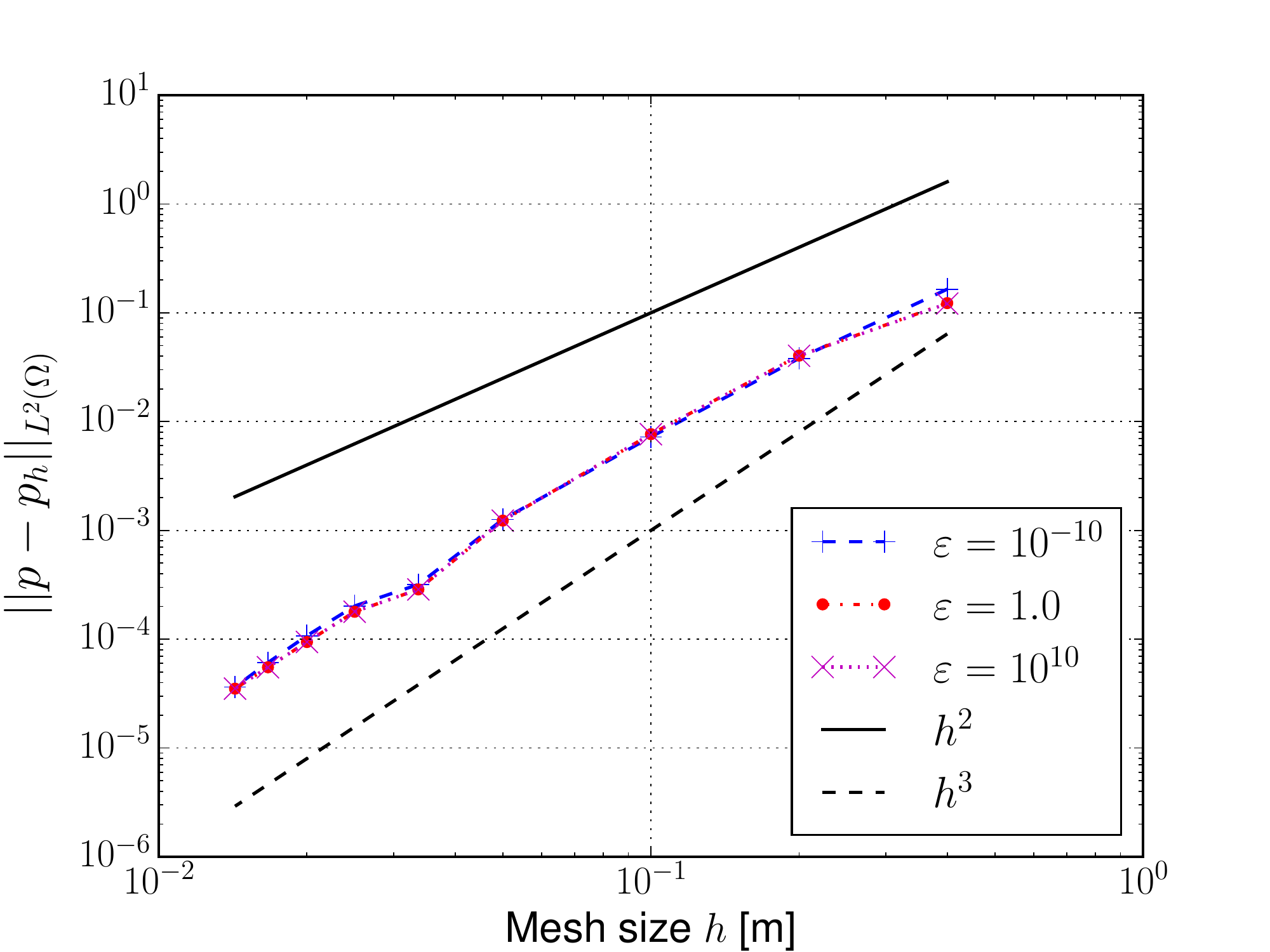}}
  \subfloat{\includegraphics[trim=1 1 44 32, clip, width=0.33\textwidth]{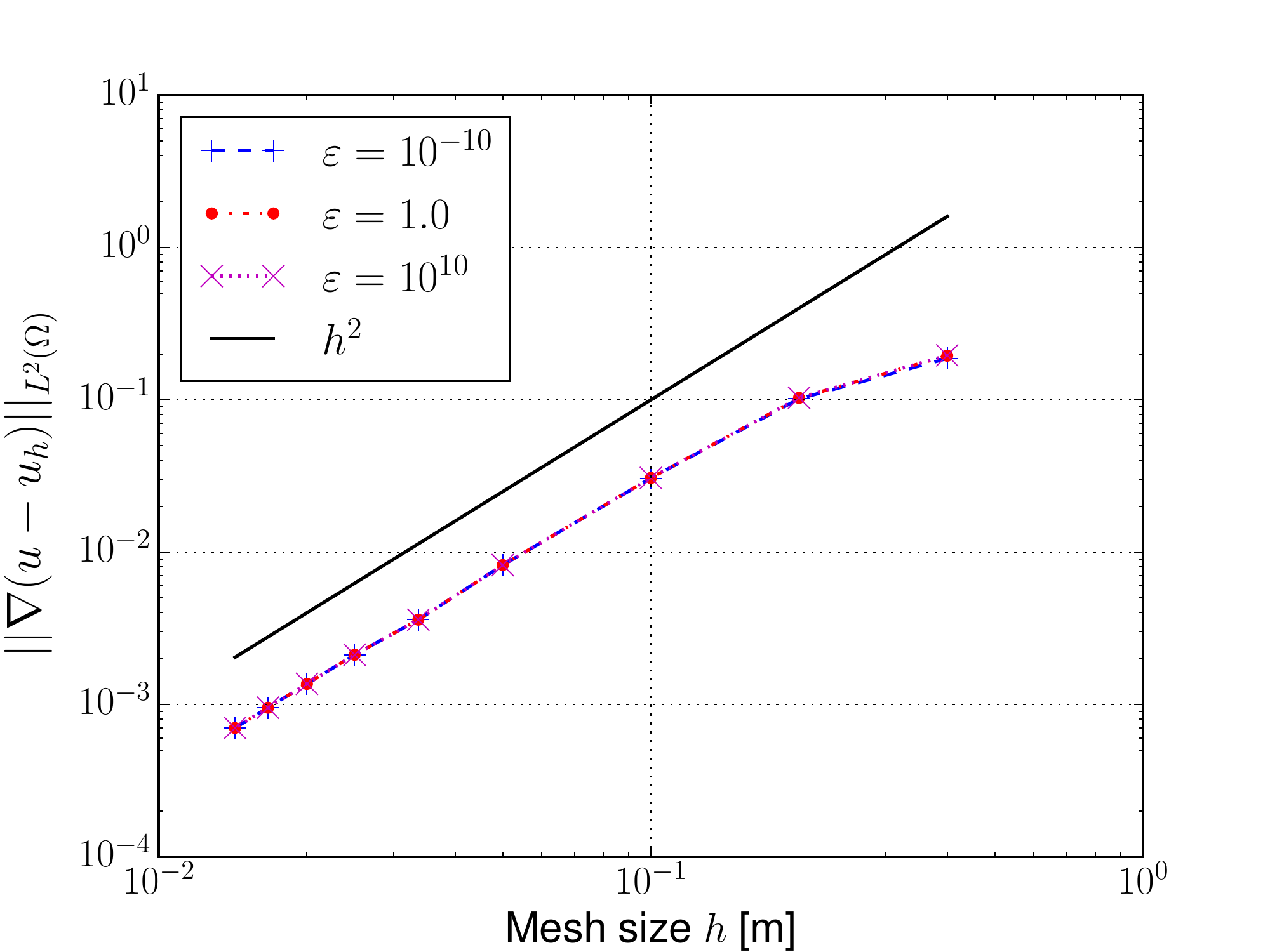}}\\
  \subfloat{\includegraphics[trim=1 1 44 32, clip, width=0.33\textwidth]{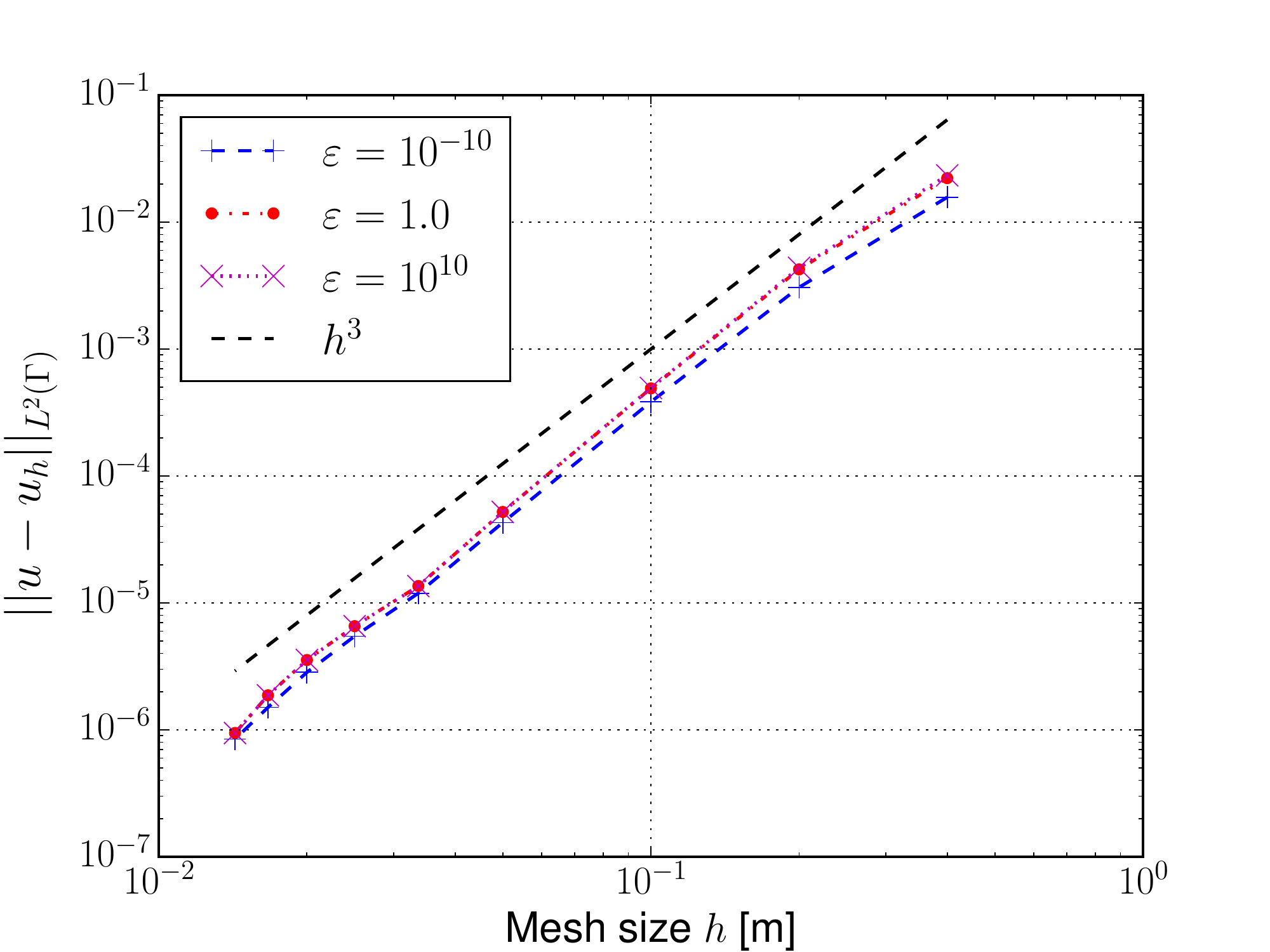}}
  \subfloat{\includegraphics[trim=1 1 44 32, clip, width=0.33\textwidth]{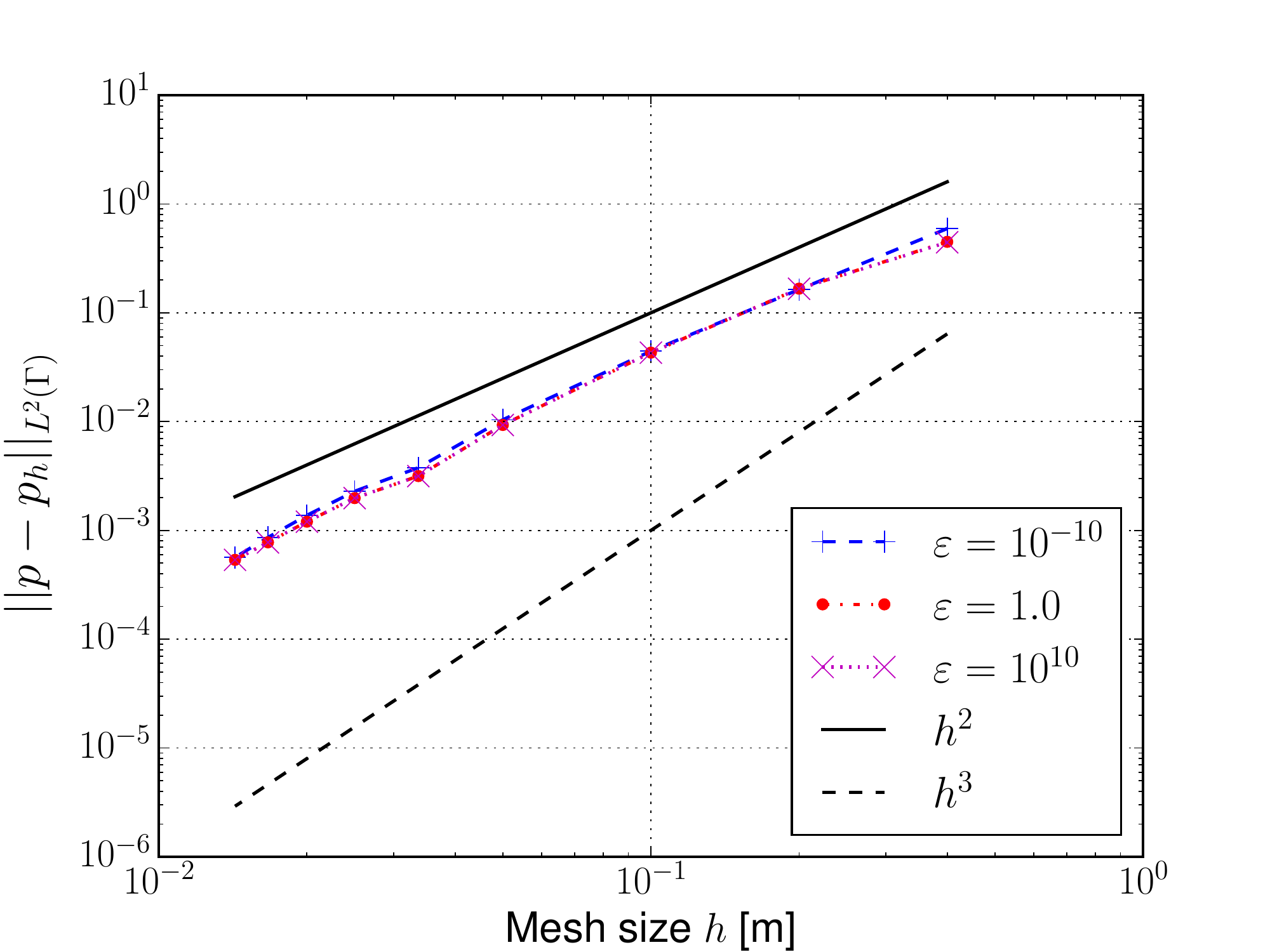}}
  \subfloat{\includegraphics[trim=1 1 44 32, clip, width=0.33\textwidth]{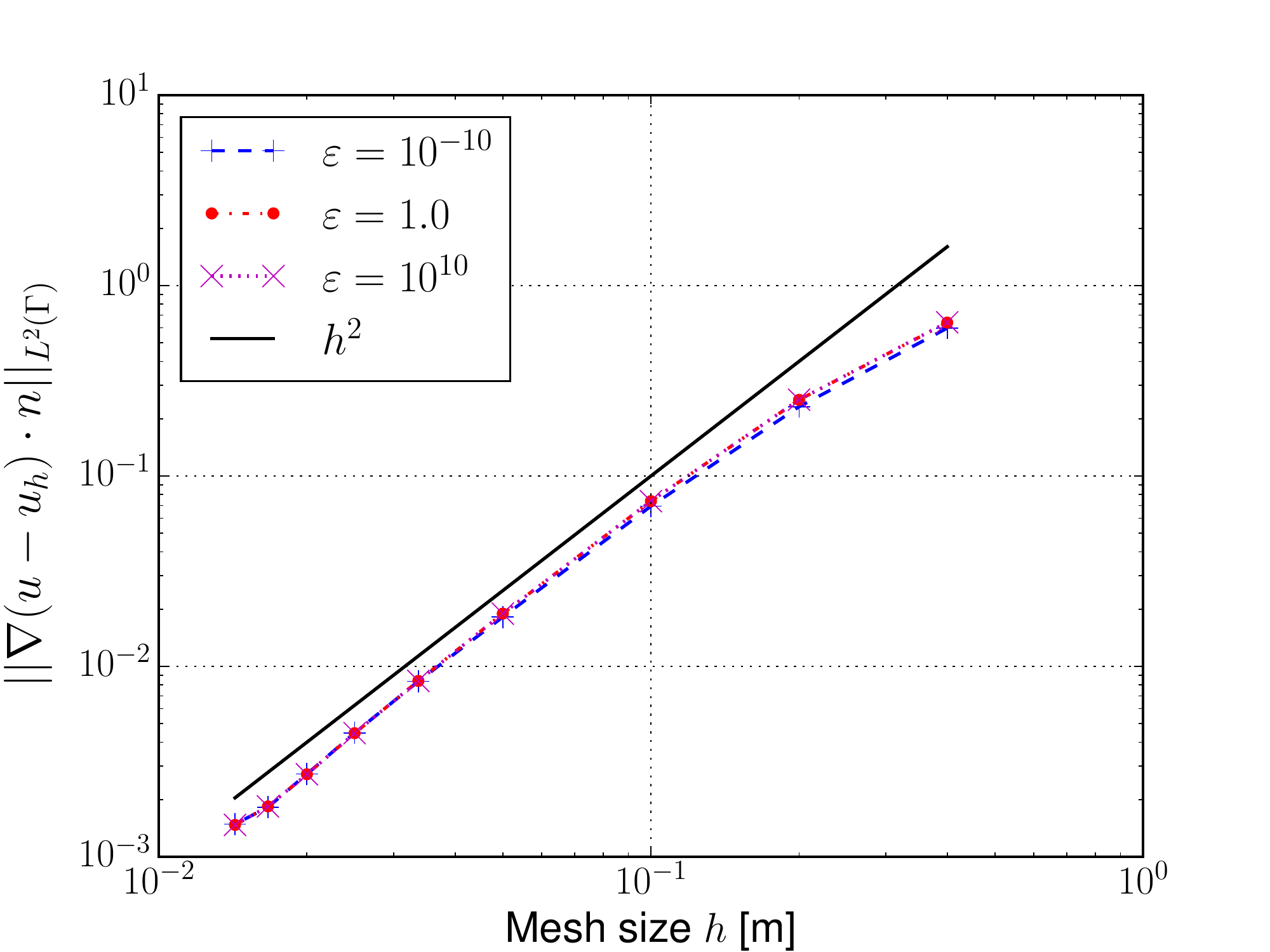}}
  \caption{Error convergence study for an adjoint-consistent Nitsche's method with $\mcQ^2$ elements: bulk errors (top row) and boundary errors (bottom row) for velocity, pressure and velocity gradient (from left to right).}
  \label{fig:mesh-convergence-study-split-hex27}
\end{figure}


%
%


\subsection{Nitsche Stabilization Parameter Study}
To observe the behavior of the method for different choices of the Nitsche parameter, a sensitivity study is conducted. 
For the adjoint consistent case the study is done with $\mcQ^1$ elements and for the adjoint inconsistent case for both $\mcQ^1$ and $\mcQ^2$ elements. In the case of linear elements a $512 \times 512$ mesh is used and for the quadratic case a $224 \times 224$ mesh.
Both the tangential and normal stabilization parameters are varied at the same time, i.e. $\gamma=\gamma^t=\gamma^n$, and studies are conducted for three different choices of slip length $\varepsilon = \{ 10^{-10}; 1.0; 10^{10} \}$.

In Figure~\ref{fig:nitstab-convergence-study-separate-hex8} the results for the adjoint-consistent formulation ($\zeta_u = 1$) can be seen. As is verified from the inf-sup stability analysis in Section~\ref{sec:stability-properties}, the error diverges as $1/\gamma$ becomes too small, and in this case the stability limit is reached at around $1/\gamma \approx  4.0$. Notable from the figure is the different behavior of the velocity error at the interface for varying~$\varepsilon$. In the case of $\varepsilon = 10^{-10}$ the error is slowly decreasing for a choice of larger $1/\gamma$, whereas for the cases with $\varepsilon = 1.0$ and $\varepsilon = 10^{10}$ the error is strongly increasing.
This discrepancy can be explained by the fact that in the latter cases the penalty effect on the normal constraint is way stronger
than for the tangential part and thus allows for worsening of the imposition of the tangential condition
while better enforcing the normal constraint for a large penalty parameter $1/\gamma$.
In contrast, in the former case the velocity is enforced in both normal and tangential direction,
thus imposing the velocity well at the boundary.
However, note that at the same time larger~$1/\gamma$ influences the velocity gradient and pressure at the boundary negatively,
as expected.

\begin{figure}[ht!]
  \centering
  \subfloat{\includegraphics[trim=1 1 44 32, clip, width=0.33\textwidth]{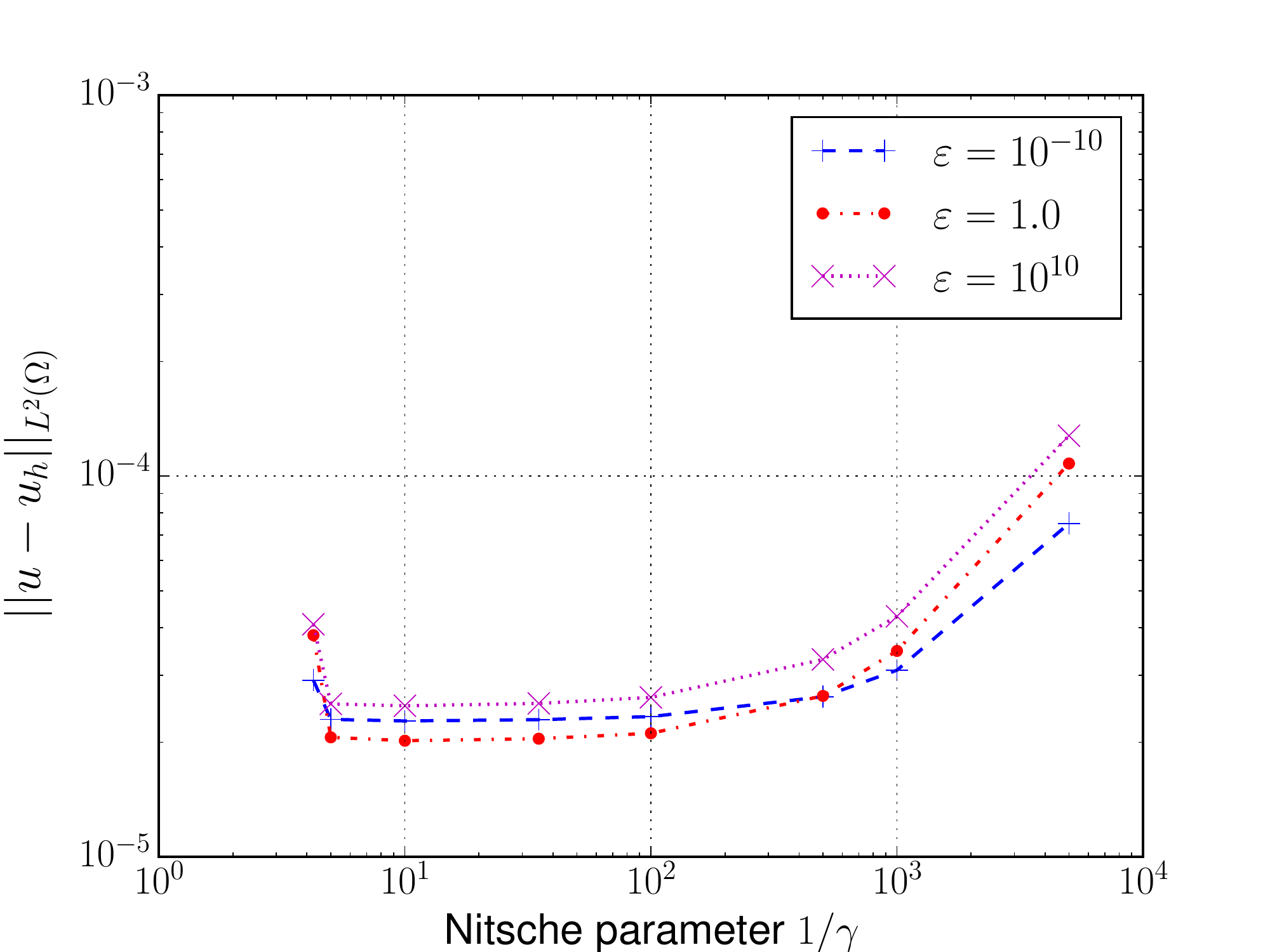}}
  \subfloat{\includegraphics[trim=1 1 44 32, clip, width=0.33\textwidth]{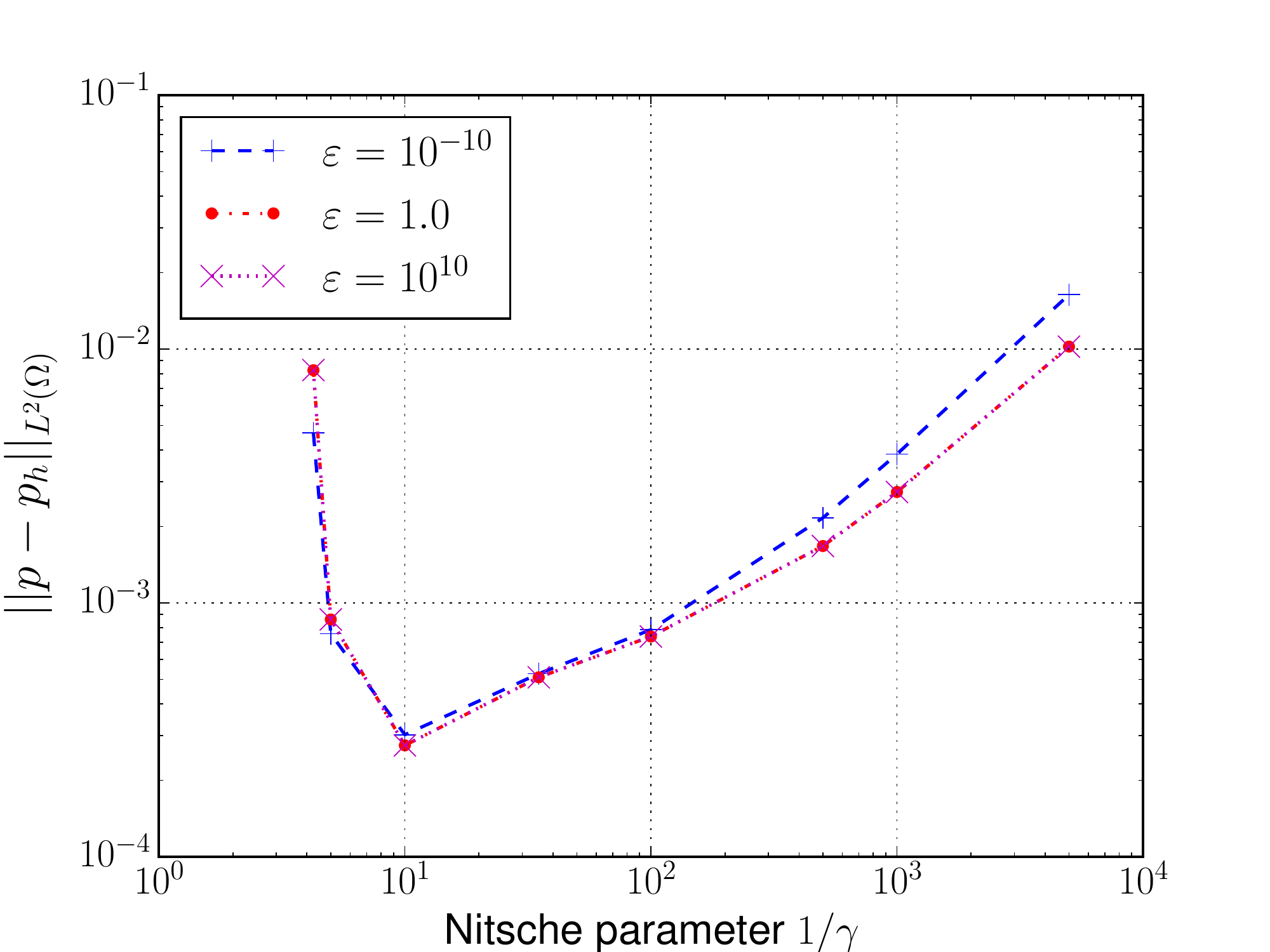}}
  \subfloat{\includegraphics[trim=1 1 44 32, clip, width=0.33\textwidth]{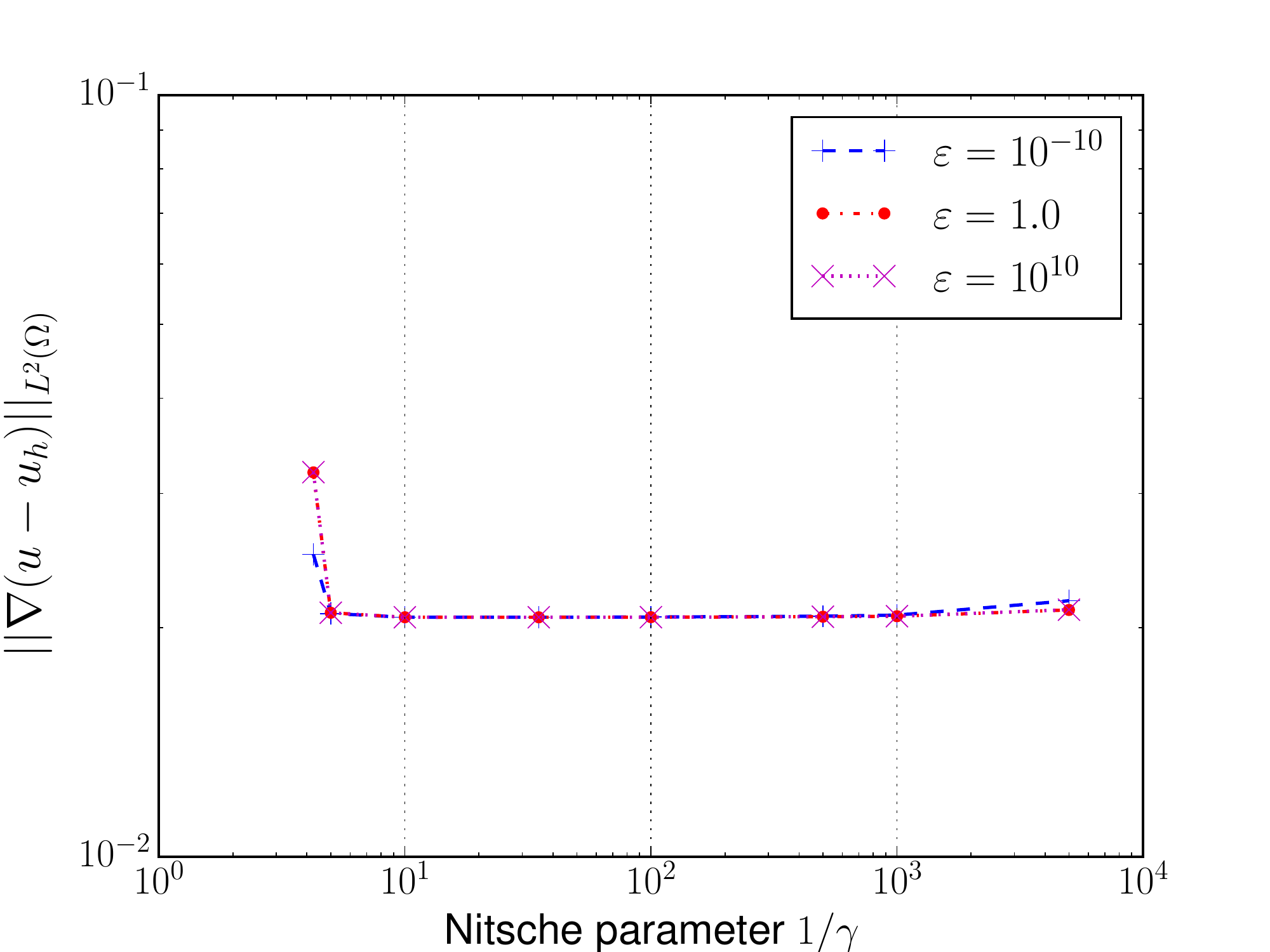}}\\
  \subfloat{\includegraphics[trim=1 1 44 32, clip, width=0.33\textwidth]{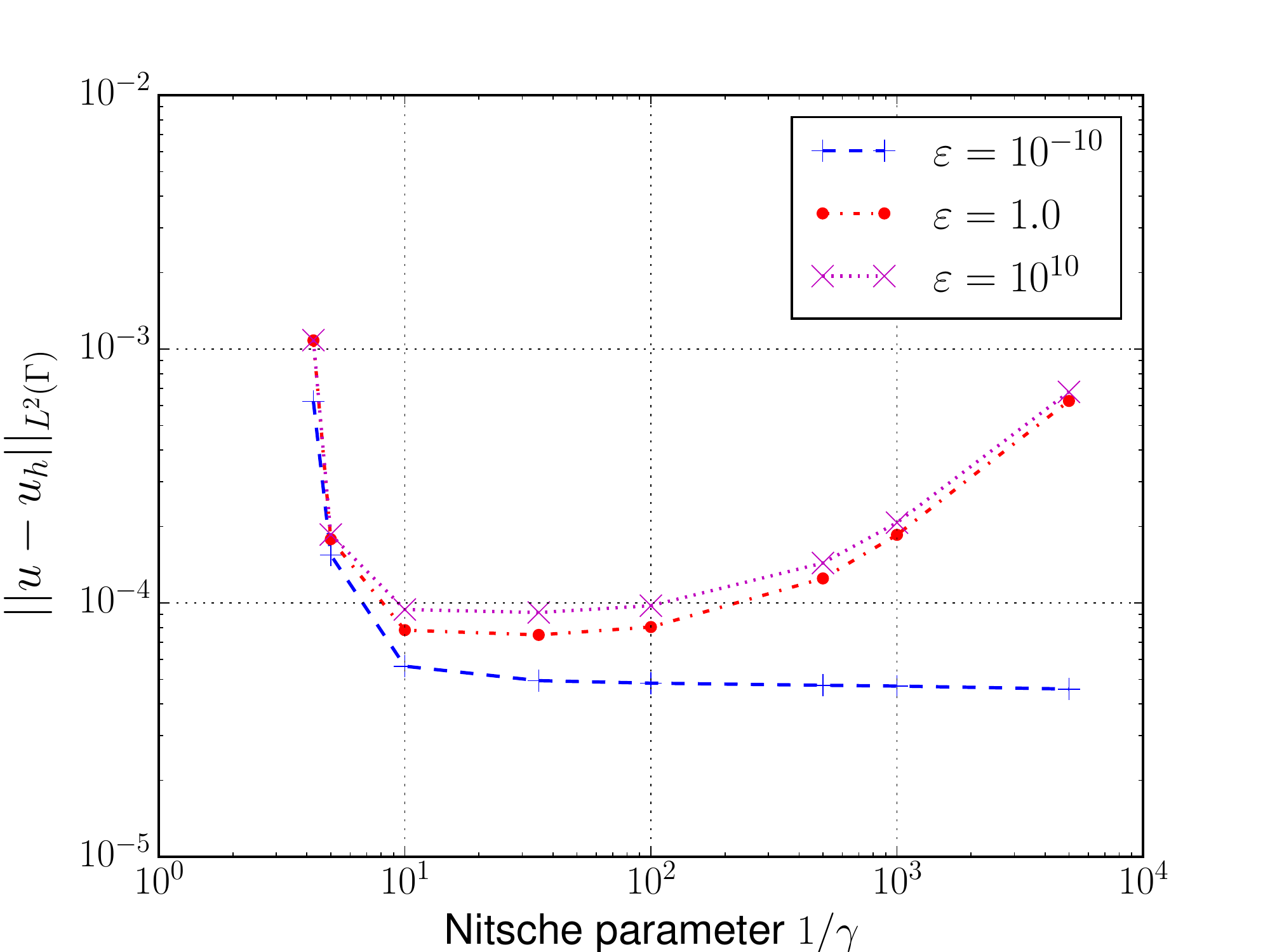}}
  \subfloat{\includegraphics[trim=1 1 44 32, clip, width=0.33\textwidth]{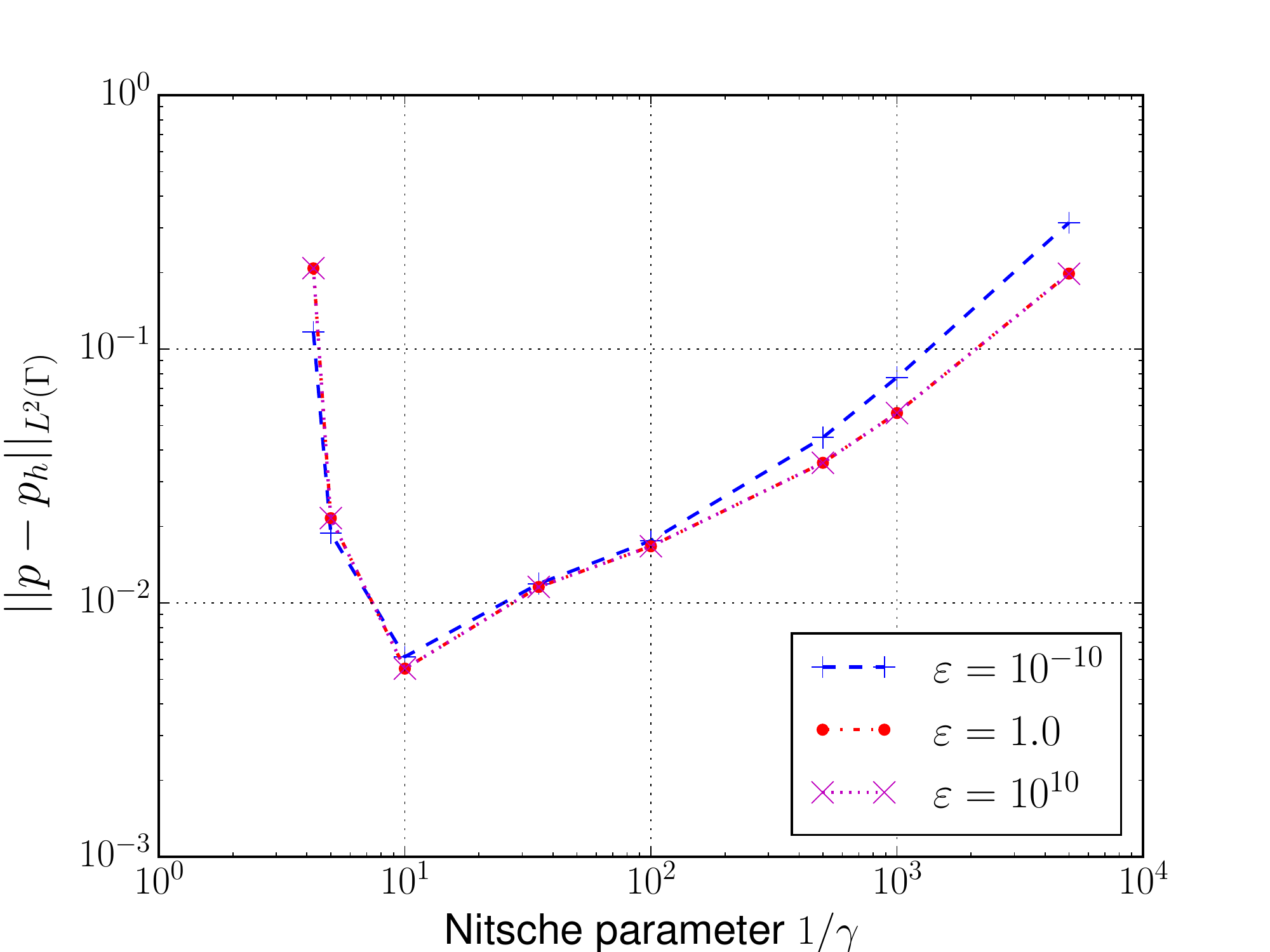}}
  \subfloat{\includegraphics[trim=1 1 44 32, clip, width=0.33\textwidth]{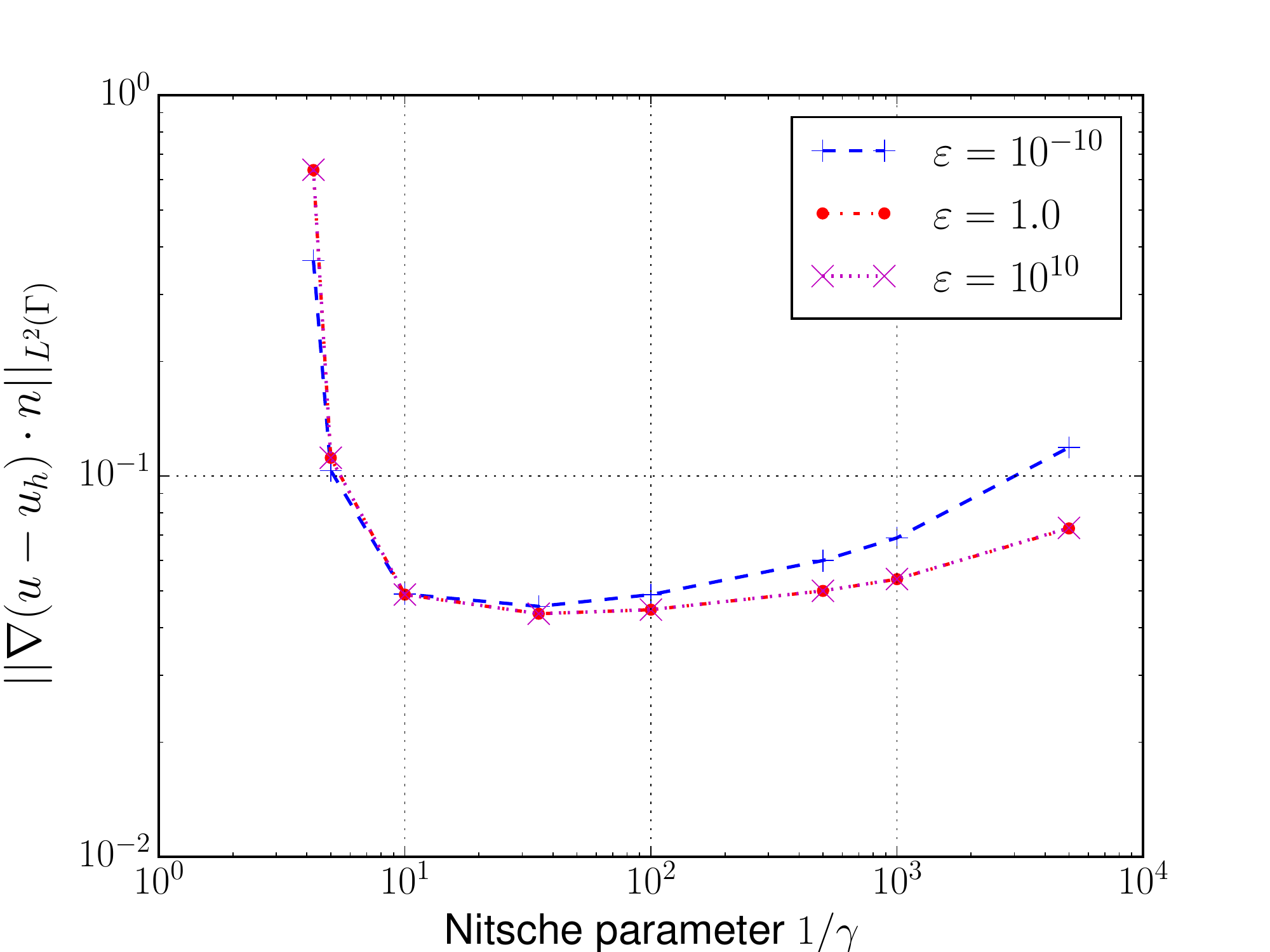}}
  \caption{Nitsche stabilization parameter study for an adjoint-consistent Nitsche's method with $\mcQ^1$ elements: bulk errors (top row) and boundary errors (bottom row) for velocity, pressure and velocity gradient (from left to right).}
  \label{fig:nitstab-convergence-study-separate-hex8}
\end{figure}

The introduced method also works for the adjoint-inconsistent case ($\zeta_u=-1$). This enables the choice of smaller values of $1/\gamma$, as predicted from theory. However, the choice of both a small $1/\gamma^t$ and a large $\varepsilon$ leads to conditioning issues, see Remark~\ref{remark:oseen_weak_formulation_large_eps}. 
This can be seen in Figures~\ref{fig:nitstab-skew-convergence-study-separate-hex8} and \ref{fig:nitstab-skew-convergence-study-separate-hex27} where for both the linear and quadratic case with $\varepsilon = 10^{10}$ small values of the stabilization parameter lead to the case where the used solver will not converge to the desired tolerance. This occurs for the linear case with $1/\gamma < 10^{-3}$ and for the quadratic case with $1/\gamma < 10^{-4}$. For the other choices of $\varepsilon$, the error remains stable for smaller choices of $1/\gamma$ confirming the results from theory which states that $1/\gamma>0$ is a sufficient choice for stability.
In the quadratic case, see Figure~\ref{fig:nitstab-skew-convergence-study-separate-hex27}, similar results as for the linear case are observed. However, the two observed minima for the pressure error are not present. This indicates that this phenomenon stems from the linear elements and is not an inherent property of the proposed method. 
Another observation is the fact that the error for the quadratic approximations remains stable longer before a large stabilization parameter $1/\gamma$ effects the solution negatively.

\begin{figure}[ht!]
  \centering
  \subfloat{\includegraphics[trim=1 1 44 32, clip, width=0.33\textwidth]{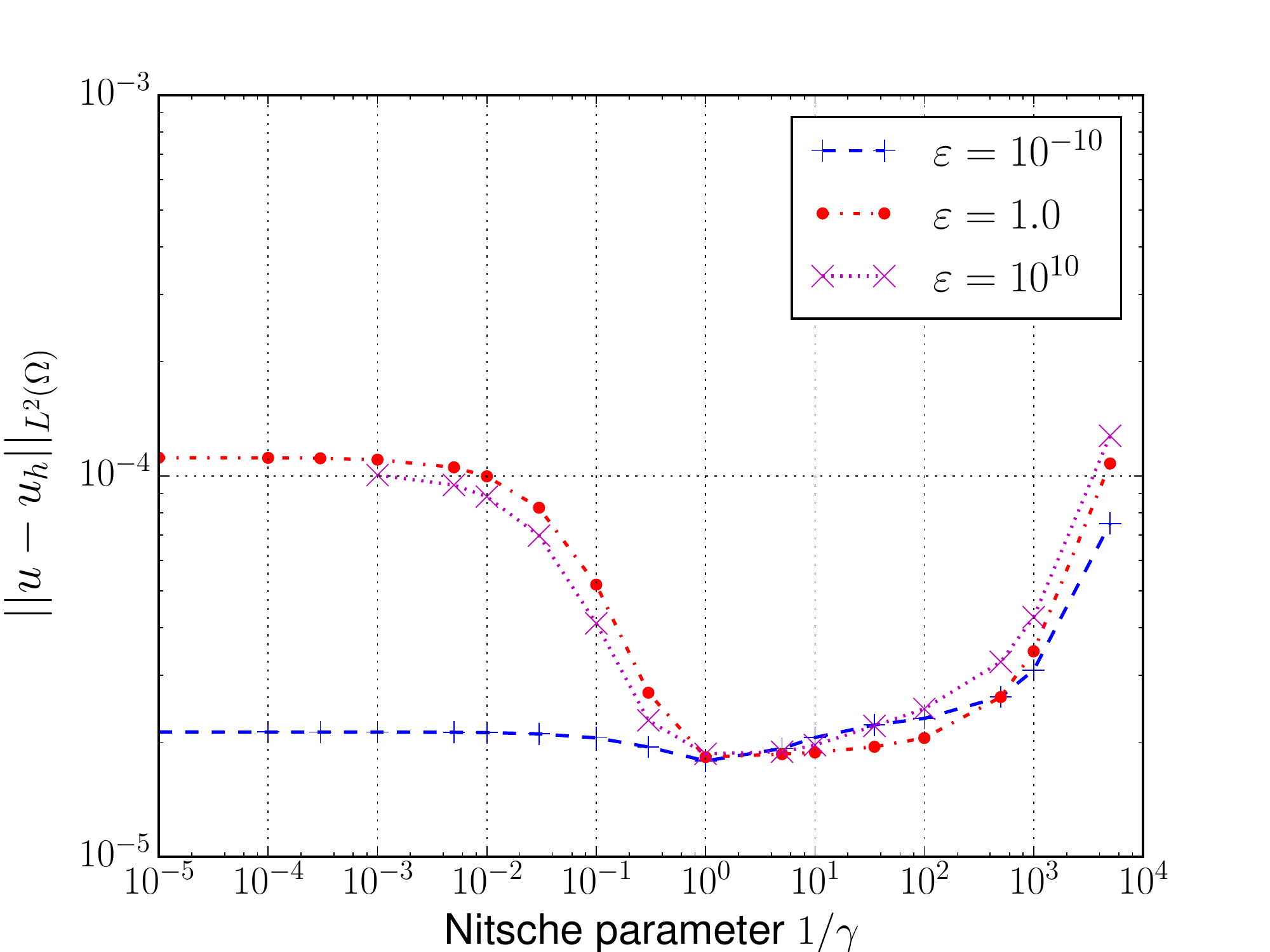}}
  \subfloat{\includegraphics[trim=1 1 44 32, clip, width=0.33\textwidth]{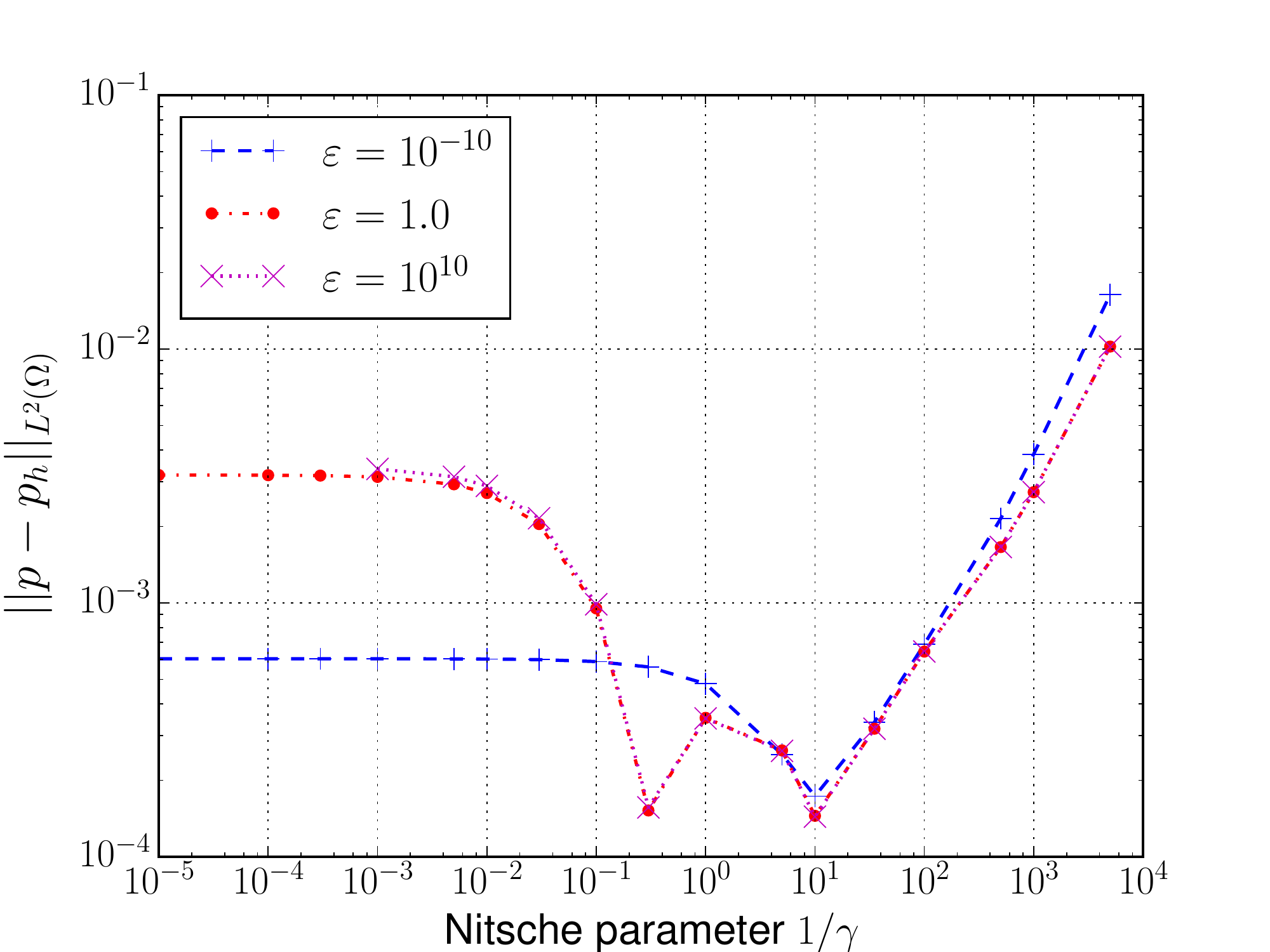}}
  \subfloat{\includegraphics[trim=1 1 44 32, clip, width=0.33\textwidth]{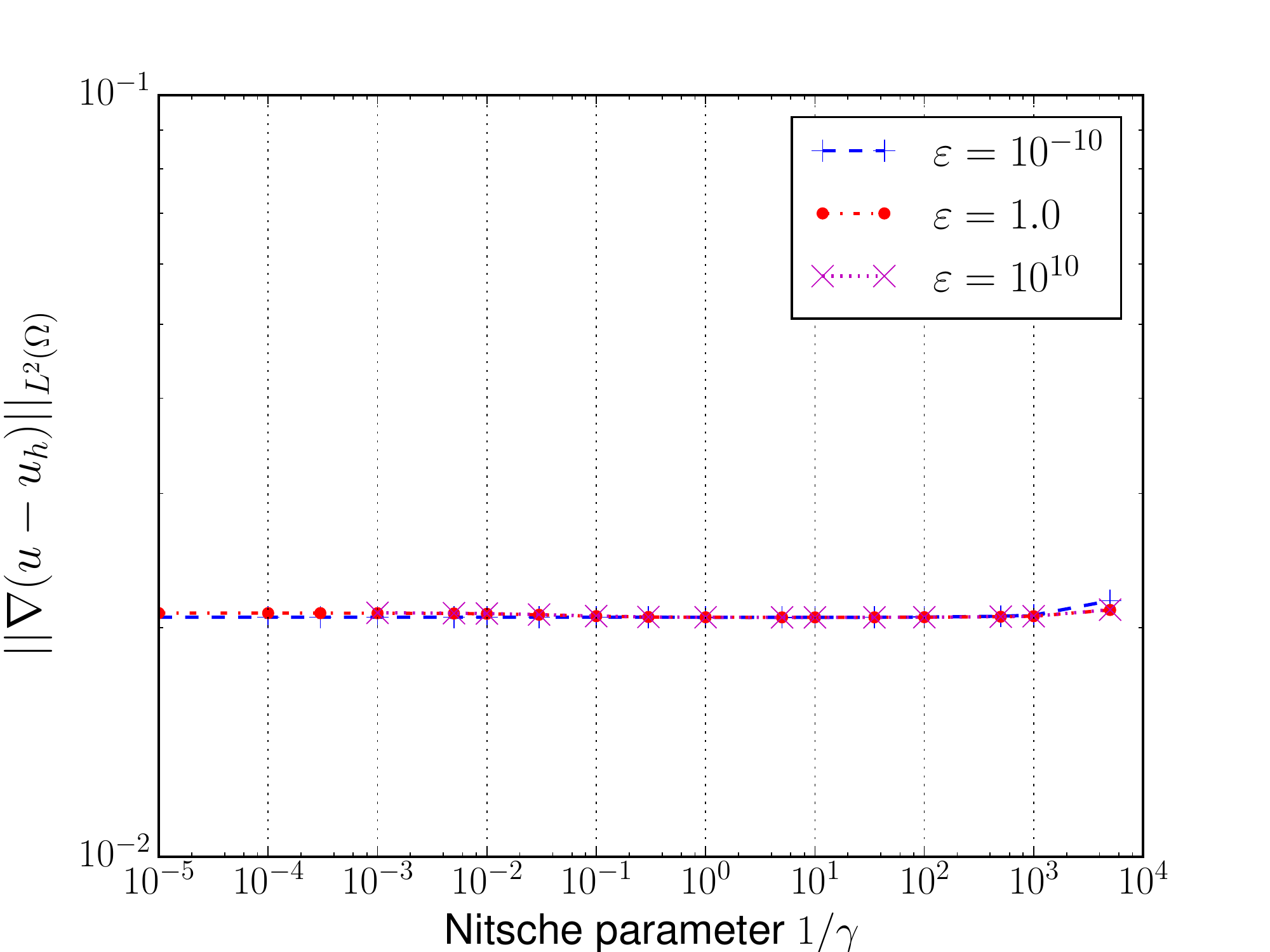}}\\
  \subfloat{\includegraphics[trim=1 1 44 32, clip, width=0.33\textwidth]{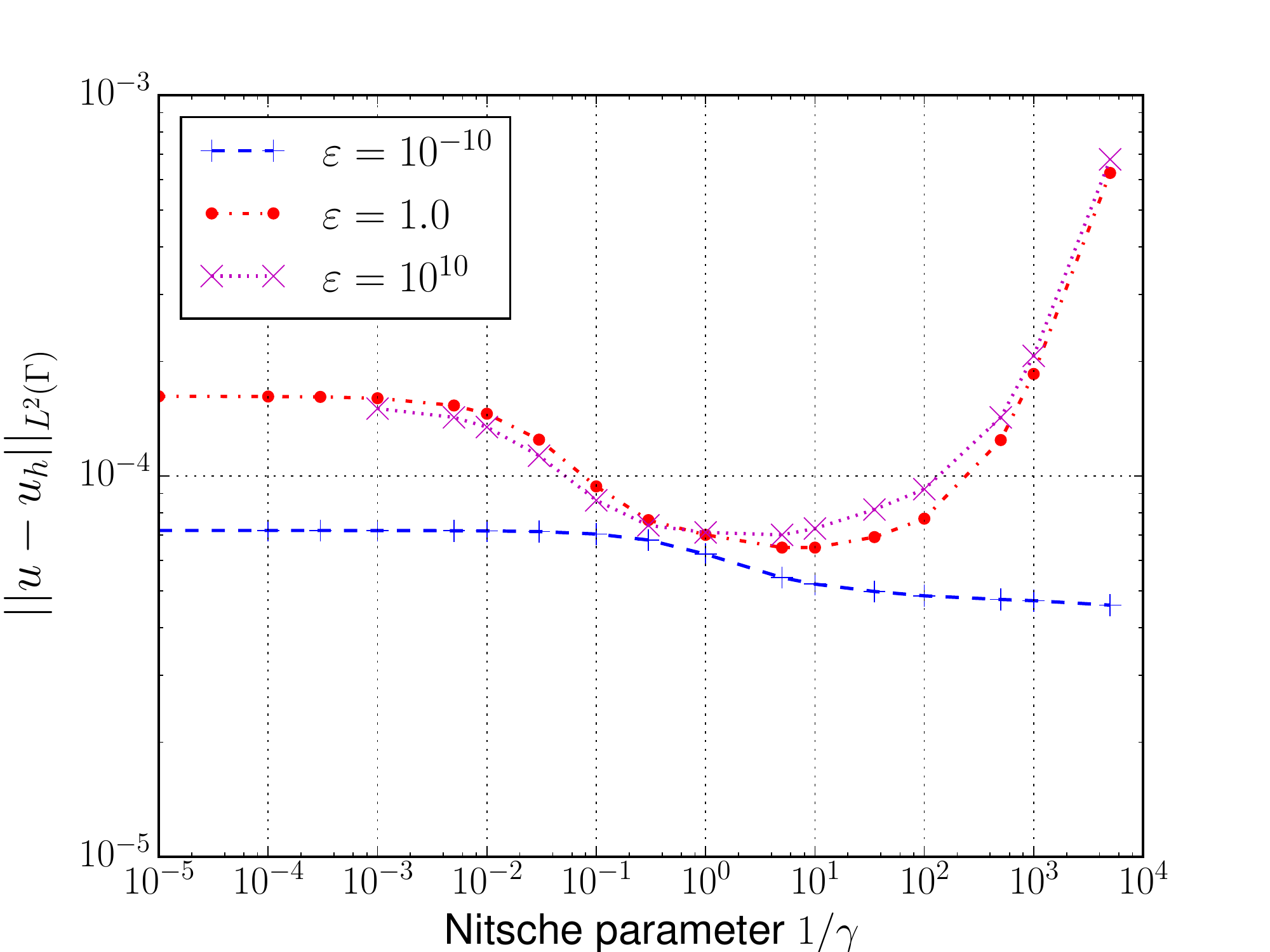}}
  \subfloat{\includegraphics[trim=1 1 44 32, clip, width=0.33\textwidth]{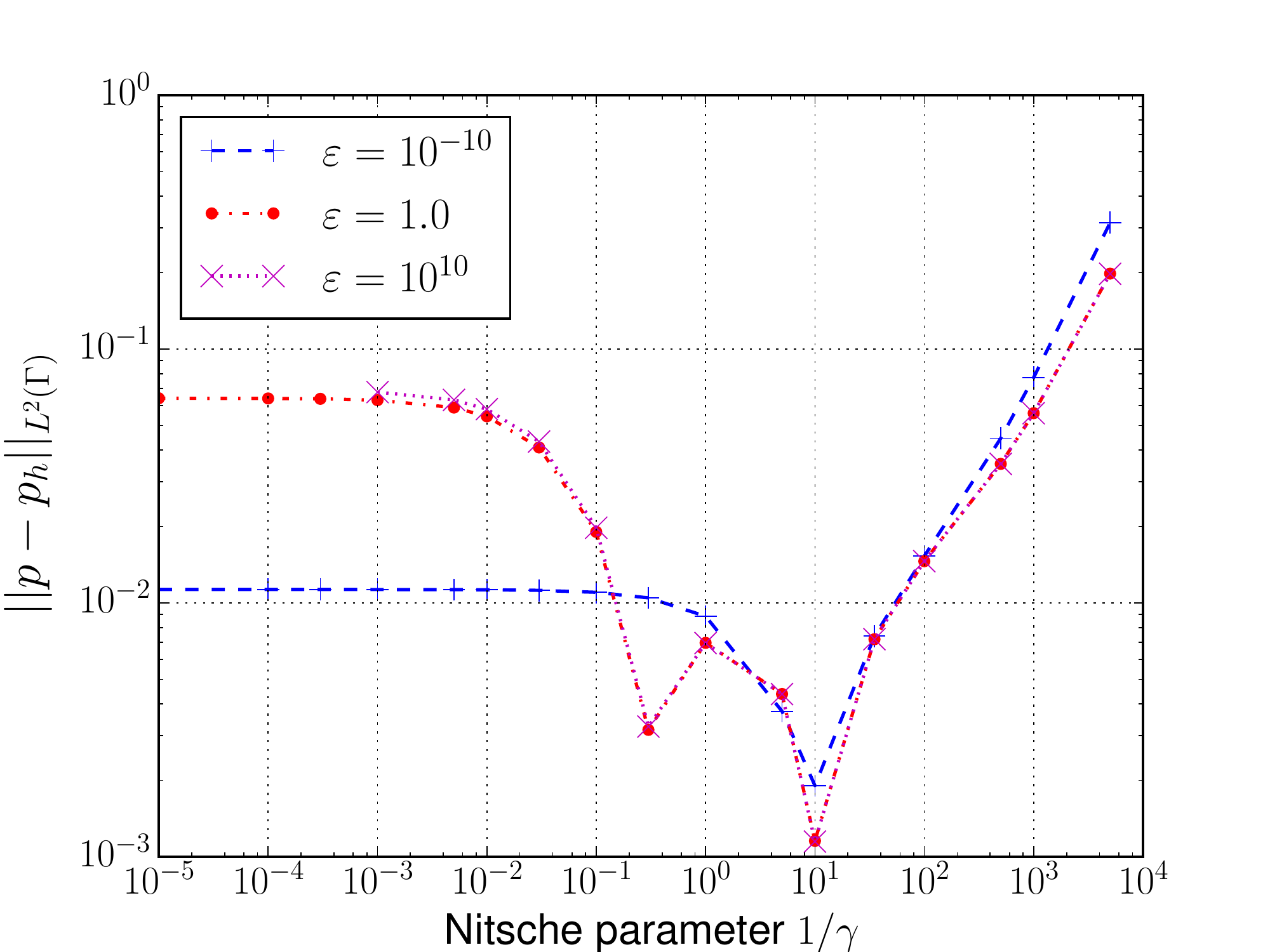}}
  \subfloat{\includegraphics[trim=1 1 44 32, clip, width=0.33\textwidth]{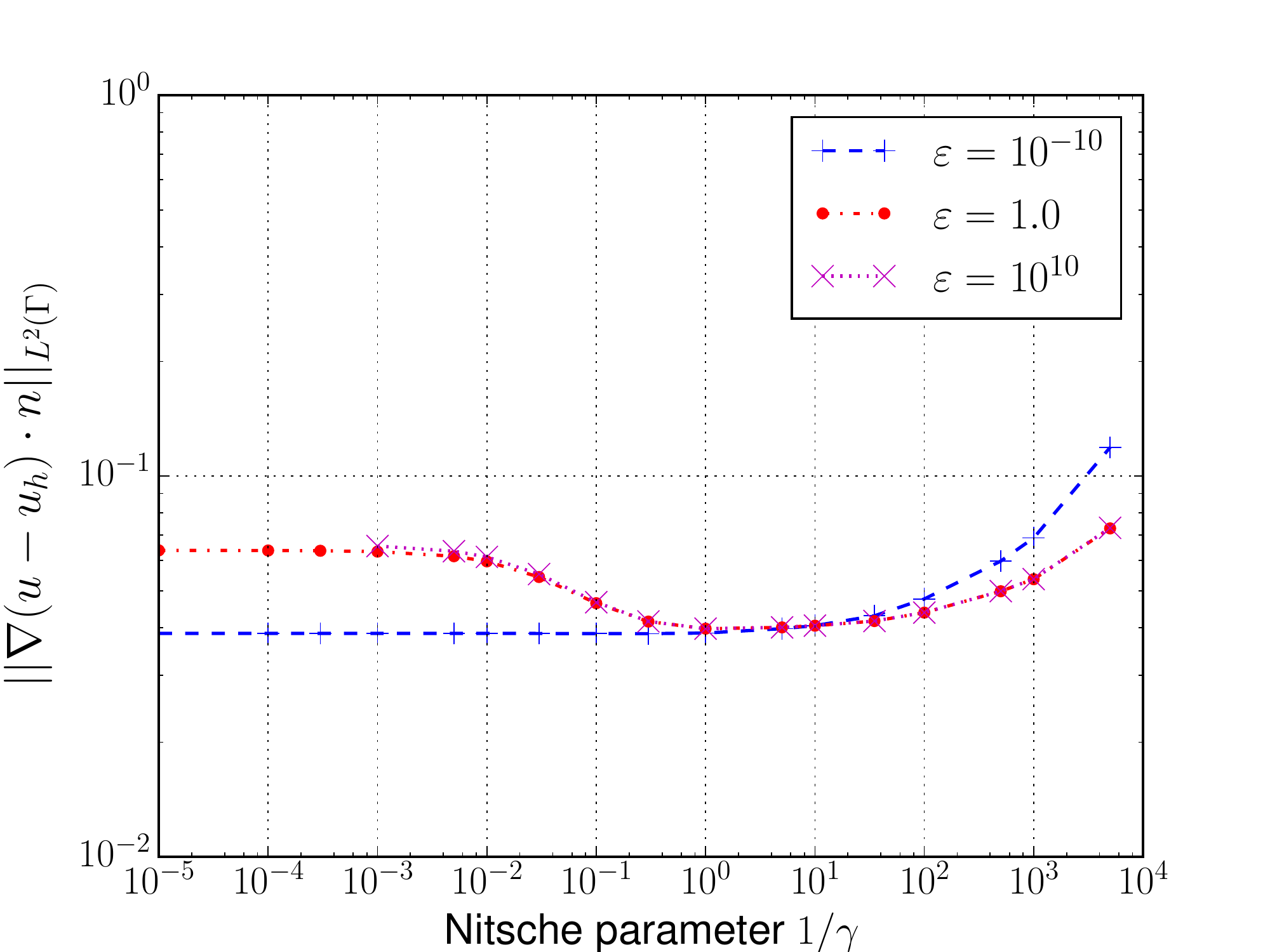}}
  \caption{Nitsche stabilization parameter study for an adjoint-inconsistent Nitsche's method with $\mcQ^1$ elements: bulk errors (top row) and boundary errors (bottom row) for velocity, pressure and velocity gradient (from left to right).}
  \label{fig:nitstab-skew-convergence-study-separate-hex8}
\end{figure}

\begin{figure}[ht!]
  \centering
  \subfloat{\includegraphics[trim=1 1 44 32, clip, width=0.33\textwidth]{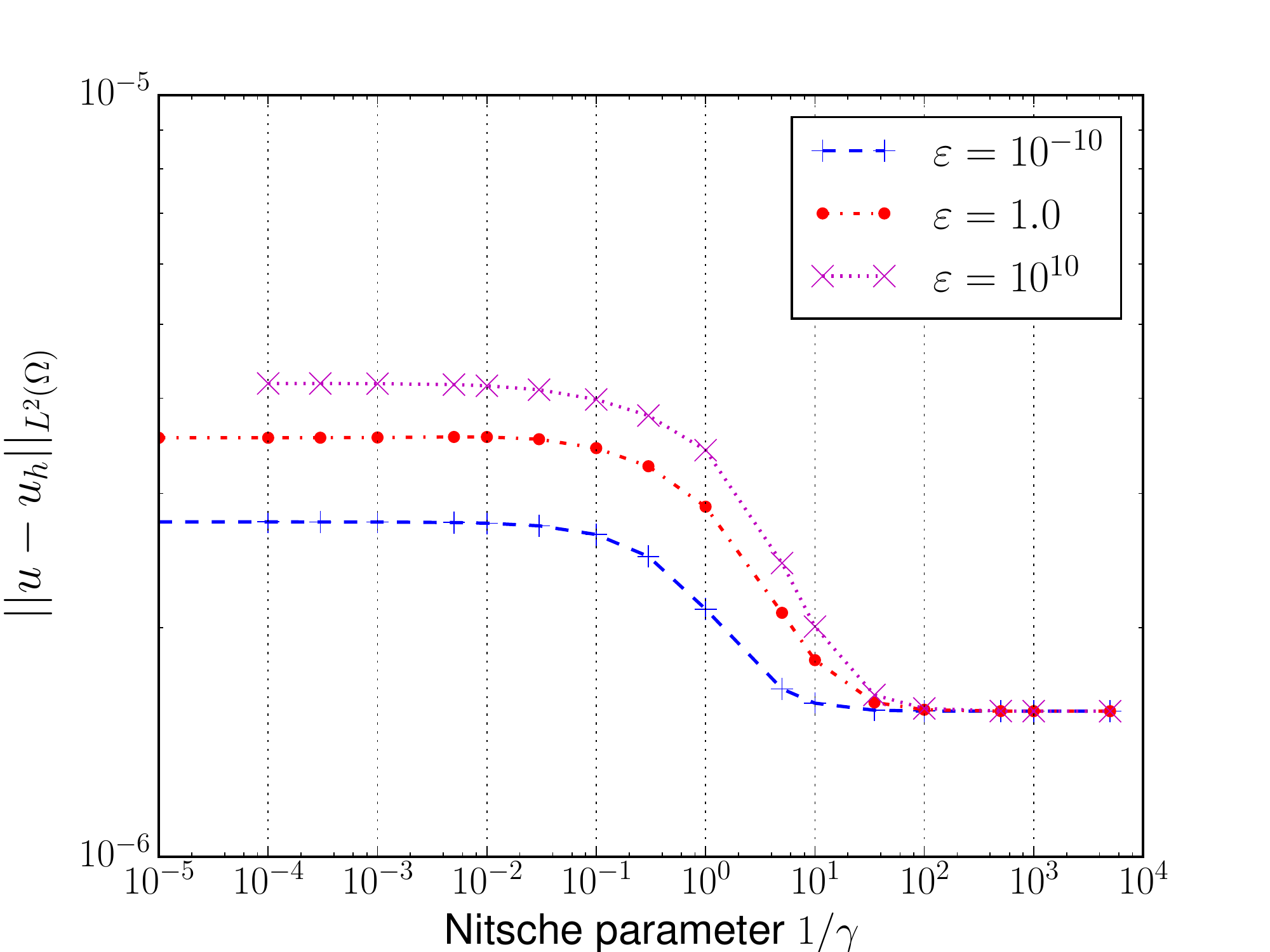}}
  \subfloat{\includegraphics[trim=1 1 44 32, clip, width=0.33\textwidth]{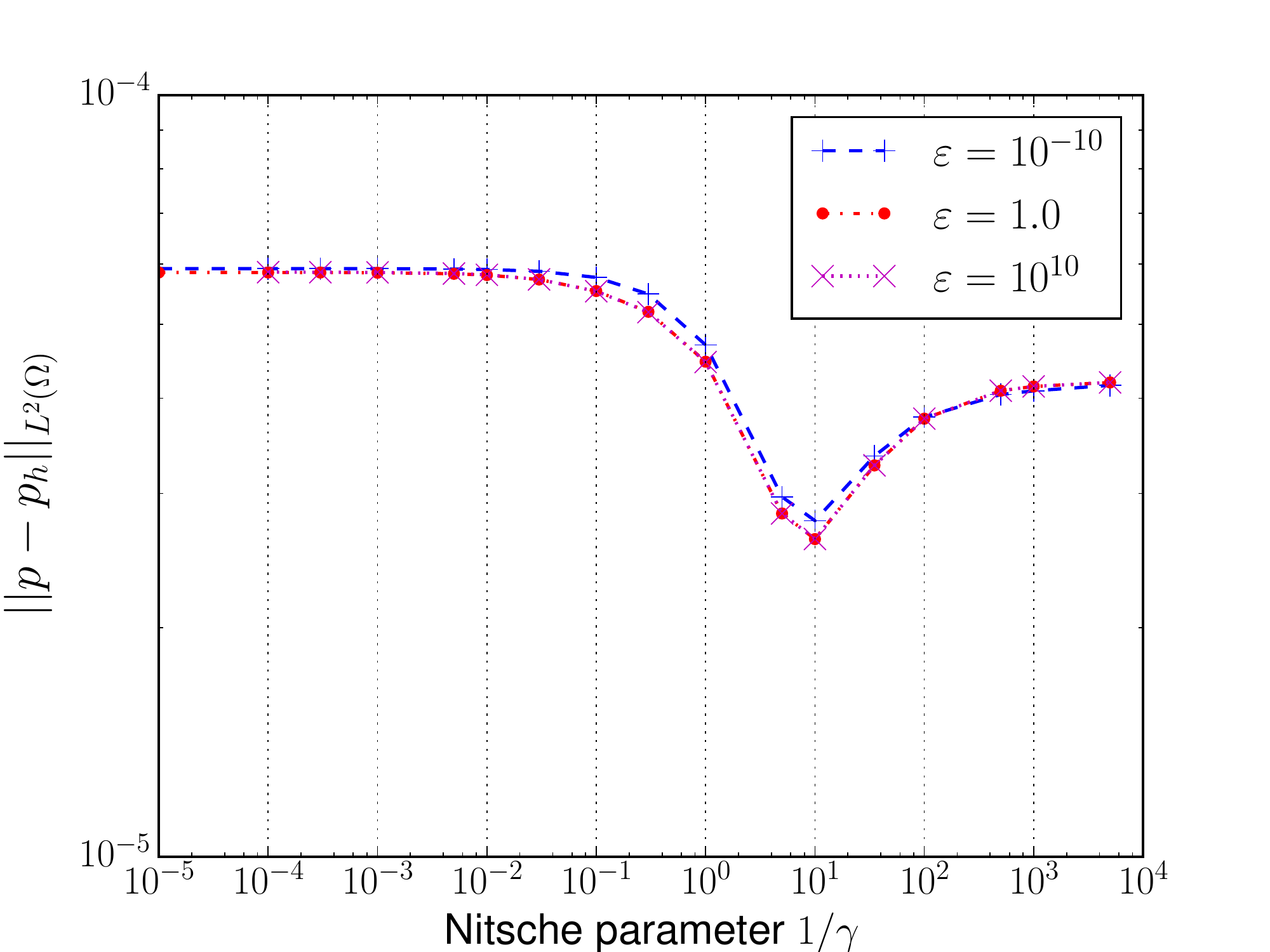}}
  \subfloat{\includegraphics[trim=1 1 44 32, clip, width=0.33\textwidth]{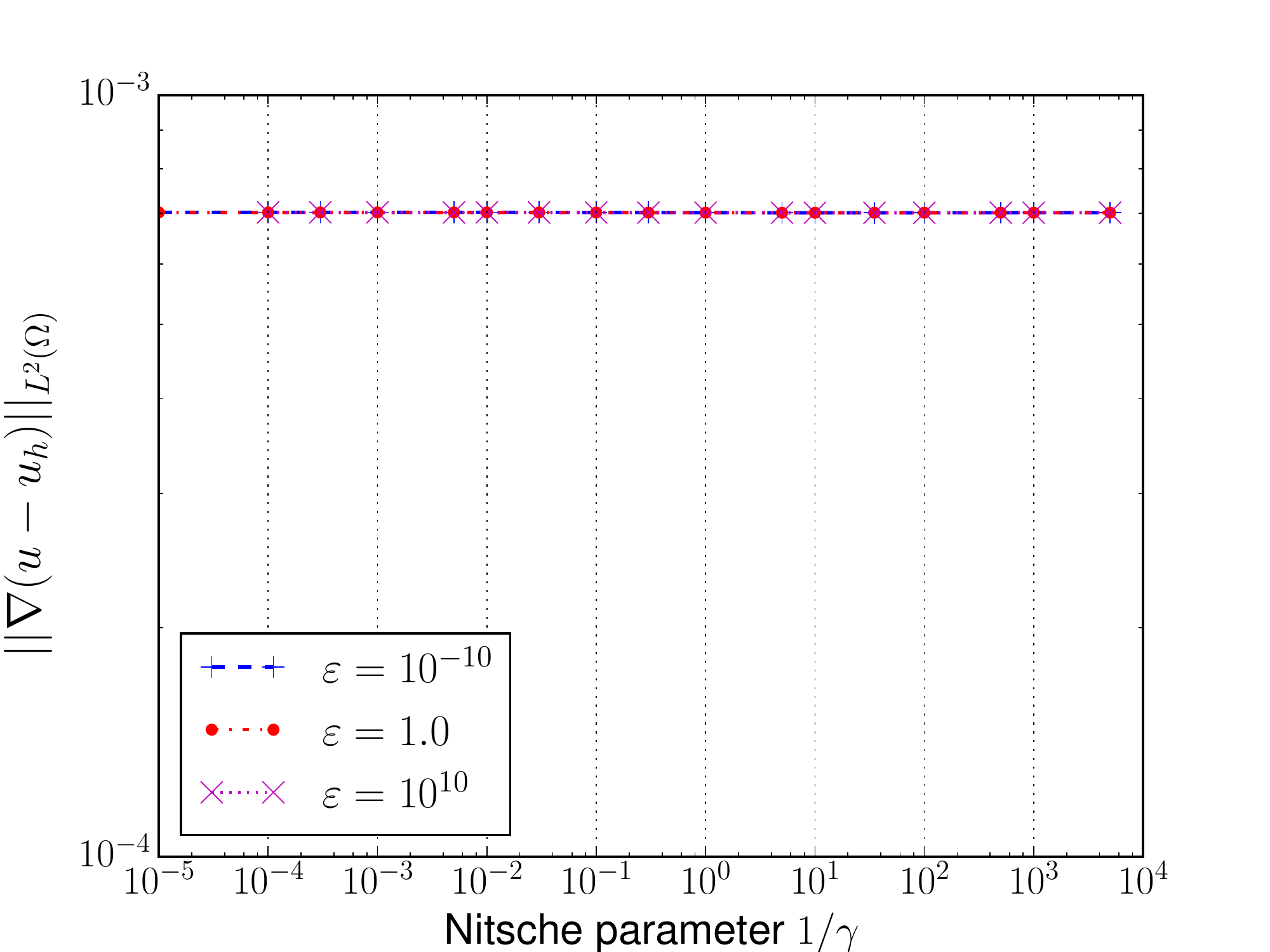}}\\
  \subfloat{\includegraphics[trim=1 1 44 32, clip, width=0.33\textwidth]{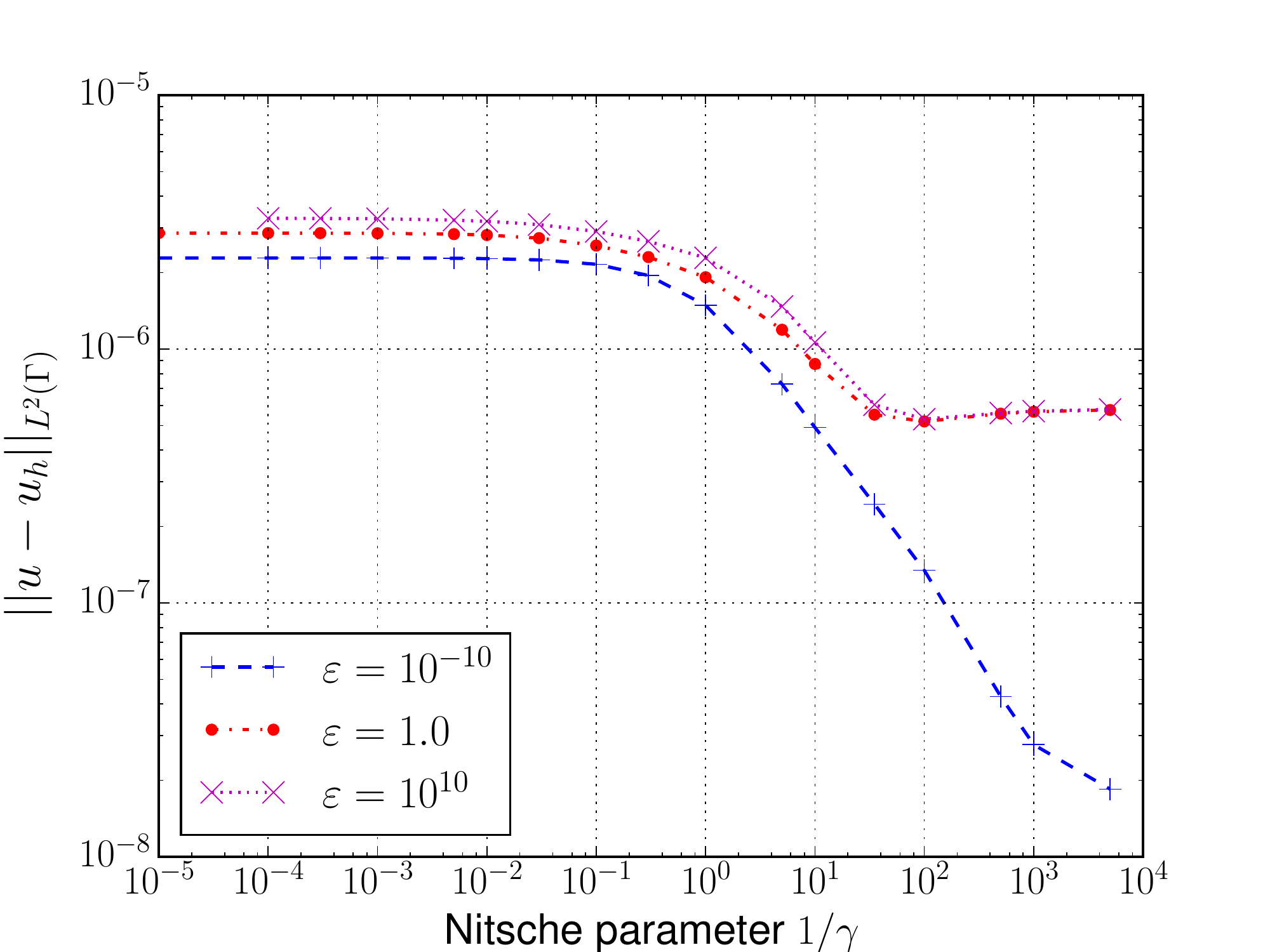}}
  \subfloat{\includegraphics[trim=1 1 44 32, clip, width=0.33\textwidth]{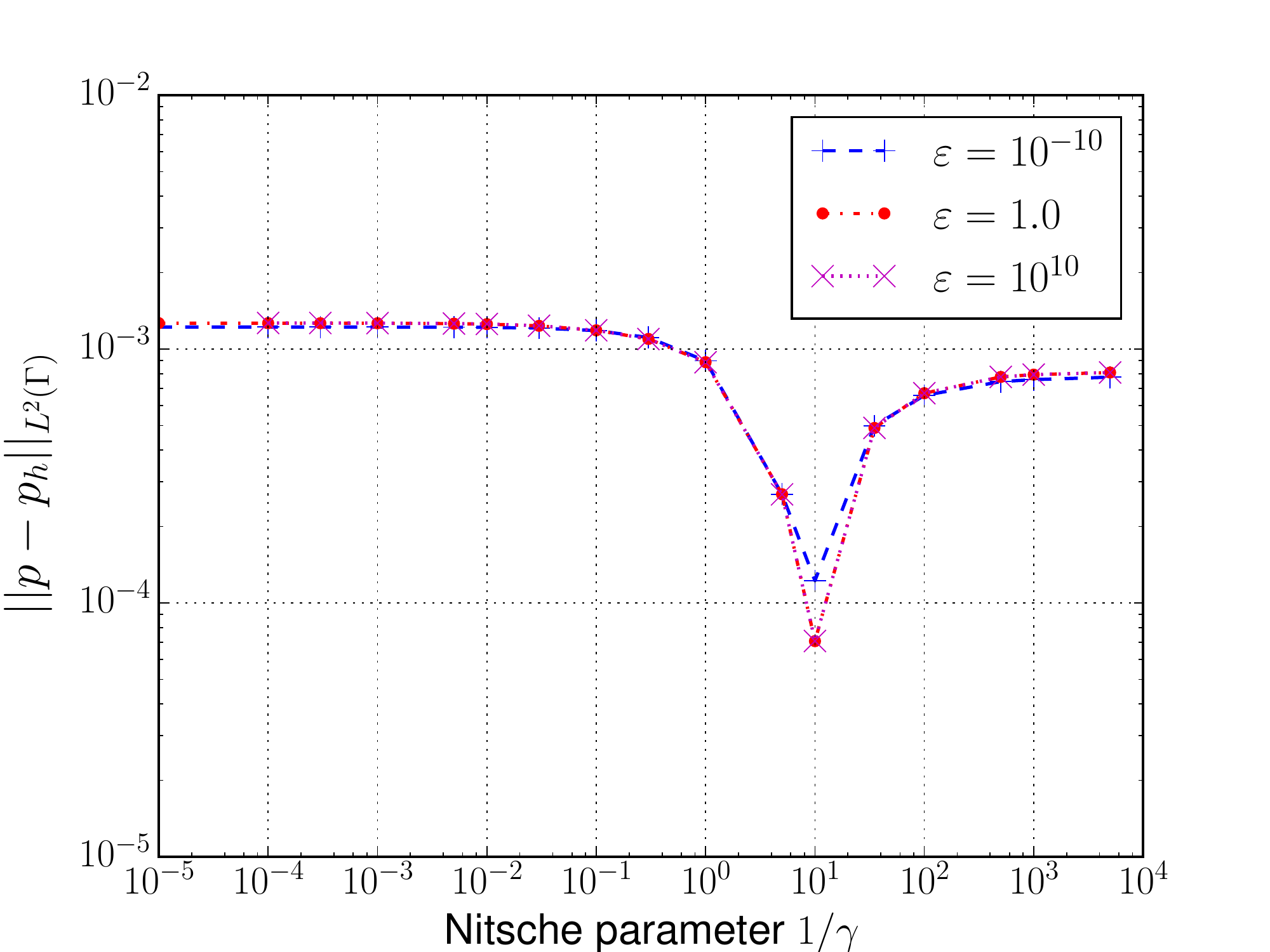}}
  \subfloat{\includegraphics[trim=1 1 44 32, clip, width=0.33\textwidth]{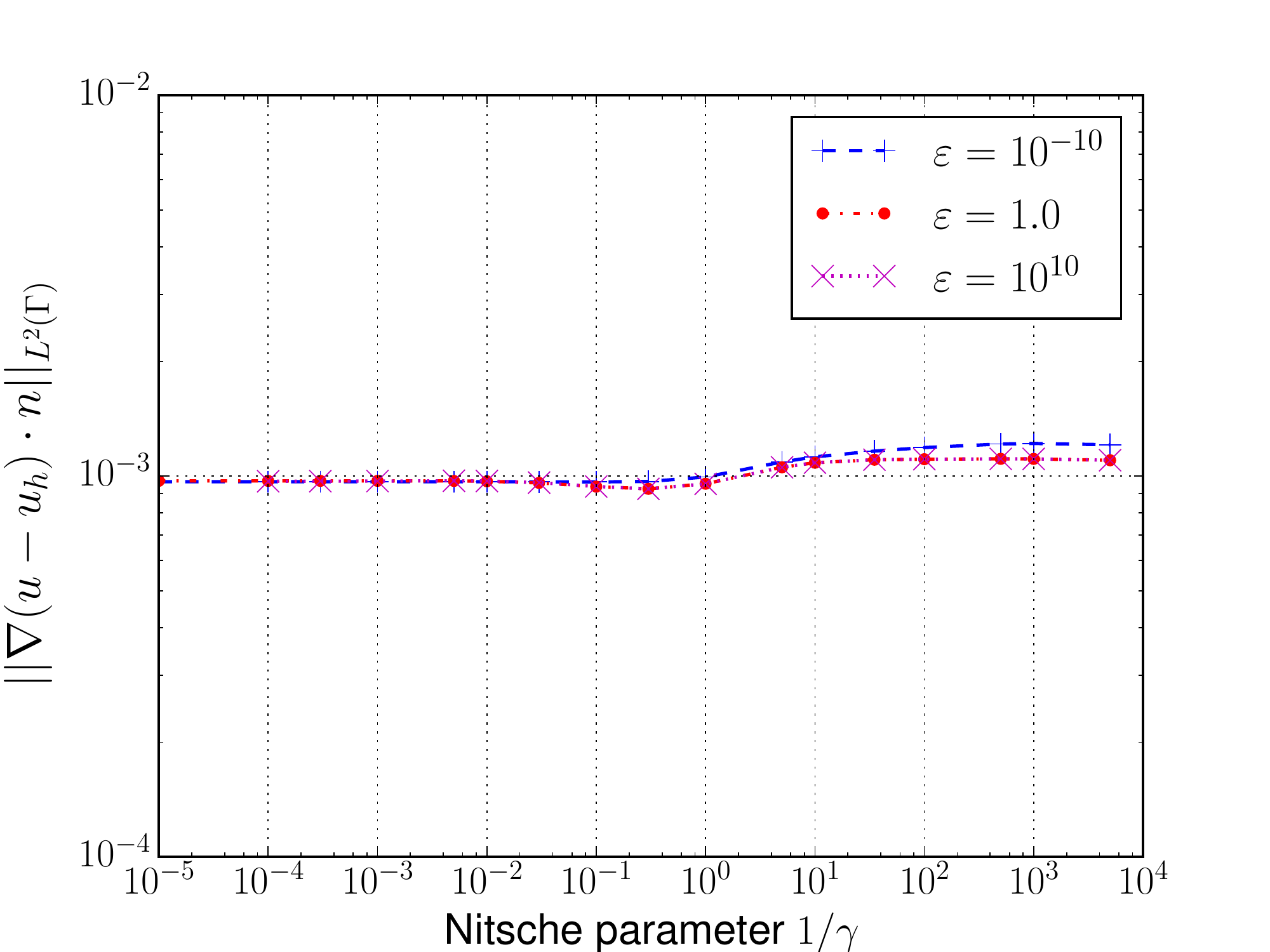}}
  \caption{Nitsche stabilization parameter study for an adjoint-inconsistent Nitsche's method with $\mcQ^2$ elements: bulk errors (top row) and boundary errors (bottom row) for velocity, pressure and velocity gradient (from left to right).}
  \label{fig:nitstab-skew-convergence-study-separate-hex27}
\end{figure}


From the graphs $1/\gamma \approx 10$ gives a clear minimum for the pressure error for both the linear and quadratic approximations. However, in the quadratic case the velocity error is smaller for a choice of $1/\gamma \approx 35$. From these observations a good choice of the stabilization parameter should be around $1/\gamma \in (10, 50)$. This proposed interval agrees well with the observed results in \citet{SchottWall2014}.

\subsection{Slip Length Sensitivity Study}
To demonstrate the robustness of the method to the choice of slip length, errors are evaluated for a series of $\varepsilon \in [10^{-10},10^{10}]$ for linear $\mcQ^1$ elements. A comparison is done with the substitution method \eqref{eq:substitution_method} explained in Section~\ref{ssec:weak_imposition_bcs_oseen} to emphasize the advantages of our method. As can be seen in Figure~\ref{fig:sliplength-convergence-study-seperate-with-urquiza-hex8}, when imposing the general Navier condition by means of our Nitsche's method the error remains almost constant and the difference between the errors of the limiting cases (i.e. $\varepsilon \rightarrow 0$ or $\varepsilon \rightarrow \infty$) are small. Even though the analytic solution is independent of $\varepsilon$ some difference between the limiting cases are expected for the numerical simulations. Furthermore, the advantage of the proposed method to the substitution method is clear from this Figure. 
The substitution method starts to produce noticeably larger errors for $\varepsilon < 10^{-5}$, and for smaller choices of slip length consecutively worse results are observed up until our linear solver could not solve the system for $\varepsilon < 10^{-8}$.
These results are expected for the substitution method \eqref{eq:substitution_method} as the conditioning of the system becomes increasingly bad when the slip length approaches zero, as discussed in Section~\ref{ssec:weak_imposition_bcs_oseen}.
Furthermore, it is worth to be noted that in the limiting case of $\varepsilon \rightarrow \infty$ the errors between the two methods are of comparable size
and the additional consistent terms \eqref{eq:ah-form-def_tangential_1_fullNeumann} and \eqref{eq:lh-form-def_tangential_1_fullNeumann} do not deteriorate the error.

\begin{figure}[ht!]
  \centering
  \subfloat{\includegraphics[trim=1 1 44 32, clip, width=0.33\textwidth]{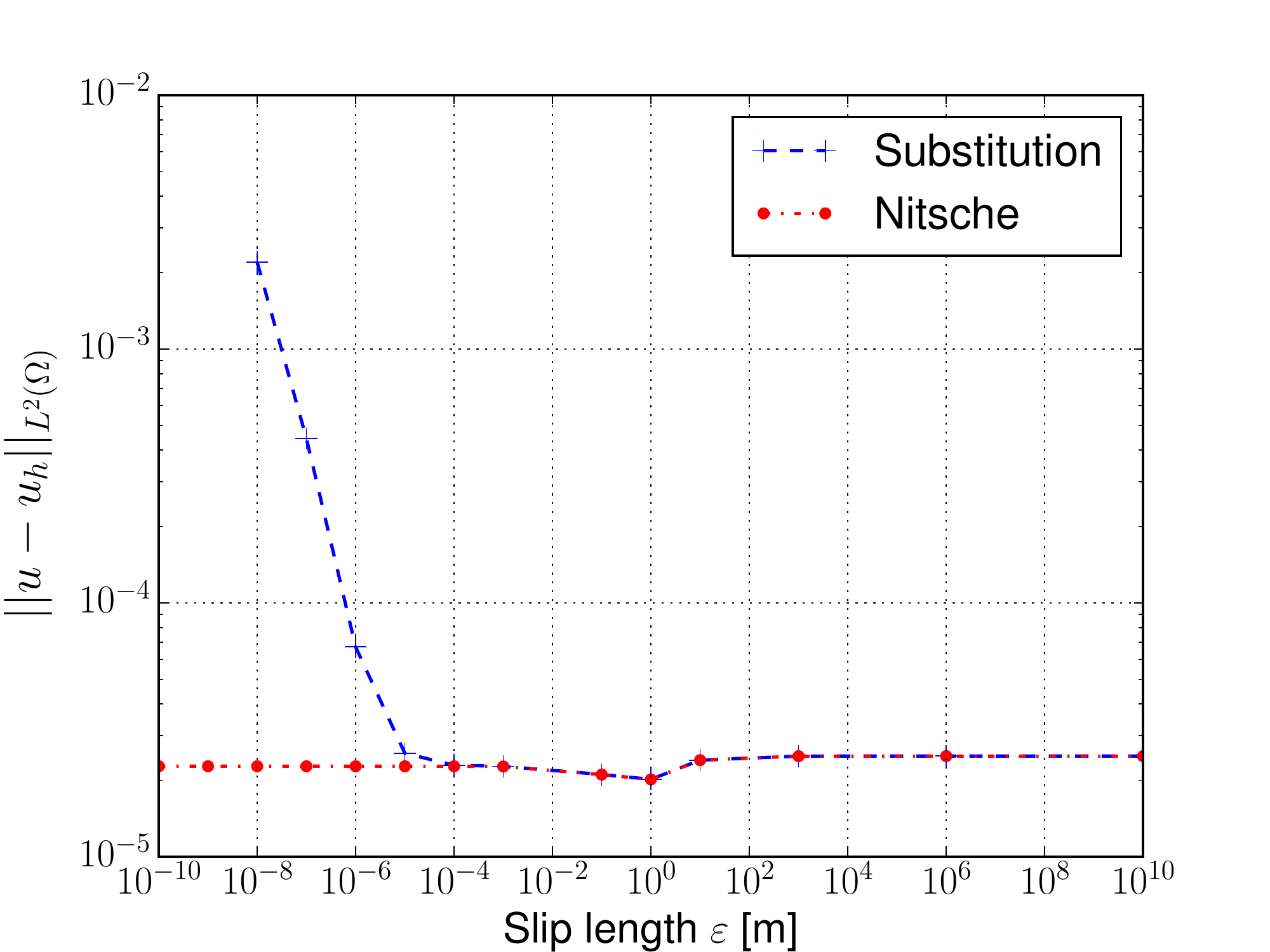}}
  \subfloat{\includegraphics[trim=1 1 44 32, clip, width=0.33\textwidth]{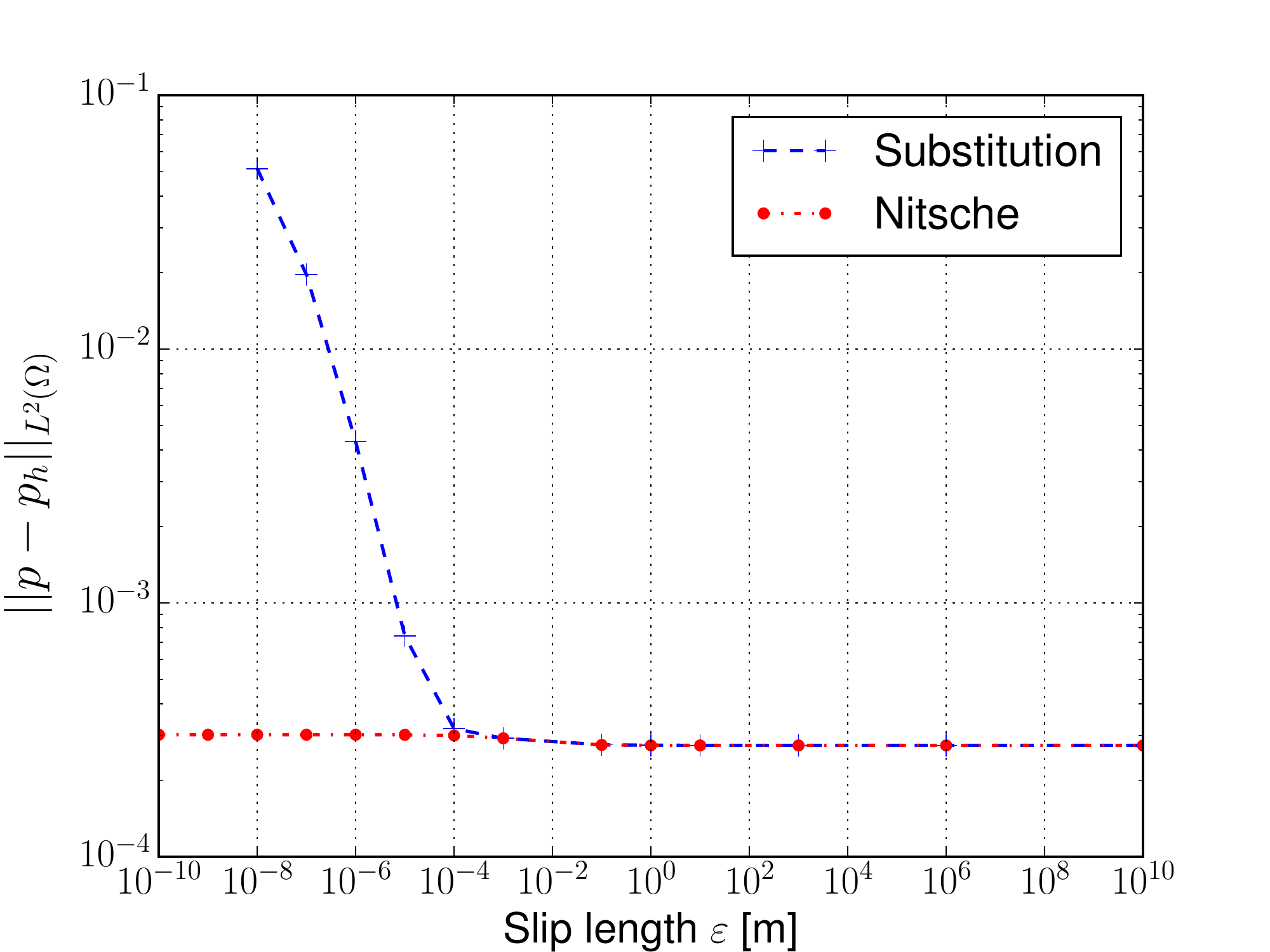}}
  \subfloat{\includegraphics[trim=1 1 44 32, clip, width=0.33\textwidth]{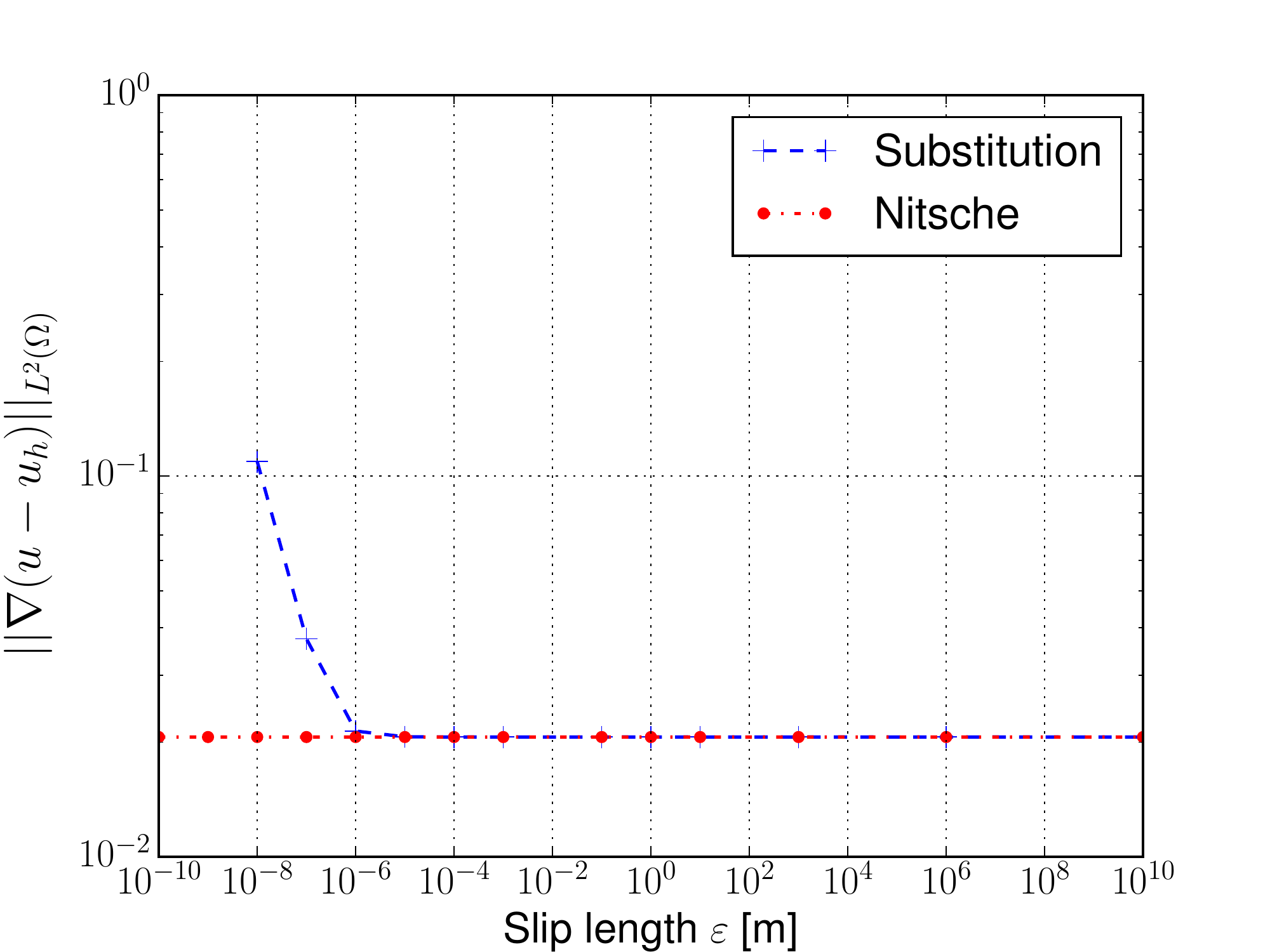}}\\
  \subfloat{\includegraphics[trim=1 1 44 32, clip, width=0.33\textwidth]{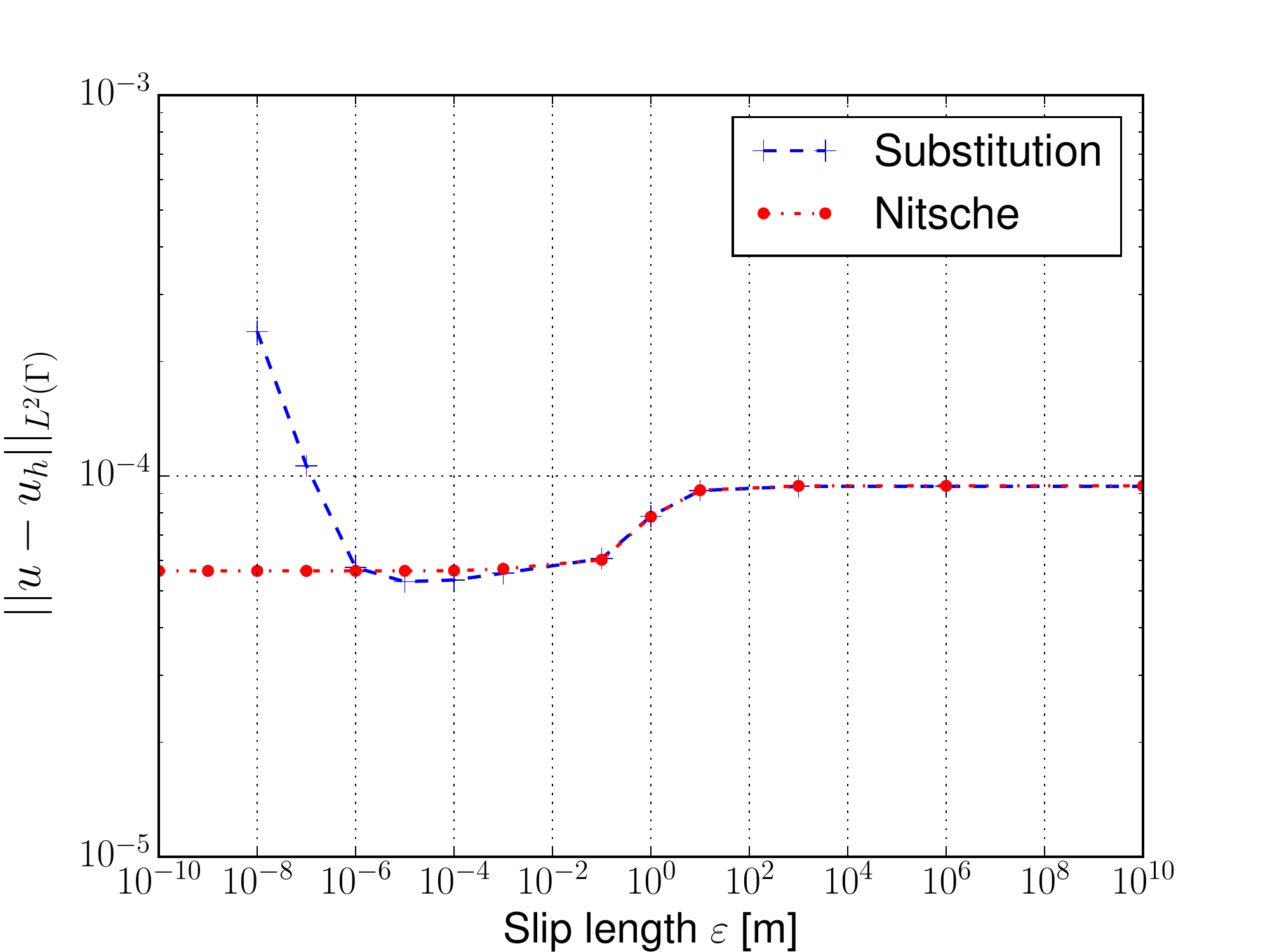}}
  \subfloat{\includegraphics[trim=1 1 44 32, clip, width=0.33\textwidth]{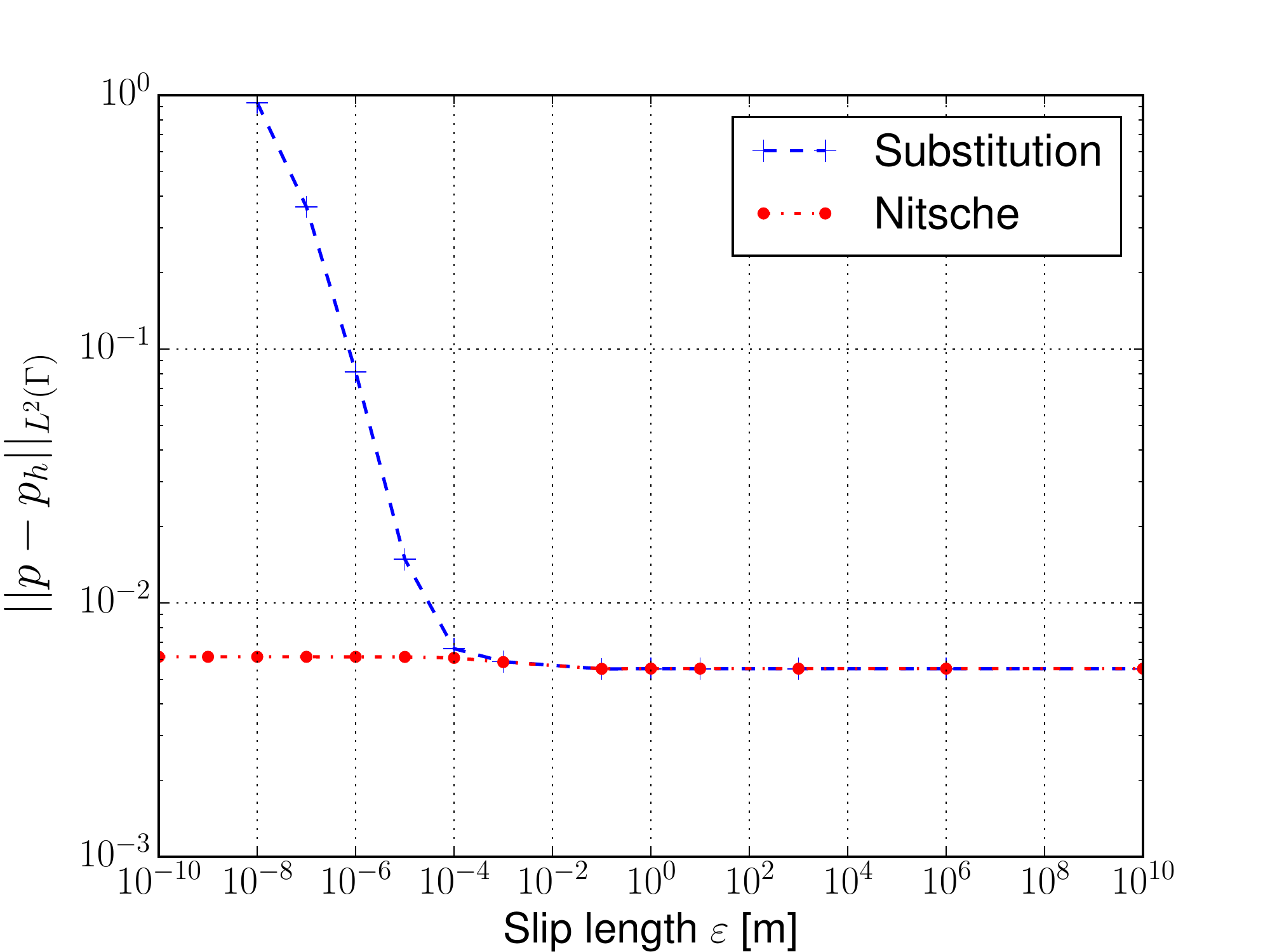}}
  \subfloat{\includegraphics[trim=1 1 44 32, clip, width=0.33\textwidth]{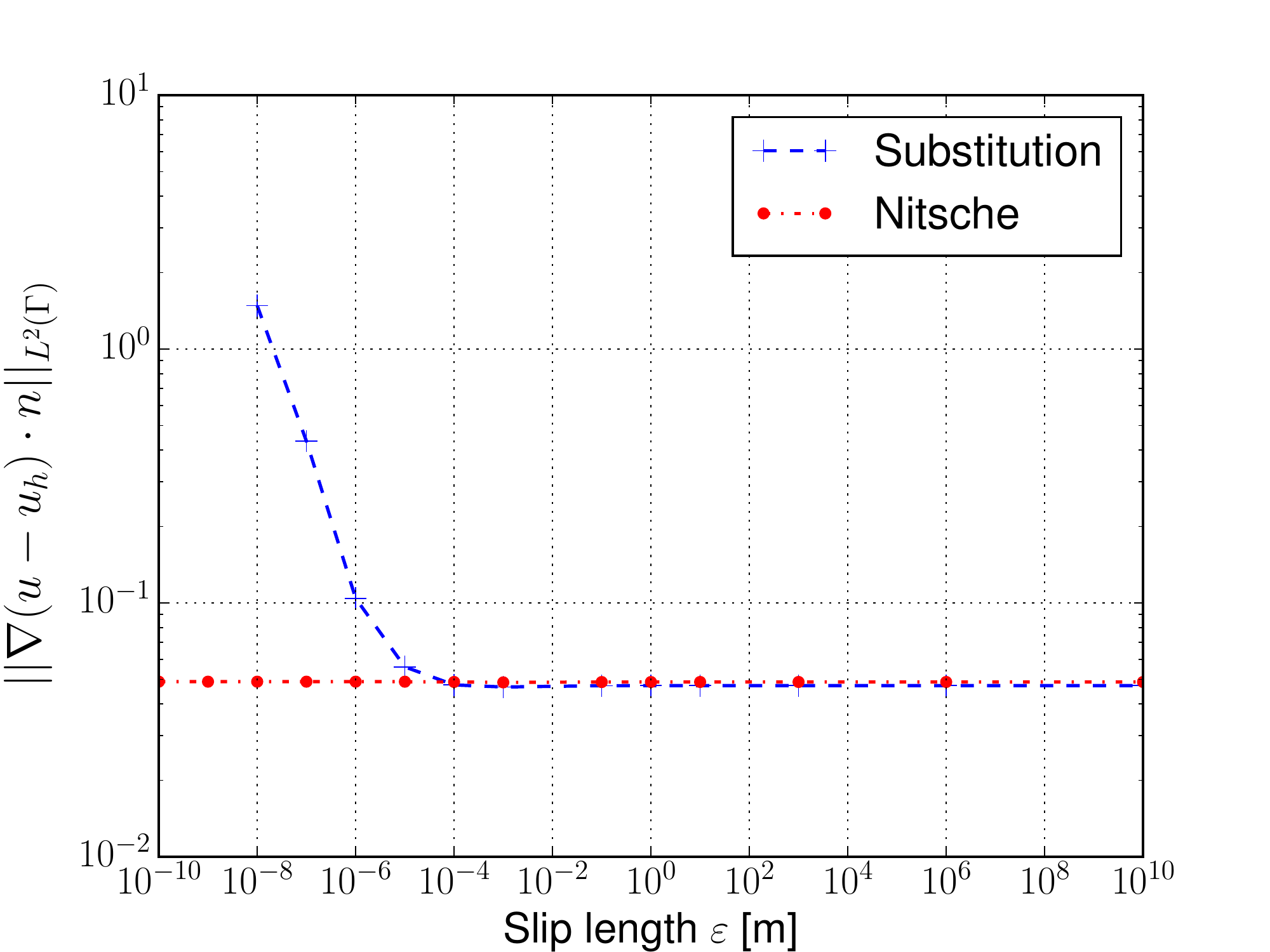}}
  \caption{Slip length parameter study for an adjoint-consistent Nitsche's method with $\mcQ^1$ elements: bulk errors (top row) and boundary errors (bottom row) for velocity, pressure and velocity gradient (from left to right).}
  \label{fig:sliplength-convergence-study-seperate-with-urquiza-hex8}
\end{figure}

\section{Conclusions}
\label{sec:conclusions}
In this work, a novel method for the Oseen problem with a general Navier slip boundary condition is introduced.
This boundary condition is best explained as a Robin condition in the tangential plane and a Dirichlet condition in the normal direction of the boundary.
The proposed method imposes the boundary condition weakly by means of a Nitsche's method for both the tangential and normal parts of the condition.
It remains well-posed and stable for all choices of slip lengths and for both low and high Reynolds numbers.
Furthermore, the presented formulation is a first step for building a more general formulation for the imposition of boundary conditions for the Oseen and Navier-Stokes equations since in the tangential direction it allows for any linear combination of Dirichlet and Neumann conditions to be set within the same framework, which previously was not possible.
This method is presented for the case of unfitted grids and an equal-order interpolated cut finite element method is used for the discretization. In the analysis, it is shown that the proposed formulation remains stable irrespective of where the boundary intersects the background mesh. 
The theoretical findings in this work remain valid also for the simpler boundary-fitted grid case and can thus readily be applied for this case as well.

To show the validity of the proposed formulation, inf-sup stability is shown. Furthermore, an \apriori~error analysis for an energy-type norm and for the $L^2$--error of the velocity are conducted for the adjoint-consistent formulation. A numerical example corroborates the findings, where optimal order of convergence was observed for the error norms of both linear and quadratic approximations.
Also observed in a numerical example is that our proposed method performs better for small slip lengths compared to the classical substitution method, and
for larger values of the slip length it performs just as well. This is in agreement with theory, as our method is well-posed in both the Dirichlet and Neumann limits for the tangential condition, whereas the classical substitution method is not defined in the Dirichlet limit, which coincides with the choice of a small slip length.



As the proposed method enables the possibility to use both no-slip and slip with the same method, it is a promising approach for a number of future applications. These include for instance fluid-structure-contact interaction as pure no-slip walls may lead to a not well-posed problem in the region where submerged bodies are close to each other. Another promising application currently investigated by the authors is the use of a localized Navier-slip condition in the modeling of contact line motion for multi-phase flows.

\section*{Acknowledgement}

This work is supported by the International Graduate School of Science and Engineering (IGSSE)
of the Technical University of Munich, Germany, under project
6.02 and the Erasmus Mundus Joint Doctorate SEED project
(European Commission). The authors would also like to thank Antonio Huerta for the discussions had about the formulation.





\section*{References}

\bibliographystyle{model1-num-names}
\bibliography{bibliography_new_3}



\end{document}